\documentclass[12pt]{amsart}
\usepackage[utf8]{inputenc} 
\usepackage[OT2,T1]{fontenc}
\usepackage[frenchb]{babel}
\usepackage[top=2.5cm,bottom=2.5cm,right=3cm,left=3cm]{geometry}

\usepackage{amsmath,amsthm,amssymb,amsfonts,mathrsfs,calligra,stmaryrd}
\usepackage{fouriernc}
\usepackage[final, babel=true]{microtype}
\usepackage{caption}
\usepackage{subcaption}
\usepackage{eucal}

\usepackage{centernot}
\usepackage{enumitem}

\usepackage{graphicx}
\usepackage{tikz}
\usetikzlibrary{matrix,arrows}
\usepackage[all]{xy}

\usepackage{hyperref}
\usepackage{xcolor}
\usepackage{url}
\definecolor{darkgreen}{rgb}{0,0.45,0}  

\hypersetup{
	colorlinks,
	linktoc=all,
	linkcolor={darkgreen},
	citecolor={blue!50!black},
	urlcolor={blue!80!black}
}

\newcommand\bD{\mathbb{D}}

\newcommand\bR{\mathbb{R}}
\newcommand\bC{\mathbb{C}}
\newcommand\bN{\mathbb{N}}
\newcommand\bZ{\mathbb{Z}}
\newcommand\bQ{\mathbb{Q}}
\newcommand\bF{\mathbb{F}}
\newcommand\bP{\mathbb{P}}

\newcommand\bH{\mathbb{H}}
\newcommand\bT{\mathbb{T}}

\newcommand\Zp{{\mathbb{Z}_{p}}}
\newcommand\Qp{\mathbb{Q}_{p}}

\newcommand\cA{\mathcal{A}}

\newcommand\cC{\mathcal{C}}
\newcommand\cD{\mathcal{D}}
\newcommand\cE{\mathcal{E}}

\newcommand\cH{\mathcal{H}}

\newcommand\cL{\mathcal{L}}

\newcommand\cO{\mathcal{O}}
\newcommand\cP{\mathcal{P}}

\newcommand\cT{\mathcal{T}}
\newcommand\cU{\mathcal{U}}

\newcommand\ga{\mathfrak{a}}

\newcommand\gd{\mathfrak{d}}

\newcommand\gm{\mathfrak{m}}

\newcommand\gp{\mathfrak{p}}

\newcommand\gG{\mathfrak{G}}

\newcommand\gM{\mathfrak{M}}

\newcommand\gQ{\mathfrak{Q}}

\newcommand\gX{\mathfrak{X}}

\newcommand{\ob}[1]{\mkern 1.5mu\overline{\mkern-1.5mu#1\mkern-1.5mu}\mkern 1.5mu}

\newcommand{\Ind}{\mathrm{Ind}}

\newcommand{\res}{\mathrm{res}}

\newcommand{\ord}{\mathrm{ord}}

\newcommand{\an}{\textrm{an}}
\newcommand{\alg}{\textrm{alg}}

\DeclareMathOperator{\Gal}{Gal}

\DeclareMathOperator{\Hom}{Hom}

\DeclareMathOperator{\im}{im}

\DeclareMathOperator{\coker}{coker}
\DeclareMathOperator{\Frac}{Frac}
\DeclareMathOperator{\rec}{rec}

\DeclareMathOperator{\Tr}{Tr}

\DeclareMathOperator{\car}{car}

\DeclareMathOperator{\Cl}{Cl}
\DeclareMathOperator{\rg}{Rang}
\DeclareMathOperator{\corg}{Corang}
\DeclareMathOperator{\pgcd}{pgcd}

\DeclareMathOperator{\Frob}{Frob}

\DeclareMathOperator{\cyc}{cyc}

\DeclareMathOperator{\Sel}{Sel}

\DeclareSymbolFont{cyrletters}{OT2}{wncyr}{m}{n}
\DeclareMathSymbol{\Sha}{\mathalpha}{cyrletters}{"58}

\DeclareMathOperator{\GL}{GL}

\newtheoremstyle{thmstyle}
{\parskip} 
{\topsep} 
{\itshape} 
{} 
{\bfseries} 
{.} 
{.5em} 
{} 

\newtheoremstyle{defstyle}
{\parskip} 
{\topsep} 
{} 
{} 
{\bfseries} 
{.} 
{.5em} 
{} 

\theoremstyle{defstyle}
\newtheorem{definition}[subsubsection]{Définition}

\newtheorem{remarque}[subsubsection]{Remarque}

\newtheorem{notation}[subsubsection]{Notation}

\theoremstyle{thmstyle}
\newtheorem{theorem}[subsubsection]{Théorème}
\newtheorem{proposition}[subsubsection]{Proposition}
\newtheorem{lemme}[subsubsection]{Lemme}
\newtheorem{corollaire}[subsubsection]{Corollaire}

\newtheorem{conjecture}[subsubsection]{Conjecture}
\newtheorem*{conjecture*}{Conjecture}

\newtheorem*{thA}{Théorème A}
\newtheorem*{thB}{Théorème B}
\newtheorem*{thC}{Théorème C}

\author{Alexandre Maksoud}
\title{Théorie d'Iwasawa des motifs d'Artin et des formes modulaires de poids 1}
\begin{document}
	\pagenumbering{roman}
	\maketitle
	
	\begin{abstract}
		Soit $p$ un nombre premier impair. Nous étudions la structure du groupe de Greenberg-Selmer cyclotomique attaché à un motif d'Artin irréductible général sur $\mathbb{Q}$ et muni d'une $p$-stabilisation ordinaire. Sous la conjecture de Schanuel $p$-adique faible, nous montrons que celui-ci est de torsion sur l'algèbre d'Iwasawa. Sous de légères hypothèses supplémentaires sur $p$, nous exprimons le terme constant de sa série caractéristique à l'aide d'un régulateur $p$-adique et mettons en évidence un phénomène de zéros triviaux. Dans un deuxième temps, nous spécialisons notre étude aux motifs d'Artin attachés à des formes modulaires classiques de poids 1, où les résultats précédents s'appliquent inconditionnellement. Nous formulons une Conjecture Principale, dont on montre une divisibilité grâce à un théorème de Kato.
	\end{abstract}
	\section*{Introduction}
	\addtocontents{toc}{\setcounter{tocdepth}{-1}}
Soit $M$ un motif sur $\bQ$ et à coefficients dans un corps de nombres $E$, et soit $p$ un nombre premier impair. Lorsque $M$ a bonne réduction et est ordinaire en $p$, Greenberg, Coates, Perrin-Riou \cite{greenbergdeformation,coatesmotivic,prasterisque} ont conjecturé l'existence d'une fonction $L$ $p$-adique interpolant les valeurs spéciales des fonctions $L$ des déformations cyclotomiques de $M$. Ils formulent en outre une Conjecture Principale d'Iwasawa donnant une interprétation arithmétique de ces objets, comme générateurs de l'idéal caractéristique d'un groupe de Selmer associé à $M$. 
Comme le remarque Greenberg \cite{greenbergartinII}, cette Conjecture Principale se déduit des résultats de Wiles \cite{wiles1990iwasawa} dans le cas particulier d'un motif attaché à une représentation d'Artin de dimension $d\geq 1$ découpant un corps totalement réel. Lorsque $M$ est un motif d'Artin quelconque, $M$ n'a plus de nécessairement de valeur critique et l'on ne peut pas appliquer la construction de fonctions $L$ $p$-adiques à la Deligne-Ribet \cite{deligneribet}.
Nous nous proposons ici de décrire dans un premier temps la structure des groupes de Selmer de motifs d'Artin généraux, puis nous formulons et étudions une Conjecture Principale dans le cas particulier où $M$ provient d'une forme modulaire classique de poids 1. 

\subsection*{Aspects algébriques}

Soit $M$ un motif irréductible et non-trivial sur $\bQ$ et à coefficients dans un corps de nombres $E \subseteq \ob{\bQ}$ provenant d'une représentation d'Artin $\rho : \Gal(\ob{\bQ}/\bQ) \longrightarrow \GL_E(V)$ de dimension $d\geq 1$, et soit $p$ un nombre premier impair. On fixe des plongements $\iota_\ell : \ob{\bQ} \hookrightarrow \ob{\bQ}_\ell$ pour tout nombre premier $\ell$ ainsi que $\iota_\infty : \ob{\bQ} \hookrightarrow \bC$. On note $V_p = V \otimes_{E,\iota_p} L$ la réalisation $p$-adique de $M$, où $L$ est une extension finie de $\Qp$ contenant $\iota_p(E)$. On dira que $M$ admet une $p$-\textit{stabilisation ordinaire}, s'il existe un sous-$E$-espace vectoriel $V^+$ de $V$ de dimension $d^+=\dim_E H^0(\bR,V)$, stable pour l'action galoisienne locale en $p$ de $\Gal(\ob{\bQ}_p/\bQ_p)$ et dont le quotient $V^-:=V/V^+$ est non-ramifié en $p$. On posera alors $V_p^\pm = V^\pm \otimes_{E,\iota_p} L$ et on notera $\rho^+$ la $\Gal(\ob{\bQ}_p/\Qp)$-représentation associée à $V_p^+$. Le choix d'une $p$-stabilisation ordinaire de $M$ ainsi que d'un $\cO_L$-réseau $\Gal(\ob{\bQ}/\bQ)$-stable $T_p$ de $V_p$ permet de définir le groupe de Greenberg-Selmer cyclotomique $X_\infty(\rho,\rho^+)$ comme étant le dual de Pontryagin du noyau de l'application de restriction global-local en cohomologie galoisienne
$$ H^1(\bQ_\infty,T_p \otimes \Qp/\Zp) \longrightarrow H^1(I_{p},T^-_p \otimes \Qp/\Zp) \times \prod_{\ell \neq p} H^1(I_{\ell},T_p \otimes \Qp/\Zp), $$
où $\bQ_\infty=\bigcup_n \bQ_n$ désigne la $\Zp$-extension cyclotomique de $\bQ$, où les $I_\ell$ sont les sous-groupes d'inertie en $\ell$ (voir Définition \ref{def:sel_rho} et Remarque \ref{rq:sel}) et où $T_p^-$ est l'image de $T_p$ dans $V_p^-$. Lorsque $d^+=d$, on a $V^+=V$ et on retrouve le groupe de Selmer attaché à une représentation d'Artin totalement paire (cf. \cite{greenbergartinII}). $X_\infty(\rho,\rho^+)$ est un module de type fini sur l'algèbre d'Iwasawa $\cO_L[[\Gal(\bQ_\infty/\bQ)]]$, que l'on identifie désormais à $\Lambda:=\cO_L[[T]]$ en envoyant un générateur topologique $\gamma\in\Gal(\bQ_\infty/\bQ)$ sur $1+T$. Si $X_\infty(\rho,\rho^+)$ est de torsion, on note $L_p^{\alg}(\rho,\rho^+;T) \in \Lambda-\{0\}$ sa fonction $L$ $p$-adique algébrique, c'est-à-dire un générateur de son idéal caractéristique (cf. Définition \ref{def:L_p(X)}).

Le théorème principal de cette première partie décrit la structure du $\Lambda$-module $X_\infty(\rho,\rho^+)$. On s'attend à ce qu'il soit toujours de torsion (voir par exemple \cite[Conjecture 3.3]{greenbergdeformation}), ce qui est nécessaire pour formuler une Conjecture Principale d'Iwasawa. Soit $H\subseteq \ob{\bQ}$ le corps de nombres découpé par $\rho$, et soit $\ob{\rho}$ la représentation résiduelle de $\rho$ déterminée par le choix de $T_p$.

\begin{thA}[= Théorème \ref{th:sel_n_fini_etc}]
	Supposons $d^+=1$, ou bien supposons que la conjecture de Schanuel $p$-adique faible est vraie (cf. Conjecture \ref{conj:Schanuel}).
	\begin{enumerate}
		\item[(i)] Le $\Lambda$-module $X_\infty(\rho,\rho^+)$ est de torsion.
		\item[(ii)] Si $H\cap \bQ_\infty =\bQ$, alors $L_p^{\alg}(\rho,\rho^+;T)$ ne s'annule pas en les $\zeta-1$, où $\zeta$ parcourt $\mu_{p^\infty}-\{1\}$. En posant $e(\rho,\rho^+):=\dim_L H^0(\Qp,V_p^-)$, on a de plus :
		$$L_p^{\alg}(\rho,\rho^+;0)\neq 0 \Longleftrightarrow e(\rho,\rho^+)=0,$$  
		$$\ord_{T} L_p^{\alg}(\rho,\rho^+;T) \geq e(\rho,\rho^+).$$
		
		\item[(iii)] Si $H^0(\bQ,\ob{\rho})=H^0(\bQ,\Hom(\ob{\rho},\mu_p))=0$, alors $X_\infty(\rho,\rho^+)$ n'a pas de sous-$\Lambda$-module fini non-nul.
	\end{enumerate}
	
\end{thA}

La preuve du Théorème A est étalée sur les Sections \ref{sec:isom_de_restriction} à \ref{sec:zerostriviaux}. On peut définir à niveau fini $n$ un groupe de Selmer $X_n(\rho,\rho^+)$,  que l'on décrit à l'aide de la suite d'inflation-restriction, pour nous ramener à l'étude du groupe de Galois de la plus grande pro-$p$ extension abélienne de $H_n=H\cdot \bQ_n$ non-ramifiée en dehors de $p$. Sa structure galoisienne dépend, par la théorie des corps de classes locale, de la manière dont le $p$-complété des unités globales de $H_n$ s'envoie dans le produit de ses unités locales. Le point (i) repose sur la preuve que $X_n(\rho,\rho^+)$ est fini quelque soit $n$, et cela nécessite en général de supposer la non-annulation d'un certain régulateur $p$-adique de taille $d^+$. Lorsque $d^+=1$, on peut se passer de ces hypothèses en faisant appel à la version $p$-adique, due à Brumer, du théorème  de transcendance des logarithmes de nombres algébriques de Baker (cf. Théorème \ref{th:BB}). On étudie ensuite le noyau et conoyau des applications de contrôle 
$$X_\infty( \rho,\rho^+)_{\Gal(\bQ_\infty/\bQ_n)} \longrightarrow X_n(\rho,\rho^+),$$
pour en déduire les points (i) et (ii). La preuve du point (iii) utilise le point (i) et les résultats généraux de Greenberg sur la structure des groupes de Selmer, que l'on reformule dans la Section \ref{sec:structure_selmer} (cf. Proposition \ref{PsN}).

La deuxième partie du Théorème A met en évidence un phénomène de "zéros triviaux". On peut voir la conjecture suivante comme l'analogue algébrique d'un phénomène conjecturé pour les fonction $L$ $p$-adiques analytiques (voir entre autres \cite{gross1981padic,benois}). 

\begin{conjecture*}[=Conjecture \ref{conj:zeros_triviaux}]
	Si $H\cap \bQ_\infty =\bQ$, alors on a 
	$$\ord_{T} L_p^{\alg}(\rho,\rho^+;T) = e(\rho,\rho^+).$$
\end{conjecture*}

Notons que les calculs menés se simplifient sous l'hypothèse que $d^+=1$, que $\rho$ est non-ramifiée en $p$ et que $p$ ne divise pas l'ordre de l'image de $\rho$. Des résultats similaires ont étés obtenus indépendamment par Greenberg et Vatsal \cite{preprint} sous ces hypothèses supplémentaires dans une récente prépublication. 

En général, changer de réseau $T_p$ multiplie la fonction $L$ $p$-adique algébrique par une puissance de $p$ (Proposition \ref{prop:dépendance_réseau}). Lorsque $\rho$ est résiduellement irréductible, $L_p^{\alg}(\rho,\rho^+;T)$ ne dépend que du choix de la $p$-stabilisation et est définie à multiplication par une unité de $\Lambda$ près. En l'absence de zéro trivial, on peut déterminer le terme constant de $L_p^{\alg}(\rho,\rho^+;T)$ à une unité $p$-adique près, sous les hypothèses suivantes (excluant un nombre fini de nombres premiers $p$) : 
\begin{itemize}
	\item[\textbf{(nr)}] $\rho$ est non-ramifiée en $p$,
	\item[\textbf{(deg)}] $p$ ne divise par l'ordre de l'image de $\rho$, \textit{i.e.}, $p\nmid \#G$, où $G:=\Gal(H/\bQ)$,
	\item[\textbf{(dim)}] $p$ ne divise pas la dimension de $V$, \textit{i.e.}, $p\nmid d$,
\end{itemize}
En particulier, $\rho$ est résiduellement irréductible, et on peut choisir l'extension $L$ non-ramifiée sur $\Qp$ et contenant $H\simeq \iota_p(H)\subseteq \ob{\bQ}_p$. L'extension linéaire  $\log_p : \cO_H^\times \otimes \cO_L \longrightarrow L$ du logarithme $p$-adique sur le groupe des unités du corps de nombres $H$ est à valeurs dans $p\cO_L$. Par ailleurs, le $\cO_L$-module $\Hom_G(T_p,\cO_H^\times \otimes \cO_L)$ est libre de rang $d^+$ en choisissant $L\subseteq \bQ_p^{\textrm{nr}}$ assez grand. Fixons-en une base $(\Phi_1,\ldots,\Phi_{d^+})$, et notons $\vec{t}_1,\ldots,\vec{t}_{d^+}$ une base de $T^+_p$. Ces choix définissent un "régulateur $p$-adique"

$$R_p(\rho,\rho^+)=\det\left(\frac{\log_p\left(\Phi_i(\vec{t}_j)\right)}{p}\right)_{1\leq i,j\leq d^+} \in \cO_L$$
de taille $d^+\times d^+$ qui est bien défini à une unité de $\cO_L$-près et qui est conjecturalement non-nul (cf. Lemme \ref{lem:non_annulation_reg}). La Section \ref{sec:preuve_terme_cst} est dédiée à la preuve du théorème suivant, qui découle des résultats de la Section \ref{section4} et de la description explicite de $X_0(\rho,\rho^+)$.

\begin{thB}[= Théorème \ref{th:terme_cst_dim_d}]
	Supposons $d^+=1$, ou bien supposons que la conjecture de Schanuel $p$-adique faible est vraie. On a $R_p(\rho,\rho^+)\neq 0$, et si $\rho$ satisfait \textbf{(nr)}, \textbf{(deg)} et \textbf{(dim)} et si $e(\rho,\rho^+)=0$, alors le terme constant de sa fonction $L$ algébrique $X_\infty(\rho,\rho^+)$ est donné par 
	$$L_p^{\alg}(\rho,\rho^+;0) = R_p(\rho,\rho^+) \cdot \sqrt[d]{\#\left(\Cl_p(H)^\rho\right)},$$
	à une unité de $\cO_L$ près, où $\Cl_p(H)^\rho$ désigne la composante $\rho$-isotypique du $p$-groupe des classes de $H$.
\end{thB}

L'élément $R_p(\rho,\rho^+)$ vit par définition dans l'ensemble $\cO_L \mod \cO_L^\times$, c'est-à-dire que seule sa valuation $p$-adique est définie. On peut définir de manière analogue un élément $R(\rho,\rho^+) \in L \mod E^\times$ en considérant des $E$-bases respectives de $\Hom_G(V,\cO_H^\times \otimes_\bZ E)$ et de $V^+$. Cette quantité apparaît dans le contexte des intégrales itérées $p$-adiques de \cite{DLR2} attachées à un triplet de formes modulaires classiques $(f,g,h)$, où $f$ est une série d'Eisenstein de poids 2 et où $g$ et $h$ sont deux formes cuspidales de poids 1, au moins lorsque $g$ est $p$-régulière et lorsque la représentation $\rho := \rho_{gh}$ est irréductible et n'a pas de $\Gal(\ob{\bQ}_p/\Qp)$-invariants (cf. notations de \textit{loc. cit.}). Cela suggère un lien entre $X_\infty(\rho_{gh},\rho_{gh}^{g_\alpha})$ et les "variantes de la Conjecture de Gross-Stark" étudiées dans \textit{loc. cit.} 

\subsection*{Aspects analytiques et Conjecture Principale}

Nous nous concentrons plus spécifiquement à l'étude des motifs d'Artin associés par Deligne et Serre à une forme cuspidale primitive $f$ de poids 1 et de niveau $\Gamma_1(N)$. La réalisation d'Artin associée, notée encore $\rho$, est absolument irréductible de dimension $d=2$ et est impaire, \textit{i.e.}, $d^+=1$. On suppose encore les hypothèses \textbf{(nr)}, \textbf{(deg)} et \textbf{(dim)} vérifiées, et en particulier on a $p\nmid N$. On note $\alpha$ et $\beta$ les valeurs propres de $\rho(\Frob_p)$, \textit{i.e.}, les racines du $p$-ième polynôme de Hecke de $f$, et on fait l'hypothèse de régularité suivante : 

$$ \textbf{(rég) : } \alpha \neq \beta. $$ 
En particulier, le choix d'une $p$-stabilisation ordinaire du motif associé à $f$ revient au choix de $\alpha$ ou de $\beta$, c'est-à-dire au choix d'une des deux $p$-stabilisations (distinctes) $f_\alpha(q):=f(q)-\beta f(q^p)$ ou $f_\beta(q):=f(q)-\alpha f(q^p)$ de $f$. On notera $X_\infty(f_\alpha)$ le groupe de Selmer associé au choix de $f_\alpha$, auquel s'appliquent inconditionnellement les résultats de la partie précédente. Lorsque $f$ a multiplication complexe par un corps quadratique imaginaire $F$ dans lequel $p$ est décomposé, on retrouve le groupe de Selmer usuel du caractère de Hecke associé (cf. Section \ref{sec:CM}). Le module obtenu lorsque $p$ est inerte mériterait d'être comparé à la définition du groupe de Selmer $\Sel^\pm$ de Kobayashi pour une courbe elliptique à multiplication complexe et à réduction supersingulière en $p$ (\cite{kobayashi,rubinpollack}).

Contrairement aux cas des formes modulaires classiques $p$-ordinaires de poids $k\geq 2$, il n'existe pas à notre connaissance de construction directe de fonction $L$ $p$-adique analytique pour les formes modulaires de poids 1, notamment car leur représentation automorphe associée n'est pas cohomologique. La forme modulaire $p$-stabilisée $f_\alpha$ varie néanmoins convenablement en famille, et il est possible de définir de manière indirecte une fonction $L$ $p$-adique analytique pour $f_\alpha$. Une première définition d'un élément  $\cL^\textrm{BD}_p(f_\alpha) \in \cO[[\Gamma]] \otimes_\Zp \Qp$, sous la seule hypothèse \textbf{(rég)}, est proposée par Bellaïche et Dimitrov (\cite[Corollary 1.3]{bellaiche2016eigencurve}), comme application de la lissité de la courbe de Hecke en le point défini par $f_\alpha$. L'élément $\cL^\textrm{BD}_p(f_\alpha)$, défini à multiplication par un élément de $L^\times$ près, est obtenu comme spécialisation d'une fonction $L$ $p$-adique "à deux variables" définie sur un voisinage affinoïde $\cU$ de $f_\alpha$, interpolant les fonctions $L$ $p$-adiques usuelles des points classiques $g\in \cU\bigcap \bZ_{\geq 2}$. 

Nos hypothèses entraînant que $\ob{\rho}$ est irréductible et $p$-distinguée, on peut, comme dans le cas des formes de poids $k\geq 2$, abaisser cette indétermination à $\cO_L^\times$. On fait appel pour cela à la construction de \cite{emerton2006variation} d'une fonction $L$ $p$-adique $L_p^{\an}(\ob{\rho};T)\in\bH_{\ob{\rho}}[[T]]$ à coefficients dans la composante locale $\bH_{\ob{\rho}}$ de l'algèbre de Hecke universelle ordinaire de niveau modéré $N$. Elle est définie à une unité de $\bH_{\ob{\rho}}$ près, dépendant du choix d'une période canonique en famille, et se spécialisant sur la fonction $L$ $p$-adique usuelle de formes propres ordinaires $p$-stabilisées de poids $k\geq 2$ et niveau modéré $N$ (cf. Section \ref{sec:Lpfamille}). La forme $f_\alpha$ définit une spécialisation $\phi : \bH_{\ob{\rho}} \longrightarrow \cO_L$. L'application de $\phi$ aux coefficients de $L_p^{\an}(\ob{\rho};T)$ définit un élément $L_p^{\an}(f_\alpha;T)$ bien défini à $\cO_L^\times$ près, que l'on appelle\textit{ fonction $L$ $p$-adique analytique de $f_\alpha$}. Par analogie avec la Conjecture Principale pour les formes paraboliques primitives $p$-ordinaires de poids $k\geq 2$ (\cite[Conjecture 3.24]{skinner2014iwasawa}), nous proposons une Conjecture Principale pour $f_\alpha$.

\begin{conjecture*}[=Conjecture \ref{IMC_f_alpha}]
	Il existe une unité $u$ de $\Lambda$ telle que 
	$$ u \cdot L_p^{\alg}(f_\alpha,T)=L_p^{\an}(f_\alpha,T).$$
\end{conjecture*}

Le théorème principal de cette partie fournit une évidence en faveur de cette conjecture.
\begin{thC}[=Théorème \ref{ThC}] Il existe un élément $u\in\Lambda \otimes \Qp$ tel que
	$$ u \cdot L_p^{\alg}(f_\alpha,T)=L_p^{\an}(f_\alpha,T).$$
	De plus, si la Conjecture Principale pour les formes paraboliques primitives $p$-ordinaires de poids $k\geq 2$ est vraie, alors la Conjecture Principale pour $f_\alpha$ est vraie.
\end{thC}

Soit \textbf{f} une famille de Hida se spécialisant en $f_\alpha$. Nous prouvons le Théorème C par un argument de passage à la limite sur les spécialisations $\textbf{f}_k$ de poids $k$ de $\textbf{f}$, lorsque $k$ tend $p$-adiquement vers 1. La famille de Hida \textbf{f} définit un certain quotient $\bH_{\textbf{f}}$ de $\bH_{\ob{\rho}}$ qui est une algèbre finie intègre sur l'anneau $\Lambda^\textrm{poids}:=\Zp[[X]]$. Nous illustrons d'abord la preuve sous l'hypothèse (très forte !) que $\bH_{\textbf{f}}$ est isomorphe à l'anneau de séries formelles $\cO_L[[X]]$, et nous expliquerons ensuite comment s'en passer. Après un changement de variables, la forme $f_\alpha$ correspond à la spécialisation $X=0$, et plus généralement, une spécialisation classique de \textbf{f} de poids $k$, de niveau $Np^r$ et caractère $\epsilon_{f}\chi_\zeta\omega^{1-k}$ correspond à la spécialisation $X=\zeta (1+p)^{k-1}-1$. En notant $g_n$ la forme obtenue en spécialisant en $X=(1+p)^{(p-1)p^n}-1$, on peut schématiquement illustrer la preuve de la divisibilité :
$$
\xymatrix{
	L_p^{\alg}(g_n;T) \ar[d]_{(c)}^{n \rightarrow +\infty} & \stackrel{(a)}{\mbox{divise}} & L_p^{\an}(g_n;T) \ar[d]^{(b)}_{n \rightarrow +\infty} \\
	L_p^{\alg}(f_\alpha;T) & & L_p^{\an}(f_\alpha;T).
}
$$ 
Toutes les fonctions $L$ $p$-adiques sont des éléments de l'anneau topologique $\Lambda$, et les divisibilités sont dans $\Lambda\otimes \Qp$. L'élément $L_p^{\alg}(g_n;T)$ est la fonction $L$ $p$-adique algébrique du groupe de Selmer attaché à $g_n$, et la divisibilité (a) est une application d'un célèbre théorème de Kato \cite[Theorem 17.4]{kato2004p}. Notons par ailleurs qu'une divisibilité dans $\Lambda$ est prouvée dans \textit{loc. cit.} lorsque la représentation galoisienne associée à $g_n$ a une grosse image, ce qui ne sera jamais le cas ici. Pour le passage à la limite (b), on considère la fonction $L$ $p$-adique analytique de \textbf{f} obtenue à partir de $L_p^{\an}(\ob{\rho};T)$.  C'est une série formelle à deux variables $L_p^{\an}(X,T)$ par notre hypothèse simplificatrice, et (b) se déduit de la convergence de $L_p^{\an}((1+p)^{(p-1)p^n}-1,T)$ vers $L_p^{\an}(0,T)$ pour la topologie naturelle de $\Lambda$. Pour (c), on construit la représentation ordinaire $\cT_\textbf{f}$ associée à \textbf{f} et à coefficients dans $\bH_\textbf{f}$. On montre que le groupe de Selmer $X_\infty(\textbf{f})$ associé n'a pas de sous-modules pseudo-nuls non-triviaux (Lemme \ref{Och}). Comme l'anneau des coefficients $\bH_\textbf{f}[[T]]$ est un anneau de séries formelles, la fonction $L$ $p$-adique algébrique $L_p^{\alg}(X,T)$ de $X_\infty(\textbf{f})$ possède les propriétés de spécialisation attendues, d'après les résultats généraux de la Section \ref{section3}. Le point (c) se montre alors comme le (b), et comme $L_p^{\alg}(f_\alpha;T)\neq 0$, on peut conclure que $L_p^{\alg}(f_\alpha;T)$ divise $L_p^{\an}(f_\alpha;T)$ par un passage à la limite dans les divisibilités (Lemme \ref{elem}).

L'hypothèse $\bH_\textbf{f}\simeq \cO_L[[X]]$ est beaucoup trop forte et implique notamment que son localisé $\left(\bH_\textbf{f}\right)_\gp$ en $f_\alpha$ est régulier, ce qui est précisément l'énoncé du théorème de lissité de la courbe de Hecke en $f_\alpha$ prouvé par Bellaïche et Dimitrov. En fait, on montre que leur résultat suffit à adapter les arguments précédents. En effet, l'anneau complété $\left(\widehat{\bH}_\textbf{f}\right)_\gp$ est isomorphe à $\ob{\bQ}_p[[X^{1/e}]]$, où $e\geq 1$ est l'indice de ramification de l'application de poids en $f_\alpha$. L'image de $\bH_\textbf{f}$ dans cet anneau est finie sur $\Lambda^\textrm{poids}$, et le théorème d'approximation d'Artin \cite{artin1968solutions} montre qu'elle est contenue dans un anneau de séries convergentes $\cO'[[Y]]$ de la variable $Y=X^{1/e}/p^r$ (pour un certain $r\geq 0$) et à coefficients dans une extension finie $\cO'$ de $\cO_L$ (Proposition \ref{phi_infty}). Cela fournit un paramétrage $\textbf{f}^\dag(Y)$ de $\textbf{f}$ permettant d'adapter entièrement la preuve précédente en remplaçant $L_p^{\alg}(X,T)$ et $L_p^{\an}(X,T)$ par leurs analogues respectifs $L_p^{\alg,\dag}(Y,T)$ et $L_p^{\an,\dag}(Y,T)\in \cO'[[Y,T]]$.

\subsection*{Remerciements} 
Les résultats de cet article sont pour l'essentiel issus de la thèse de doctorat de l'auteur sous la direction de Mladen Dimitrov. Nous le remercions chaleureusement pour nous avoir proposé de travailler sur ce sujet ainsi que pour son aide dans la réalisation de ce travail. Nous remercions également Henri Darmon et Olivier Fouquet pour la relecture et les corrections apportées à la thèse. Merci également à Denis Benois, Gautami Bhowmik, Sheng-Chi Shih, Vincent Pilloni, Jacques Tilouine, Emiliano Torti et à Gabor Wiese pour les discussions autour de l'article. Ce travail est soutenu par le Fonds National de la Recherche, Luxembourg, INTER/ANR/18/12589973 GALF.

\subsection*{Notations}

\textit{Corps de nombres et plongements :} On fixe une clôture algébrique $\ob{\bQ}$ de $\bQ$, ainsi que les plongements $\iota_\ell$ (pour tout nombre premier $\ell$) et $\iota_\infty$ de l'introduction. Lorsque $\ell=p$, on notera toujours $v$ la place de $\ob{\bQ}$ au-dessus de $p$ déterminée par $\iota_p$. On notera $G_E$ le groupe de Galois absolu d'une extension intermédiaire $E$ de $\ob{\bQ}/\bQ$ ou de $\ob{\bQ}_\ell/\bQ_\ell$. Pour une extension galoisienne $E \subseteq F \subseteq \ob{\bQ}$ on notera aussi $I_\ell(F/E)\subseteq \Gal(F/E)$ le groupe d'inertie de la place de $F$ définie par $\iota_\ell$. 

\textit{Représentations de groupes finis :} Soit $G$ un groupe fini et soit \textbf{C} un corps algébriquement clos de caractéristique 0. Toute \textbf{C}-représentation irréductible $(\pi,V_\pi)$ de $G$ définit un idempotent
$$ e_\pi = \dfrac{\dim \pi}{\#G} \sum_{g\in G}\Tr \circ \pi(g^{-1})g \in \textbf{C}[G]. $$
\'Etant donnée une \textbf{C}-représentation $W$ de $G$, on notera $W^\rho:=e_\pi W \subseteq W$ sa composante $\pi$-isotypique.

\textit{Dualité de Pontryagin :} 	
Pour tout $\Zp$-module localement compact $M$, on note $M^\vee=\Hom_\textrm{ct}(M,\Qp/\Zp)$ le dual de Pontryagin de $M$. Le foncteur $M \mapsto M^\vee$ induit une équivalence de catégories entre la catégorie des $\Zp$-modules discrets et de torsion et les $\Zp$-modules compacts.

\textit{Caractères cyclotomiques et algèbre d'Iwasawa :} Soit $\tilde{\chi}_{\cyc} : G_\bQ \longrightarrow \bZ^\times_p$ le caractère cyclotomique et soit $\omega : G_\bQ \longrightarrow \mu_{p-1}$ le caractère de Teichmüller. $\tilde{\chi}_{\cyc}$ induit un isomorphisme entre $\tilde{\Gamma}=\Gal(\bQ(\mu_{p^\infty})/\bQ)$ et $\bZ_p^\times$, et $\omega$ induit un isomorphisme entre $\Gal(\bQ(\mu_{p})/\bQ)$ et $\mu_{p-1}$. Le caractère $\chi_{\cyc}:=\tilde{\chi}_{\cyc}\omega^{-1}$ se factorise via le groupe de Galois $\Gamma$ de la $\Zp$-extension extension cyclotomique, notée $\bQ_\infty/\bQ$, et réalise un isomorphisme $\Gamma \simeq 1+p\Zp$. On notera $\gamma \in \Gamma$ l'image inverse du générateur topologique $u=1+p$ de $1+p\Zp$. On a ainsi $\Gamma \simeq \gamma^\Zp$. L'algèbre d'Iwasawa sur une extension finie $\cO$ de $\Zp$ est l'anneau de groupe complété $\Lambda:=\cO[[\Gamma]]$. Elle est isomorphe à l'anneau des séries formelles en une variable $\cO[[T]]$ après identification de $\gamma$ avec $1+T$.

\textit{Tour cyclotomique et caractères :} On note $\bQ_n$ le $n$-ème étage de la tour cyclotomique, c'est-à-dire le sous-corps de $\bQ_\infty$ fixé par $\Gamma^{p^n}$, ainsi que $\Gamma_n=\Gal(\bQ_n/\bQ) \simeq \bZ/p^n\bZ$. Si $\zeta \in \mu_{p^\infty}(\ob{\bQ}_p)$ est une racine de l'unité d'ordre $p^n$, on notera $\chi_\zeta : G_\bQ \longrightarrow \ob{\bQ}_p^\times$ le caractère galoisien se factorisant par $\Gamma_n$ et envoyant $\gamma$ sur $\zeta$. Pour $n\geq 1$, le caractère $\chi_\zeta$ est d'ordre $p^n$ et de conducteur $p^{n+1}$.
	
	\addtocontents{toc}{\setcounter{tocdepth}{2}}

	\section{Groupe de Greenberg-Selmer d'une représentation ordinaire}
\label{section3}
	\pagenumbering{arabic}
	
\subsection{Idéaux caractéristiques et spécialisations}\label{section2}
\subsubsection{Idéaux caractéristiques}
Nous rappelons quelques définitions et résultats utiles sur les modules de type fini et de torsion sur un anneau $A$ noethérien et intégralement clos (ou, plus généralement, sur un anneau de Krull, cf. \cite[Chap. VII]{bourbaki2007algebre}). Soit $\cP_A$ l'ensemble des idéaux premiers de hauteur 1 de $A$. Pour tout $\gp \in \cP_A$, le localisé $A_\gp$ de $A$ est un anneau de valuation discrète, dont la valuation (normalisée) $v_\gp$ s'étend à $\Frac(A)$. Si $M$ est un module de type fini et de torsion sur $A$, alors le localisé $M\otimes A_\gp$ est de longueur finie $\ell_\gp(M)$ sur $A_\gp$, et l'idéal caractéristique de $M$ est défini comme étant l'idéal entier 
$$\car_A(M) =  \{ x \in A \ : \ v_\gp(x) \geq \ell_\gp(M), \quad	  \forall \gp \in \cP_A\}.$$
L'application qui, à un $A$-module de type fini et de torsion, associe son idéal caractéristique est multiplicative sur les suites exactes. Autrement dit, si $M$ est l'extension de $P$ par $N$ dans la catégorie des $A$-modules de type fini et de torsion, alors  $\car_A(M)=\car_A(N)\car_A(P)$. Lorsque $\car_A (M)$ est l'idéal unité de $A$, c'est-à-dire quand $M_\gp=0$ pour tout $\gp\in \cP_A$, alors on dit que $M$ est pseudo-nul. 
%
%

Lorsque $A$ est isomorphe à un anneau des séries formelles $\cO[[T]]$ à coefficients dans l'anneau des entiers d'une extension finie de $\Qp$, on a un théorème de structure des modules de type fini sur $A$. Notons $\varpi\in\cO$ une uniformisante de $\cO$. Les idéaux premiers de hauteur 1 de $\cO[[T]]$ sont principaux et engendrés soit par $\varpi$, soit par un polynôme $P(T)\in\cO[T]$ distingué (\textit{i.e.}, un polynôme unitaire tel que $P(T) \equiv T^n \mod \varpi$). 

\begin{theorem}[Théorème de structure sur l'algèbre d'Iwasawa] \label{th:structure_sur_algebre_iwasawa}
	Les modules de type fini et pseudo-nuls sur $\cO[[T]]$ sont finis, et réciproquement. Si $M$ est un $\cO[[T]]$-module de type fini, alors il existe des entiers $r\in\bN$, $\mu_i>0$ et $m_j>0$, des polynômes distingués irréductibles $P_j(T)$ et une suite exacte de $\cO[[T]]$-modules :
	$$\xymatrix{
		\left(\textrm{fini}\right) \ar[r] & M  
		\ar[r]& \cO[[T]]^{\oplus r} \bigoplus \left(\bigoplus_i \cO[[T]]/(\varpi^{\mu_i}) \right) \bigoplus \left( \bigoplus_j \cO[T]/(P_j(T)^{m_j}) \right) \ar[r] & \left(\textrm{fini}\right).
	}$$
\end{theorem}

\begin{proof}
	Voir par exemple \cite[Chap. 5, Theorem 3.1]{langCYC}.
\end{proof}

\subsubsection{Changement de bases}
On rappelle ici des résultats élémentaires et bien connus des experts sur les idéaux caractéristiques et la réduction modulo $\gp$. 
\begin{lemme}\label{lemme_pseudo}
	Soit $M$ un $A$-module de type fini et de torsion, soit $\gp \in\cP_A$.
	\begin{enumerate}
		\item Si $M$ est pseudo-nul, alors $M[\gp]$ et $M/\gp M$ sont de torsion sur $A/\gp$.
		\item Si $M/\gp M$ est de torsion sur $A/\gp$, alors $\gp$ ne divise pas $\car_A(M)$.
	\end{enumerate}
\end{lemme}
%
%

\begin{proposition}\label{prop_pseudo}
	Soit $M$ un $A$-module de type fini et de torsion, soit $\gp \in \cP_A$. Supposons que $\gp$ est principal et que $A/\gp$ est encore un anneau intégralement clos. 
	\begin{enumerate}
		\item Si $M$ est pseudo-nul, alors on a $\car_{A/\gp}(M[\gp])=\car_{A/\gp}(M/\gp M)$.
		\item Si $M$ n'a pas de sous-$A$-modules pseudo-nuls différents de $\{0\}$ et si $M /\gp M$ est de torsion sur $A/\gp$, alors $$\car_{A/\gp}(M/\gp M)= \car_A(M)/\gp \car_A(M).$$
	\end{enumerate}
\end{proposition}

\begin{proof}
	
	La preuve du (1) est donnée dans \cite[Lemma 3.1]{ochiai2005euler}, donc nous prouvons uniquement le point (2). Supposons que $M$ n'a pas de sous-modules pseudo-nuls non-triviaux, et que $M /\gp M$ est de torsion sur $A/\gp$. Alors d'après le Lemme \ref{lemme_pseudo}, on sait que $\gp$ n'apparaît pas dans la suite des idéaux premiers $\gp_1, \ldots,\gp_d \in \cP_A$ divisant $\car_A M$. De plus, le théorème de structure des modules de type fini et de torsion sur $A$ (\cite[Chap. VII, §4.4, Théorème 5]{bourbaki2007algebre}) fournit une suite exacte courte 
	$$\xymatrix{
		0 \ar[r] & M  
		\ar[r]& \bigoplus_{i=1}^d A/\gp_i^{m_i} \ar[r] & D \ar[r]&0,
	}$$
	où les $m_i$ sont des entiers positifs et $D$ est un $A$-module pseudo-nul. En considérant la multiplication par un générateur $x$ de $\gp$, on en déduit une suite exacte longue
	$$\xymatrix{
		\bigoplus_i A/\gp_i^{m_i}[\gp] \ar[r] &  D[\gp]
		\ar[r]& M/\gp M \ar[r]& \bigoplus_i A/\gp_i^{m_i} \otimes_A A/\gp \ar[r] & D/\gp D \ar[r]&0.
	}$$
	Comme $x$ n'appartient pas aux $\gp_i$, le premier terme de cette suite exacte est nul. D'après le point (1) appliqué à $D$, on a donc $\car_{A/\gp} (M/\gp M) = \car_{A/\gp} \left(\bigoplus_i A/\gp_i^{m_i} \otimes_A A/\gp\right)$, c'est-à-dire, 
	\[\car_{A/\gp} (M/\gp M) = \prod_i \gp_i^{m_i}/\gp \gp_i^{m_i} = \car_A (M) /\gp \car_A (M).\]
\end{proof}

\subsubsection{Séries caractéristiques}
Lorsque $A$ est factoriel, ses idéaux premiers de hauteur 1 sont tous principaux. C'est le cas, par exemple, quand $A$ est régulier d'après le théorème de Auslander–Buchsbaum, et \textit{a fortiori} lorsque $A$ est une algèbre de séries formelles sur un anneau de valuation discrète. En particulier, l'idéal caractéristique d'un $A$-module $M$ de type fini et de torsion est un idéal non-nul et principal de $A$. On appelle série caractéristique de $M$ tout générateur de $\car_A(M)$. Notons qu'elle est unique à multiplication par une unité de $A$ près.

Dans le cas particulier où $A={\cO}[[T]]$, et où ${\cO}$ est l'anneau des entiers d'une extension finie de $\Qp$, la série caractéristique est une série formelle dont on peut interpréter le terme constant (bien défini à une unité de $\cO$ près) comme un quotient de Herbrand. Le lemme suivant se déduit de la Proposition \ref{prop_pseudo} (2) :

\begin{lemme}\label{cst}
	Supposons que $A={\cO}[[T]]$. Soit $M$ un $A$-module de type fini et de torsion, de série caractéristique $L(M;T)\in \cO[[T]]-\{0\}$. Si $M$ est sans sous-modules finis non-triviaux et si $M/TM$ est fini, alors $L(M;0)\neq 0$, et on a 
	$$L(M;0)\ \dot{=}\ \#(M/TM),$$
	où $\dot{=}$ désigne l'égalité à multiplication par une unité de ${\cO}$ près.
\end{lemme}

\subsection{Groupe de Selmer et fonction $L$ $p$-adique algébrique}

Soient $\cA$ une $\Zp$-algèbre profinie, et $\cT$ un $\cA$-module libre de rang fini muni d'une action continue de $G_\Sigma:=\Gal(\bQ_\Sigma/\bQ)$, où $\Sigma$ est un ensemble fini de places de $\bQ$ contenant $p$ et $\infty$, et où $\bQ_\Sigma \subseteq \ob{\bQ}$ est la plus grande extension algébrique de $\bQ$ non-ramifiée en dehors de $\Sigma$. On se donne une suite exacte courte de $G_{\Qp}$-modules libres sur $\cA$
$$\xymatrix{
	0 \ar[r] & \cT^+ \ar[r] & \cT \ar[r] & \cT^- \ar[r] & 0.
}
$$
On pose $\cD=\cT \otimes_{\cA} {\cA}^\vee$ et $\cD^\pm = \cT^\pm \otimes_{\cA} {\cA}^\vee$, où $\cA^\vee=\Hom_\textrm{ct}(\cA,\Qp/\Zp)$ désigne le dual de Pontryagin de $\cA$. Ce sont des $\cA$-modules co-libres, c'est-à-dire que leurs duaux de Pontryagin sont libres sur $\cA$. Pour tout $n \in \bN \cup\{\infty\}$, on définit $\cD_n:=\Ind_{{\bQ_n}}^{{\bQ}} \cD$ comme étant l'induite de $G_{\bQ_n}$ à $G_{\bQ}$ de $\cD$, vu comme $G_{\bQ_n}$-module discret. Autrement dit, il s'agit du $\cA$-module des fonctions localement constantes $f : G_{\bQ} \longrightarrow \cD$ telles que 
$$\forall g\in G_{\bQ}, \ \forall h \in G_{\bQ_n}, \ f(hg)=hf(g).$$ 
Il est muni de l'action de $G_{\bQ}$ donnée par la formule $(\gamma.f)(g)=f(g\gamma)$, et s'il on pose $\cA_n=\cA[\Gamma_n]$ si $n$ est fini et $\cA_\infty=\cA[[\Gamma]]$, alors il est canoniquement isomorphe à $\cD \otimes_{\cA} \cA_n$, muni de l'action de $G_{\bQ}$ sur les deux facteurs du produit tensoriel. On définit de même $\cD_n^\pm$ comme étant $\Ind_{{\bQ_{p,n}}}^{{\bQ_p}} \cD^\pm$, où $\bQ_{p,n}$ est le $n$-ième étage de la tour cyclotomique de $\Qp$. Comme $\Gal(\bQ_{p,n}/\Qp)\simeq\Gamma_n$, le $G_{\Qp}$-module $\cD_n^\pm$ est canoniquement isomorphe à $\cD^\pm \otimes_{\cA} \cA_n$ et on obtient une suite exacte courte
$$\xymatrix{
	0 \ar[r] & \cD_n^+ \ar[r] & \cD_n \ar[r] & \cD_n^- \ar[r] & 0.
}
$$

Nous travaillerons avec la définition suivante (cf. \cite[Section 3]{skinner2014iwasawa} ou \cite{greenberg2015structure}).

\begin{definition}\label{Selgen}
	Soit $n\in \bN \cup\{\infty\}$. On appelle groupe de Selmer de niveau $n$ attaché à $(\cT,\cT^+)$ le $\cA_n$-module défini comme étant le noyau de la flèche de restriction globale-locale : 
	
	\[\begin{array}{rcl} \Sel_n(\cT,\cT^+)
	&:=&\ker \left[ H^1(\bQ_\Sigma/\bQ,\cD_n) \longrightarrow H^1(I_{p},\cD_n^-) \times \prod_{\ell \in \Sigma, \ell\neq p} H^1(I_{\ell},\cD_n)  \right] \\
	&\simeq& \ker \left[ H^1(\bQ,\cD_n) \longrightarrow H^1(I_{p},\cD_n^-) \times \prod_{\ell\neq p} H^1(I_{\ell},\cD_n)  \right].\end{array}\]
	Le groupe de Selmer dual est défini comme étant $$X_n(\cT,\cT^+)=\Sel_n(\cT,\cT^+)^\vee.$$
\end{definition}

D'après les propriétés de base de la cohomologie des modules discrets, on a $\Sel_\infty(\cT,\cT^+) = \varinjlim_n \Sel_n(\cT,\cT^+)$. Par ailleurs, $H^1(\bQ_\Sigma/\bQ,\cD_n)^\vee$ étant de type fini sur $\cA_n$, cela est encore vrai pour $X_n(\cT,\cT^+)$ (voir par exemple \cite[Chap VIII, §3, Theorem 20]{neukirchcohomology}). Par ailleurs, on ne considérera jamais la cohomologie locale aux places archimédiennes, puisque qu'elle est toujours triviale lorsque $p$ est impair.

\begin{definition}\label{def:L_p(X)}
	Supposons que le groupe de Selmer dual $X_\infty(\cT,\cT^+)$ est de torsion sur $\cA[[\Gamma]]$. La fonction $L$ $p$-adique algébrique $L^\alg_p(\cT,\cT^+) \in \cA[[\Gamma]]-\{0\}$ attachée à $(\cT,\cT^+)$ est la série caractéristique associée à $X_\infty(\cT,\cT^+)$.
\end{definition}

\subsection{Lemme de Shapiro}

Soit $G$ un sous-groupe fermé de $G_{\Sigma}$ et soit $H$ un sous-groupe ouvert normal de $G$. Le lemme de Shapiro relie la cohomologie de $\cD$, vu comme $H$-module, à la cohomologie du $G$-module induit $\Ind_H^G \cD$. On a, pour tout entier $i\geq 0$, l'isomorphisme de Shapiro (fonctoriel par rapport à $\cD$ et à $H$, voir \cite[Chap. 8, §1.]{sel}) : 
\begin{align}\label{Shapiro}
H^i(H,\cD) \simeq H^i(G,\Ind^G_H \cD).
\end{align}
On en déduit le lemme suivant (par passage à la limite si $n=\infty$).

\begin{lemme}\label{Def_alternative_Sel_n}
	L'isomorphisme de Shapiro induit un isomorphisme : 
	$$
	\Sel_n(\cT,\cT^+)\simeq\ker \left[ H^1(\bQ_{\Sigma}/\bQ_n,\cD) \longrightarrow H^1(I_{p}(\bQ_{\Sigma}/\bQ_n),\cD^-) \times \prod_{\ell \in \Sigma,\ell\neq p} H^1(I_{\ell}(\bQ_{\Sigma}/\bQ_n),\cD)  \right].
	$$
\end{lemme}

\begin{remarque}\label{inertie_décomp}
	Soit $\lambda$ la place de $\bQ_\infty$ définie par $\iota_\ell$, pour un nombre premier $\ell\neq p$. Le groupe de Galois $\Gal(\bQ_\ell^\textrm{nr}/\bQ_{\infty,\lambda})=G_{\bQ_{\infty,\lambda}}/I_\ell(\bQ_\Sigma/\bQ_\infty)$ est d'ordre premier à $p$, donc la suite exacte d'inflation-restriction montre que l'application de restriction 
	$$H^1(\bQ_{\infty,\lambda},\cD) \longrightarrow H^1(I_\ell(\bQ_\Sigma/\bQ_\infty),\cD)$$
	est injective. $\Sel_\infty(\cT,\cT^+)$ s'identifie ainsi au noyau de l'application 
	$$
	H^1(\bQ_{\Sigma}/\bQ_\infty,\cD) \longrightarrow H^1(I_{p}(\bQ_{\Sigma}/\bQ_\infty),\cD^-) \times \prod_{\lambda |\ell \in \Sigma,\ell\neq p} H^1(\bQ_{\infty,\lambda},\cD).
	$$
	Le lemme de Shapiro montre alors que l'on a aussi :
	$$\Sel_\infty(\cT,\cT^+)=\ker \left[ H^1(\bQ_\Sigma/\bQ,\cD_\infty) \longrightarrow H^1(I_{p}(\bQ_\Sigma/\bQ_\infty),\cD_\infty^-) \times \prod_{\ell \in \Sigma, \ell\neq p} H^1(\bQ_\ell,\cD_\infty)  \right],$$
	où $\cD_\infty$ est défini plus haut comme étant $\cD \otimes_\cA \cA[[\Gamma]]$. 
\end{remarque}

\subsection{Changement de bases}
On conserve les notations précédentes. Pour tout idéal $\ga \subseteq \cA$, on a un isomorphisme naturel 
$$ (\cA/\ga)^\vee = \cA^\vee[\ga],$$
et donc on a $\cT/\ga \otimes_{\cA/\ga} (\cA/\ga)^\vee =\cD[\ga]$ (et similairement pour $\cT^+$). En particulier, on a une application naturelle 
\begin{equation}\label{eq:changement_base}
\Sel_\infty(\cT/\ga,\cT^+/\ga) \longrightarrow \Sel_\infty(\cT,\cT^+)[\ga].
\end{equation}
Il s'agit d'un isomorphisme sous certaines conditions générales, comme le précise la proposition suivante, inspirée de \cite[Proposition 3.7]{skinner2014iwasawa}.

\begin{proposition}\label{changement_base}
	Soit $n\in\bN\cup\{\infty\}$. Supposons que l'idéal $\ga$ est principal, que $I_p(\ob{\bQ}/\bQ_n)$ agit trivialement sur $\cD^-$, et que les $\cA$-modules $\cD^{G_{\bQ_n}}$ et $\cD^{I_\ell(\bQ_\Sigma/\bQ_n)}$ sont divisibles pour tout $\ell \in \Sigma$.
	
	Alors l'application de changement de base (\ref{eq:changement_base}) est un isomorphisme 
	$$\Sel_n(\cT/\ga,\cT^+/\ga) \simeq \Sel_n(\cT,\cT^+)[\ga],$$ 
	et donc elle induit par dualité un isomorphisme 
	$$X_n(\cT,\cT^+)/\ga X_n(\cT,\cT^+) \simeq X_n(\cT/\ga,\cT^+/\ga).$$
\end{proposition}

\begin{proof}
	Soit $x\in A$ un générateur de $\ga$. D'après le Lemme \ref{Def_alternative_Sel_n}, on a 
	$$
	\Sel_n(\cT,\cT^+)=\ker \left[ H^1(\bQ_{\Sigma}/\bQ_n,\cD) \longrightarrow H^1(I_{p}(\bQ_{\Sigma}/\bQ_n),\cD^-) \times \prod_{\ell \in \Sigma,\ell\neq p} H^1(I_{\ell}(\bQ_\Sigma/\bQ_n),\cD)  \right].
	$$
	On montre d'abord que l'application naturelle $\tau:H^1(\bQ_{\Sigma}/\bQ_n,\cD[\ga]) \longrightarrow H^1(\bQ_{\Sigma}/\bQ_n,\cD)[\ga]$ est un isomorphisme. Comme $\ga$ est engendré par $x$ et que $\cD$ est $\cA$-divisible, on a une suite exacte courte induite par la multiplication par $x$ :
	$$\xymatrix{ 0 \ar[r] & \cD[\ga] \ar[r] & \cD \ar[r] & \cD \ar[r] &  0. }$$
	Celle-ci induit une suite exacte longue associée en cohomologie :
	$$\xymatrix{ \cD^{G_{\bQ_n}} \ar[r] & \cD^{G_{\bQ_n}} \ar[r] & H^1(\bQ_{\Sigma}/\bQ_n,\cD[\ga]) \ar[r]^{\tau} & H^1(\bQ_{\Sigma}/\bQ_n,\cD) \ar[r] &  H^1(\bQ_{\Sigma}/\bQ_n,\cD), }$$
	où la première et la dernière flèche sont induites par la multiplication par $x$. Par hypothèse la première flèche est surjective, et par définition le noyau de la dernière flèche est égal à $H^1(\bQ_{\Sigma}/\bQ_n,\cD)[\ga]$. Donc $\tau$ est un isomorphisme. 
	
	Pour montrer que $\tau$ envoie $\Sel_n(\cT/\ga,\cT^+/\ga)$ sur $\Sel_n(\cT,\cT^+)[\ga]$, il suffit de prouver que les applications 
	$$ H^1(I_{p}(\bQ_{\Sigma}/\bQ_n),\cD^-[\ga]) \longrightarrow H^1(I_{p}(\bQ_{\Sigma}/\bQ_n),\cD^-)[\ga]$$
	et, pour $\ell\in \Sigma$,
	$$ H^1(I_{\ell}(\bQ_\Sigma/\bQ_n),\cD[\ga]) \longrightarrow H^1(I_{\ell}(\bQ_\Sigma/\bQ_n),\cD)[\ga])$$
	sont injectives. La première l'est clairement, car les $H^1$ sont des $\Hom$ par hypothèse, et la deuxième est injective par le même argument que précédemment (et nécessitant que $\cD^{I_\ell(\bQ_\Sigma/\bQ_n)}$ soit $\cA$-divisible).
\end{proof}

\subsection{Structure des groupes de Selmer}
\label{sec:structure_selmer}
Jusqu'à la fin de cette section, nous supposons que $\cA$ est un anneau de séries formelles à coefficients  dans l'anneau des entiers $\cO$ d'une extension finie $\Qp$, de corps résiduel $\bF$. On a ainsi $\cA=\cO[[X_1,\ldots,X_m]]$, $m\geq 1$ ou bien $\cA=\cO$. Notons $\gm_{\cA}$ son idéal maximal. On conserve les notations des paragraphes précédents, et on note $d$ (resp. $d^\pm$) le rang de $\cT$ (resp. de $\cT^\pm$). On introduit de plus le $\cA$-module libre $\cT^*=\Hom_\Zp(\cD,\mu_{p^\infty})$ de rang $d$ et son dual $\cD^*=\cT^* \otimes_{\cA} {\cA}^\vee$. On note encore les modules induits avec un indice. Le résultat principal de \cite{greenberg2015structure} donne des conditions suffisantes pour qu'un groupe de Selmer dual, défini de manière générale, n'admettent pas de sous-modules pseudo-nuls différents de 0. Combinée à la Proposition \ref{prop_pseudo}, cette propriété sera utile pour calculer les diverses spécialisations de fonctions $L$ $p$-adiques algébriques. Nous simplifions le critère dans le but de l'appliquer dans la suite aux groupes de Selmer attachés à un motif muni d'une $p$-stabilisation ordinaire, ou bien à une déformation ordinaire. 

\begin{proposition}\label{PsN}
	Supposons que : 
	\begin{itemize}
		\item[(a)] $X_\infty(\cT,\cT^+)$ est de torsion sur $\cA[[\Gamma]]$.
		\item[(b)] $H^2(\bQ_\Sigma/\bQ_\infty,\cD)^\vee$ est de torsion sur $\cA[[\Gamma]]$.
		\item[(c)] $\rg_{\cA} \cT^{\Frob_\infty}=d^+$
		\item[(d)] L'inertie en $p$ de $\bQ_\infty$ agit trivialement sur $\cD^-$.
		\item[(e)] En tant que $\bF[G_\bQ]$-représentation, $\cD[\gm_{\cA}]$ n'a pas de quotient isomorphe à $\bF$ ou à $\mu_p$.
	\end{itemize}
	Alors $X_\infty(\cT,\cT^+)$ n'a pas de sous-modules pseudo-nuls non-nuls.
\end{proposition}

\begin{proof}
	On va montrer que l'on peut appliquer \cite[Proposition 4.1.1]{greenberg2015structure} au $\cA[[\Gamma]]$-module $\cD_\infty$. Avec les notations de \textit{loc. cit.}, on pose $L_\ell=0$ pour $l \in \Sigma, l\neq p$ et enfin 
	$$L_p=\ker\left[H^1(\bQ_p,\cD_\infty) \longrightarrow H^1(I_p,\cD_\infty^-)\right].$$
	La Remarque \ref{inertie_décomp} montre que le groupe de Selmer défini dans \textit{loc. cit.} est bien égal à $\Sel_\infty(\cT,\cT^+)$, et par ailleurs, la conclusion de la Proposition 4.1.1 de \textit{loc. cit.} est équivalente à ce que $X_\infty(\cT,\cT^+)$ n'a pas de sous-modules pseudo-nuls non-triviaux d'après \cite[Proposition 2.4]{greenberg2006structure}. 
	
	La condition $\textrm{RFX}(\cD_\infty)$ revient à dire que $\cA[[\Gamma]]\simeq \cO[[X_1,\ldots,X_{m+1}]]$ est un module réflexif sur $\Zp[[X_1,\ldots,X_{m+1}]]$, ce qui est clairement vrai puisqu'il est libre. La condition $\textrm{LEO}(\cD_\infty)$ est aussi vérifiée d'après notre hypothèse (b). Montrons que l'on a $\textrm{LOC}_\ell^{(1)}(\cD_\infty)$ pour tout $\ell$, c'est-à-dire que $(\cT_\infty^*)^{G_{\bQ_\ell}}=0$. Il suffit de prouver que le rang sur $\cA[[\Gamma]]$ de $(\cT_\infty^*)^{G_{\bQ_\ell}}$ est nul. D'après \cite[Proposition 3.10]{greenberg2006structure}, celui-ci est égal au corang de $\left(\cD_\infty^*\right)^{G_{\bQ_\ell}}=H^0(\bQ_\ell,\cD_\infty^*)$. D'après le lemme de Shapiro, on a $H^0(\bQ_\ell,\cD_\infty^*) = \prod_{\lambda|\ell} H^0(\bQ_{\infty,\lambda},\cD^*)$. Ce dernier module est de corang fini sur $\cA$, en particulier son corang sur $\cA[[\Gamma]]$ est nul, ce qu'on voulait démontrer.
	
	Montrons maintenant que $L_p^\vee$ n'a pas de sous-$\cA[[\Gamma]]$-modules pseudo-nuls non-triviaux. On a 
	$$L_p=\ker\left[ \xymatrix{H^1(\bQ_p,\cD_\infty) \ar[r]^a & H^1(\bQ_p,\cD_\infty^-) \ar[r]^b & H^1(I_p,\cD_\infty^-)}\right].$$
	L'application $a$ est surjective. En effet, $\coker a$ s'injecte dans $H^2(\bQ_p,\cD_\infty^+)$. D'après le théorème de dualité locale de Tate, on a $H^2(\bQ_p,\cD_\infty^+) \simeq \left( (\cT_\infty^{+,*})^{G_{\Qp}}\right)^\vee$ qui est trivial car $(\cT_\infty^{+,*})^{G_{\Qp}}=0$ par le même argument que précédemment. On a ainsi une suite exacte courte :
	$$\xymatrix{ 0 \ar[r] &  \ker a \ar[r] & L_p \ar[r] & \ker b \ar[r] & 0. }$$
	Il suffit de montrer que $(\ker a)^\vee$ et $(\ker b)^\vee$ n'ont pas de sous-modules pseudo-nuls non-triviaux. Comme $\ker a$ est un quotient de $H^1(\Qp,\cD_\infty^+)$, il suffit de montrer la même propriété pour $H^1(\Qp,\cD_\infty^+)^\vee$. Or ceci est vrai d'après \cite[Proposition 3.7]{greenberg2006structure}, qui s'applique car on sait que $H^2(\bQ_p,\cD_\infty^+)=0$ (et les $H^3$ apparaissant dans \textit{loc. cit.} sont nuls car $G_{\Qp}$ est de dimension cohomologique égale à 2, cf. \cite[§5.1.b)]{serrecohomologie}). Décrivons maintenant $\ker b$. D'après la suite exacte d'inflation-restriction, on a $\ker b \simeq H^1(\hat{\bZ},H^0(I_p,\cD_\infty^-))$ où $\hat{\bZ}$ est topologiquement engendré par $\Frob_p$. D'après le lemme de Shapiro, on a $H^0(I_p,\cD_\infty^-) = H^0(I_p(\ob{\bQ}/\bQ_\infty),\cD^-)$ qui est égal à $\cD^-$ d'après notre hypothèse (d). On a donc $\ker b \simeq \cD^-/(\Frob_p-1)\cD^-$ et ainsi $(\ker b)^\vee$ est un sous-module de $(\cD^-)^\vee$. En tant que modules sur l'anneau $\cA[[\Gamma]] \simeq \cA[[T]]$, on a $(\cD^-)^\vee \simeq {\cA}^{d^-} =  (\cA[[T]]/(T))^{d^-}$. Ce dernier module n'a pas de sous-modules pseudo-nuls non-triviaux, donc $(\ker b)^\vee$ non plus et donc $L_p^\vee$ non plus.
	
	On vérifie maintenant l'hypothèse $\textrm{CRK}(\cD_\infty,\cL)$, c'est-à-dire l'égalité des corangs sur $\cA[[\Gamma]]$ : 
	$$\corg (H^1(\bQ_\Sigma/\bQ,\cD_\infty)) = \corg (\Sel_\infty(\cT,\cT^+)) + \sum_{\ell\in \Sigma-\{\infty\}} \corg (Q_\ell(\cT,\cT^+)),$$
	où $Q_\ell(\cT,\cT^+)=H^1(\bQ_\ell,\cD_\infty)/L_\ell$ pour tout premier $l \in \Sigma$. On va montrer que les deux termes de l'égalité sont égaux à $d^-$ sous les hypothèses de la proposition. D'après \cite[Proposition 4.1]{greenberg2006structure}, on a 
	$$\corg(H^1(\bQ_\Sigma/\bQ,\cD_\infty)) = \corg(\cD_\infty^{G_\bQ}) + \corg(H^2(\bQ_\Sigma/\bQ_\infty,\cD)) + d- \corg(\cD_\infty^{\Gal(\bC/\bR)}).$$
	D'après les hypothèses (b) et (e), les deux premiers termes de la somme sont nuls. D'autre part, on a $\corg(\cD_\infty^{\Gal(\bC/\bR)}) = \corg_{\cA} (\cD^{\Gal(\bC/\bR)}) = \rg_{\cA} \cT^{\Gal(\bC/\bR)} =d^+$. Cela montre que le corang de $H^1(\bQ_\Sigma/\bQ,\cD_\infty)$ est bien égal à $d^-$.  Calculons maintenant le terme de droite de l'égalité de $\textrm{CRK}(\cD_\infty,\cL)$. On va montrer que le corang de $Q_\ell(\cT,\cT^+)$ est $0$ si $l\neq p$ et qu'il est égal à $d^-$ pour $\ell=p$. Cela terminera la vérification de $\textrm{CRK}(\cD_\infty,\cL)$ vu que, d'après l'hypothèse (a), le $\cA[[\Gamma]]$-module $\Sel_\infty(\cT,\cT^+)$ a un corang égal à 0. Si $\ell\neq p$, on a $Q_\ell(\cT,\cT^+) = H^1(\bQ_\ell,\cD_\infty)$. Son corang est égal à la somme des corangs de $H^0(\bQ_\ell,\cD_\infty)$ et de $H^0(\bQ_\ell,\cT_\infty^*)^\vee$ d'après la Proposition 4.2(b) de \textit{loc. cit.} et d'après la dualité locale de Tate. Or ces deux modules sont de co-torsion sur $\cA[[\Gamma]]$ d'après le même argument utilisant le lemme de Shapiro que précédemment, donc de corang nul. Enfin, pour $\ell=p$, on a vu que $L_p$ et $\ker a$ ont le même corang. Comme $a$ est surjectif, $Q_p(\cT,\cT^+)$ et $H^1(\Qp,\cD_\infty^-)$ ont le même corang. D'après la Proposition 4.2(a) de \textit{loc. cit.} et un argument identique à celui pour $\ell\neq p$, ce dernier a un corang égal à $d^-$. Ceci termine la vérification de l'hypothèse $\textrm{CRK}(\cD_\infty,\cL)$. 
	
	Pour appliquer \cite[Proposition 4.1.1]{greenberg2015structure}, il suffit maintenant de vérifier l'une des trois hypothèses supplémentaires. La condition (b) stipulant que $\cD_\infty$ est co-libre et que  la $G_{\bQ}$-représentation résiduelle de $\cD_\infty$ n'a pas de quotient isomorphe à $\mu_p$ est clairement vérifiée, car cette dernière est isomorphe à $\cD[\gm_{\cA}]$ et qu'elle vérifie la même condition d'après notre hypothèse (e).
\end{proof}
	
	\section{Structure du groupe de Selmer d'une représentation d'Artin}
\label{section4}
\subsection{Définition et résultat principal} \label{sec:différente_inverse}
Soit $M=[\rho]$ un motif d'Artin \textit{absolument irréductible et non-trivial} sur $\bQ$, à coefficients dans un corps de nombres $E$. On note $\rho : G_\bQ \longrightarrow \GL_E(V)$ la représentation d'image finie qui lui est associée. Quitte à agrandir $E$, on suppose qu'il contient toutes les valeurs propres des éléments de l'image de $\rho$. On note $L \subseteq \ob{\bQ}_p$ le complété de $E$ en $v$, et $\cO=\cO_L$. Le groupe de Selmer attaché à $M$ dépend du choix que l'on fixe dans toute la suite, 
\begin{itemize}
	\item d'une $p$-stabilisation ordinaire $V^+$ de $V$, c'est-à-dire un sous-$E$-espace vectoriel $V^+$ de $V$ de dimension $d^+=\dim_E H^0(\bR,V)$, stable par $G_{\Qp}$, et dont le quotient $V^-:=V/V^+$ est $G_{\Qp}$-non-ramifié,
	\item et d'un $\cO$-réseau $G_{\bQ}$-stable $T_p$ de la réalisation $p$-adique $V_p:=V \otimes_{E,\iota_p} L$ de $[\rho]$.  
\end{itemize}

On note encore $\rho : G_\bQ \longrightarrow \GL_L(V_p)$ la représentation associée à $V_p$, et on note $V_p^\pm= V^\pm \otimes_{E,\iota_p} L$.

\begin{remarque}\label{rem:choix_structure_rationnelle}
	Il est important de noter qu'un choix arbitraire de sous-$L[G_{\Qp}]$-module $V_p^+$ de $V_p$ de la bonne dimension et avec un quotient non-ramifié ne donne \textbf{pas} nécessairement une $p$-stabilisation de $M$, car $V_p^+$ n'a pas toujours de structure $E$-rationnelle (\textit{i.e.}, il n'existe pas nécessairement un sous-$E$-espace vectoriel $V^+$ de $V$ tel que $V_p^+=V^+ \otimes_{E,\iota_p} L \subseteq V_p$). Néanmoins, si par exemple $\rho$ est non-ramifiée en $p$, alors $\rho(\Frob_p)_{|V}$ est diagonalisable, et il suffit que chacune des valeurs propres de $\rho(\Frob_p)_{|V_p^+}$ soient des valeurs propres simples de $\rho(\Frob_p)_{|V}$ pour que $V_p^+$ ait une structure $E$-rationnelle.  
\end{remarque}

On définit à présent le groupe de Selmer du motif $[\rho]$.  On note $T_p^+ := V_p^+ \cap T_p$ et $T_p^-=\im\left(T_p \subseteq V_p \hookrightarrow V_p^-\right)$. Le $\cO$-module $T_p^\pm$ est libre de rang $d^\pm$, et définit une suite exacte courte de $\cO[G_{\Qp}]$-modules 
$$\xymatrix{
	0 \ar[r] & T_p^+ \ar[r] & T_p \ar[r] & T_p^- \ar[r] & 0.
}
$$
L'application $y \mapsto \left( x \mapsto \Tr_{L/\Qp}(xy)\right)$ induit un isomorphisme
$$L/\gd^{-1}_L \simeq \Hom_\textrm{ct}(\cO,\Qp/\Zp)=\cO^\vee,$$
où $\gd^{-1}_L$ différente inverse de l'extension $L/\Qp$. On fixe dans la suite un générateur de $\gd^{-1}_L$, ce qui permet d'identifier $\cO^\vee$ avec $L/\cO$, et donc d'identifier le $\cO$-module co-libre $D_p:=T_p \otimes_\cO \cO^\vee$ avec $V_p/T_p$.

\begin{definition}\label{def:sel_rho}
	Soit $n \in \bN \bigcup \{\infty\}$. Le groupe de Selmer de niveau $n$ attaché au motif $[\rho]$, muni d'une $p$-stabilisation ordinaire $V^+\subseteq V$ et d'un réseau $G_{\bQ}$-stable $T_p$, est défini comme étant
	$$\Sel_n(\rho,\rho^+) = \Sel_n(T_p,T_p^+).$$
	On note aussi $$X_n(\rho,\rho^+) := \Sel_n(\rho,\rho^+)^\vee,$$ 
	le groupe Selmer dual. C'est un module de type fini sur $\cO[\Gamma_n]$ si $n\in\bN$, et sur $\Lambda=\cO[[\Gamma]]$ si $n=\infty$. Si de plus $X_\infty(\rho,\rho^+)$ est de torsion sur $\Lambda$, alors on note 
	$$L_p^{\alg}(\rho,\rho^+;T) := L_p^{\alg}(X_\infty(\rho,\rho^+);T)$$
	sa fonction $L$ $p$-adique algébrique (cf. Définition \ref{def:L_p(X)}).
\end{definition}

\begin{remarque}\label{rq:sel}
	D'après le Lemme \ref{Def_alternative_Sel_n}, on a les descriptions alternatives suivantes :
	$$\begin{array}{rcl}
	\Sel_n(\rho,\rho^+)&=& 
	\ker \left[ H^1(\bQ_\Sigma/\bQ_n,D_p) \longrightarrow H^1(I_{p},D_p^-) \times \prod_{\ell\in \Sigma, \ell \neq p} H^1(I_{\ell},D_p)  \right] \\
	&=&\ker \left[ H^1(\bQ_n,D_p) \longrightarrow H^1(I_{p},D_p^-) \times \prod_{\ell \neq p} H^1(I_{\ell},D_p)  \right].
	\end{array}$$
\end{remarque}

Deux $\cO$-réseaux $G_{\bQ}$-stables $T_1$ et $T_2$ de $V_p$ définissent le même groupe de Selmer quand ils sont homothétiques. Plus généralement, on a le résultat facile suivant :

\begin{proposition}\label{prop:dépendance_réseau}
	Le rang de  $\Lambda$-module $X_\infty(\rho,\rho^+)$ ne dépend pas du choix du réseau $T_p$. S'il est de torsion, alors sa fonction $L$ $p$-adique algébrique est bien définie au $\mu$-invariant près (\textit{i.e.}, à une puissance d'une uniformisante de $\cO$ près). Si de plus $\rho$ est résiduellement irréductible, alors elle ne dépend pas du choix de $T_p$.
\end{proposition}

On note $\varpi$ une uniformisante de $\cO$ et on note $H\subseteq \ob{\bQ}$ le corps de nombres découpé par $\rho$.

\begin{theorem}\label{th:sel_n_fini_etc}
	Supposons $d^+=1$, ou bien supposons que la conjecture de Schanuel $p$-adique faible est vraie (cf. Conjecture \ref{conj:Schanuel}).
	\begin{enumerate}
		
		\item[(i)] Pour tout entier $n\in\bN$, si $H^0(\bQ_n,V)=0$, alors le $\cO$-module $\Sel_n(\rho,\rho^+)$ est fini.
		\item[(ii)] Le $\Lambda$-module $X_\infty(\rho,\rho^+)$ est de torsion.
		\item[(iii)] Si $H\cap \bQ_\infty =\bQ$, alors $L_p^{\alg}(\rho,\rho^+;T)$ ne s'annule pas en les $\zeta-1$, où $\zeta$ parcourt $\mu_{p^\infty}-\{1\}$. En posant $e(\rho,\rho^+):=\dim_L H^0(\Qp,V_p^-)$, on a de plus :
		$$L_p^{\alg}(\rho,\rho^+;0)\neq 0 \Longleftrightarrow e(\rho,\rho^+)=0,$$  
		$$\ord_{T} L_p^{\alg}(\rho,\rho^+;T) \geq e(\rho,\rho^+).$$
		
		\item[(iv)] Supposons que $D_p[\varpi]^G=\Hom_G(D_p[\varpi],\mu_p)=0$. Alors $X_\infty(\rho,\rho^+)$ n'a pas de sous-$\Lambda$-module fini non-nul.
	\end{enumerate}
\end{theorem}

Lorsque $H^0(\bQ_\infty,V)\neq 0$, on peut déjà donner une preuve directe du Théorème \ref{th:sel_n_fini_etc}.
 
 \begin{proposition}
 	Si $H^0(\bQ_\infty,V)\neq 0$ alors Théorème \ref{th:sel_n_fini_etc} est vrai.
 \end{proposition}
\begin{proof}
	On vérifie d'abord que $V$ est en fait de dimension 1 et que $\rho$ est un caractère multiplicatif de $\Gamma$. $V$ est en effet supposée irréductible, donc le sous-espace $G$-stable $H^0(\bQ_\infty,V)$ doit être égal à $V$ tout entier. Donc $H\subseteq \bQ_{_\infty}$ est abélien sur $\bQ$, et donc $\rho$ est nécessairement de dimension 1 et peut être vue comme un caractère multiplicatif de $\Gamma_{n_0}$ pour $n_0\geq 0$ assez grand. On a en particulier $d^+=d=1$ car $\rho$ est pair. Comme $\rho$ est supposée non-triviale, la seule assertion du Théorème \ref{th:sel_n_fini_etc} à prouver est (ii). Pour $n\geq n_0$, on a $H^1(\bQ_n,D_p)=\Hom(G_{\bQ_n}, \Qp/\Zp)\otimes_{\Zp} \cO$, et même 
	$$\Sel_n(\rho,\rho^+)=\Hom(\gX_n, \Qp/\Zp)\otimes_{\Zp} \cO,$$
	où $\gX_n$ est le groupe de Galois de la pro-$p$-extension abélienne maximale de $\bQ_n$ non-ramifiée en-dehors de $p$. En prenant la limite sur $n$ ainsi que le dual de Pontryagin, $X_\infty(\rho,\rho^+)\simeq \gX_\infty \otimes_{\Zp}\cO$, qui est de $\Lambda$-torsion d'après \cite[Theorem 17]{iwasawa1973zl} ou \cite[Theorem 13.31]{washington1997introduction}.
\end{proof}

La suite et fin de la Section \ref{section4} est dédiée à la preuve du Théorème \ref{th:sel_n_fini_etc} lorsque $H^0(\bQ_\infty,V)=0$. La preuve du (i) suit le Théorème \ref{th:sel_n_fini} et les Sections \ref{sec:isom_de_restriction},  \ref{sec:TCC} et \ref{sec:BB}. La preuve du (ii) et (iii) suit le Corollaire \ref{coro:sel_infty_torsion+zéros_triviaux}, et celle du (iv) est donnée après la Proposition \ref{prop:sel_infty_sans_sous_modules_finis}.

\subsection{Inflation-restriction}\label{sec:isom_de_restriction}
\begin{notation}
	Soit $G$ le groupe de Galois de l'extension $H/\bQ$. Pour $n\in \bN \cup \{\infty\}$, on note aussi $H_n$ le compositum de $H$ et de $\bQ_n$, et $M_n/H_n$ la pro-$p$-extension abélienne maximale de $H_n$ non-ramifiée en-dehors des places de $H_n$ divisant $p$. On pose enfin
	
	\[G_n:=\Gal(H_n/\bQ), \qquad G^{(n)}:=\Gal(H_n/\bQ_n), \]
	\[ \gX_n:=\Gal(M_n/H_n). \]
	
\end{notation} 



On a la suite exacte d'inflation-restriction :

$$\xymatrix{ 
	0 \ar[r] & H^1(G^{(n)},D_p) \ar[r]^{\inf} & H^1(\bQ_n,D_p)  \ar[r]^{\res \quad} & \Hom_{G^{(n)}}(G_{H_n},D_p) \ar[r] & H^2(G^{(n)},D_p).  
} $$

L'image de $\Sel_n(\rho,\rho^+)$ par l'application de restriction dans $\Hom_{G^{(n)}}(G_{H_n},D_p)$ est incluse dans $\Hom_{G^{(n)}}(\gX_n,D_p)$, et plus précisément dans 
$$\Sel'_n(\rho,\rho^+):= \ker \left[ \Hom_{G^{(n)}}(\gX_n,D_p) \longrightarrow \Hom_{G^{(n)}_v}(I_p(M_n/H_n),D_p^-) \right],$$
où $G^{(n)}_v$ désigne le sous-groupe de décomposition de $G^{(n)}$ en la place $v$ de $H_n$ définie par $\iota_p$.

\begin{lemme}\label{lem:selnprime}
	Pour tout $n\in\bN \cup \{\infty\}$, l'application de restriction $\Sel_n(\rho,\rho^+) \longrightarrow \Sel'_n(\rho,\rho^+)$ a un noyau et un conoyau finis. C'est un isomorphisme dès que $p\nmid \#G$.  
\end{lemme}

\begin{proof}
	On a un diagramme commutatif à lignes exactes 
	$$\xymatrix{ 0 \ar[r] & \Sel_n(\rho,\rho^+)  \ar[r] \ar[d]^\alpha & H^1(\bQ_n,D_p) \ar[r] \ar[d]^\res & 
		{\begin{array}{c}
			\Hom(I_p(\ob{\bQ}/\bQ_n),D_p^-)  \\ \times \prod_{\ell\neq p} H^1(I_\ell(\ob{\bQ}/\bQ_n),D_p) 
			\end{array}}
		\ar[d]^{\alpha'} \\ 
		0 \ar[r] & \Sel'_n(\rho,\rho^+) \ar[r] & \Hom_{G^{(n)}}(G_{H_n},D_p) \ar[r] & {\begin{array}{c} \Hom(I_p(\ob{\bQ}/H_n),D_p^-) \\ \times \prod_{\ell\neq p} H^1(I_\ell(\ob{\bQ}/H_n),D_p). \end{array}} }$$ 
	
	Pour $i=1,2$ on sait que $H^i(G^{(n)},D_p)$ est fini, et il est même trivial si $p\nmid \#G^{(n)}$. Donc il en est de même pour les noyaux et conoyaux de la flèche $\res$. Cela est aussi vrai pour $\ker \left(\alpha'\right)$, car $H_n/\bQ_n$ ne ramifie qu'en un nombre fini de places. Le lemme du serpent montre alors que $\alpha$ a un noyau et un conoyau fini, et que $\alpha$ est un isomorphisme lorsque $p$ ne divise pas l'ordre de $G$.
\end{proof}

\begin{lemme}\label{lem:selnsecond}
	Le conoyau de l'application $\kappa_n : \Hom_{G^{(n)}}(\gX_n,V_p) \longrightarrow \Hom_{G^{(n)}}(\gX_n,D_p)$ induite par la projection canonique $V_p \twoheadrightarrow D_p$ est fini.   
\end{lemme} 

\begin{proof}
	Notons $\left(\gX_n\right)_\textrm{div}$ le quotient de $\gX_n$ par son sous-groupe de torsion. Comme $\Hom_{G^{(n)}}(\left(\gX_n\right)_\textrm{div},V_p)=\Hom_{G^{(n)}}(\gX_n,V_p)$, on a des inclusions $$\im\left(\kappa_n\right) \subseteq \Hom_{G^{(n)}}(\left(\gX_n\right)_\textrm{div},D_p) \subseteq \Hom_{G^{(n)}}(\gX_n,D_p).$$ 
	La deuxième inclusion est d'indice fini car $\gX_n$ est de type fini sur $\Zp$, donc il suffit de montrer que la première inclusion est d'indice fini, \textit{i.e.}, le quotient $Q:=\Hom_{G^{(n)}}(\left(\gX_n\right)_\textrm{div},D_p)/\im(\kappa_n)$ est fini. Comme $\left(\gX_n\right)_\textrm{div}$ est libre sur $\Zp$, le foncteur $\Hom(\left(\gX_n\right)_\textrm{div},-)$ est exact d'où la suite exacte courte de $G^{(n)}$-modules
	$$\xymatrix{0 \ar[r] & \Hom(\gX_n,T_p) \ar[r] & \Hom(\gX_n,V_p) \ar[r] & \Hom(\left(\gX_n\right)_\textrm{div},D_p) \ar[r] & 0}.$$
	En prenant les $G^{(n)}$-invariants, on voit que $Q$ s'injecte dans $H^1(G^{(n)},\Hom(\gX_n,T_p))$, qui est fini car $\Hom(\gX_n,T_p)$ est de type fini.
\end{proof}

Considérons le diagramme commutatif suivant à lignes exactes :
$$\xymatrix{
	0 \ar[r] & \Hom_{G^{(n)}}(\gX_n,T_p) \ar[r] \ar[d] & \Hom_{G^{(n)}}(\gX_n,V_p) \ar[r]^{\kappa_n} \ar[d]^{\beta_n} & \Hom_{G^{(n)}}(\gX_n,D_p)  \ar[d]  \\
	0 \ar[r] & \Hom_{G^{(n)}_v}(I_p,T_p^-) \ar[r]  & \Hom_{G^{(n)}_v}(I_p,V_p^-) \ar[r]  & \Hom_{G^{(n)}_v}(I_p,D_p^-), 
}$$
où $I_p=I_p(M_n/H_n)$. Comme $I_p$ est compact, on a $\Hom_{G^{(n)}_v}(I_p,T_p^-) \otimes \Qp = \Hom_{G^{(n)}_v}(I_p,V_p^-)$, donc la flèche verticale gauche a un conoyau fini si $\beta_n$ est surjective.

\begin{corollaire}\label{coro:beta_n_isom?}
	Si $\beta_n$ est un isomorphisme, alors $\Sel_n(\rho,\rho^+)$ est fini.
\end{corollaire}

Le morphisme $\beta_n \otimes 1 : \Hom_{G^{(n)}}(\gX_n,\ob{\bQ}_p \otimes V_p) \longrightarrow \Hom_{G^{(n)}_v}({I_p},\ob{\bQ}_p\otimes V_p^-) $ est équivariant pour l'action de $G_n / G^{(n)} \simeq\Gamma_n$. On peut donc le décomposer en une somme d'applications $\beta_{n,\chi}$ (où $\chi$ parcourt $\widehat{\Gamma_n}=\Hom(\Gamma_n,\ob{\bQ}_p^\times)$) de la manière suivante. L'induite à $G_n$ de la représentation $\ob{\bQ}_p \otimes V_p$ restreinte à $G^{(n)}$ est donnée par
\begin{equation}\label{eq:decomposition_induite}
\Ind_{G^{(n)}}^{G_n}\ob{\bQ}_p \otimes V_p = \bigoplus_{\chi \in \widehat{\Gamma_n}} V_{p,\chi},
\end{equation}
où $V_{p,\chi} = V_p\otimes \ob{\bQ}_p(\chi)$ est le $\ob{\bQ}_p[G_n]$-module associé à $\rho$, tordu par un caractère $\chi \in\widehat{\Gamma_n}$. Notons que $V_{p,\chi}$ est irréductible car $V_p$ l'est. Ainsi, la loi de réciprocité de Frobenius donne un isomorphisme canonique 
$$\Hom_{G^{(n)}}(\gX_n,\ob{\bQ}_p \otimes V_p) \simeq \bigoplus_{\chi \in \widehat{\Gamma_n}} \Hom_{G_n}(\gX_n,V_{p,\chi}).$$ 

\begin{notation}\label{nota:beta_n_chi}
	Pour tout caractère $\chi\in\widehat{\Gamma_n}$, on pose $V_{p,\chi}^\pm:= V_p^\pm \otimes_{\ob{\bQ}_p} \ob{\bQ}_p(\chi) $, et on note $\beta_{n,\chi}$ l'application de restriction 
	$$ \beta_{n,\chi} : \Hom_{G_n}(\gX_n,V_{p,\chi}) \longrightarrow \Hom_{G_{n,v}}({I_p},V_{p,\chi}^-),$$
	où $G_{n,v}$ désigne le sous-groupe de décomposition de $G_{n}$ en la place $v$ de $H_n$ définie par $\iota_p$. 
\end{notation}

\begin{corollaire}\label{coro:beta_n_chi_isom?}
	Si $\beta_{n,\chi}$ est un isomorphisme pour tout $\chi \in \widehat{\Gamma_n}$, alors $\Sel_n(\rho,\rho^+)$ est fini.
\end{corollaire}

\subsection{Théorie du corps de classes}\label{sec:TCC}

On fixe désormais un entier $n\in\bN$. La théorie des corps de classes permet d'étudier la structure galoisienne du $\Zp$-module $\gX_n$. On a un isomorphisme de $\Zp$-algèbres 
\begin{equation}\label{eq:completes_p_adiques}
\cO_{H_n} \otimes \Zp \simeq \prod_{w|p} \cO_{H_{n,w}}
\end{equation}
donné par les plongements de $H$ dans ses complétés $p$-adiques. Par analogie avec le plongement logarithmique de Minkowski, on a des morphismes de $G_n$-modules à gauche : 

\begin{equation}\label{eq:plongement_log_p}
\begin{array}{rcccl} 
\cO_{H_n}^\times & \longrightarrow & \left(\cO_{H_n} \otimes_\bZ \Zp\right)^{\times} & \xrightarrow{\lambda_p} & \ob{\bQ}_p[G_n] \\ 
\epsilon & \mapsto & \epsilon \otimes 1 ; \quad x \otimes c & \mapsto & \sum_{g\in G_n} \log_p\left(g^{-1}(x) \otimes c\right).g,
\end{array}
\end{equation}
où $\log_p : \ob{\bQ}^\times \otimes \ob{\bQ}_p : $ est la composée du logarithme $p$-adique d'Iwasawa avec $\iota_p : \ob{\bQ} \hookrightarrow \ob{\bQ}_p$, que l'on étend scalairement.

\begin{notation}
	Pour un entier naturel $n$ et une place $w$ de $H_n$ au-dessus de $p$, on note :
	\begin{itemize}
		\item $U_{n,w}=\left\{ x \in \cO_{H_{n,w}}^\times \ /\ x-1 \notin \cO_{H_{n,w}}^\times \right\}$ le $\Zp$-module des unités locales principales de $H_{n,w}$.
		\item $U_n:= \prod_{w|p}U_{n,w}$ le produit des unités locales principales de $H_n$ au-dessus de $p$. 
		\item $\cE_{n}:= \cO_{H_n}^\times$ le groupe des unités de $H_n$, et $\cE_{p,n}=\Zp \otimes_\bZ \cE_n$ son complété $p$-adique formel. 
		\item $\cC_n$ est le $p$-Sylow du groupe des classes d'idéaux de $H_n$.
	\end{itemize}
\end{notation}

L'isomorphisme (\ref{eq:completes_p_adiques}) identifie $U_n$ avec un sous-$\Zp$-module d'indice fini de $\left(O_{H_n} \otimes_\bZ \Zp\right)^{\times}$. On note encore $\lambda_p$ la composée $U_n \longrightarrow \ob{\bQ}_p[G_n]$. Elle est de noyau fini, et envoie $U_{n,v}$ dans $\ob{\bQ}_p[G_{n,v}]$. Après extension des scalaires, on a un isomorphisme $\lambda_p \otimes 1 : \ob{\bQ}_p \otimes U_n \simeq \ob{\bQ}_p[G_n]$, et donc une décomposition en facteurs irréductibles de $\ob{\bQ}_p[G_n]$-modules :
$$\ob{\bQ}_pU_n \simeq \bigoplus_\pi V_{p,\pi}^{\dim V_{p,\pi}},$$
où $(\pi,V_{p,\pi})$ parcourt l'ensemble des (classes d'isomorphisme de) $\ob{\bQ}_p$-représentations irréductibles de $G_n$. La preuve de Minkowski du théorème des unités de Dirichlet donne d'autre part une décomposition similaire pour $\cE_{n}$ (ainsi que pour $\cE_{p,n}$), donnée par 
$$ \ob{\bQ} \otimes \cE_{n} \simeq \bigoplus_{\pi\neq 1} V_\pi^{d^+_\pi},$$
où $(\pi,V_{\pi})$ parcourt l'ensemble des $\ob{\bQ}$-représentations irréductibles de $G_n$ et où $d^+_\pi$ est la dimension de $V^{\Frob_\infty=1}_\pi$.

D'après la théorie des corps de classes (voir par exemple \cite[Chap X, Lemma 3.13]{neukirchcohomology}), on a une suite exacte de $\Zp[G_n]$-modules 

\begin{gather}
\begin{aligned}
\xymatrix{
	\cE_{p,n} \ar[r]^\iota & U_n \ar[r]^{\rec\ } & \gX_n \ar[r] & {\cC_n} \ar[r] & 0,
}
\end{aligned}
\label{ex1}
\end{gather} où $\rec$ est l'application de réciprocité d'Artin et où $\iota$ est l'extension $\Zp$-linéaire du plongement diagonal $\cE_{n} \hookrightarrow \left(O_{H_n} \otimes_\bZ \Zp\right)^{\times}$ décrit en (\ref{eq:plongement_log_p}). Par ailleurs, pour toute place $w$ de $H_n$ divisant $p$, on sait que $\rec$ envoie $U_{n,w}$ isomorphiquement sur le sous-groupe d'inertie de $\gX_n$ en $w$. En particulier, ${I_p}$ est l'image isomorphe de $U_{n,v}$. 

\begin{conjecture}[Conjecture de Leopoldt pour $H_n$ et $p$ \cite{leopoldt}]\label{conj:leopoldt}
	L'application $\iota$ est injective. 
\end{conjecture}

\begin{theorem}[Théorème de Baker-Brumer \cite{brumer}]\label{th:BB}
	Soient $\alpha_1, \cdots,\alpha_n \in \ob{\bQ}^\times$ des nombres algébriques non-nuls. Si les nombres $\log_p(\alpha_1), \cdots,\log_p(\alpha_n) \in \ob{\bQ}^\times_p$ sont linéairement indépendants sur $\bQ$, alors ils le sont encore sur $\ob{\bQ}$. 
\end{theorem}

Le Théorème \ref{th:BB} implique que l'extension $\ob{\bQ}$-linéaire du logarithme $p$-adique, 
\begin{equation}\label{eq:log_p_inj}
\log_p : \ob{\bQ} \otimes \cE_n \longrightarrow \ob{\bQ}_p,
\end{equation}
est injective, car elle envoie une base de $\cE_n$ modulo sa torsion sur une famille $\ob{\bQ}$-libre.

\begin{proposition}\label{prop:iota_non_nulle}
	Soit $\pi$ une $\ob{\bQ}_p$-représentation irréductible de $G_n$ apparaissant dans $\ob{\bQ}_p \otimes \cE_{p,n}$. Alors l'image par $\iota$ de la composante $\pi$-isotypique $\left(\ob{\bQ}_p \otimes \cE_{p,n}\right)^\pi$ de $\ob{\bQ}_p \otimes \cE_{p,n}$ dans $\ob{\bQ}_p \otimes U_n$ est non-nulle. 
\end{proposition}

\begin{proof}
	On a un diagramme commutatif :
	\[
	\xymatrix{ \ob{\bQ} \otimes \cE_{n} \ar[r] \ar@{^{(}-}[d]\ar@<1ex>@{-}[d] & \ob{\bQ}_p \otimes U_n \ar[r]^{\lambda_p\ }_{\simeq\ } &  \ob{\bQ}_p[G_n], \\ 
		\ob{\bQ}_p \otimes \cE_{p,n} \ar[ru]^{\iota} & & 
	}
	\]
	où la flèche en haut à gauche est l'extension linéaire à $\ob{\bQ}$ du plongement diagonal des unités globales dans les unités locales. D'une part, la composée des flèches de la ligne du haut est injective car l'application définie en (\ref{eq:log_p_inj}) est injective. D'autre part, $\pi$ descend en une $\ob{\bQ}$-représentation que l'on note encore $\pi$. Comme l'image de $\left(\ob{\bQ}\otimes \cE_{n}\right)^\pi$ dans $\ob{\bQ}_p \otimes U_n$ est non-nulle, l'image par $\iota$ de $\left(\ob{\bQ}_p\otimes \cE_{p,n}\right)^\pi$ est \textit{a fortiori} non-nulle.
\end{proof}

\begin{proposition}\label{prop:dim_source_but_beta} Supposons que $H^0(\bQ_n,V)=0$ et fixons $\chi \in \widehat{\Gamma_n}$.
\begin{enumerate}
	\item[(i)] La dimension de la source du morphisme $\beta_{n,\chi}$ (cf. Notation \ref{nota:beta_n_chi}) est supérieure ou égale à $d^-$, et la dimension du but de $\beta_{n,\chi}$ est égale à $d^-$.
	\item[(ii)] Si $d^+=1$, ou bien si la Conjecture de Leopoldt pour $H_n$ et $p$ est vraie, alors les dimensions de la source et du but de $\beta_{n,\chi}$ sont toutes deux égales à $d^-$.
\end{enumerate}
\end{proposition}

\begin{proof}
	On commence par vérifier que $V_{p,\chi}$ est non-trivial. En effet, la loi de réciprocité de Frobenius montre que $H^0(G_n,\Ind_{G^{(n)}}^{G_n}V_p)=H^0(G^{(n)},V_p)=H^0(\bQ_n,V)=0$, donc $H^0(G_n,V_{p,\chi})=0$ d'après l'égalité (\ref{eq:decomposition_induite}). 
	 
	Considérons la source $\Hom_{G_n}(\gX_n,V_{p,\chi})$ de $\beta_{n,\chi}$. D'après le lemme de Schur, sa dimension est égale à la multiplicité de $\rho\otimes\chi$ dans $\ob{\bQ}_p \otimes \gX_n$. En tensorisant avec $\ob{\bQ}_p$ la suite exacte courte (\ref{ex1}) (ce qui tue le groupe des classes $\cC_n$) et en prenant la partie $\rho\otimes\chi$-isotypique, on obtient une autre suite exacte
	$$\xymatrix{
		\left(\ob{\bQ}_p \otimes \cE_{p,n}\right)^{\rho\otimes\chi} \ar[r] & \left(\ob{\bQ}_p \otimes U_n\right)^{\rho\otimes\chi} \ar[r] & \left(\ob{\bQ}_p \otimes \gX_n\right)^{\rho\otimes\chi} \ar[r] & 0.
	}$$ 
	 Comme $\chi$ est pair et que la représentation irréductible $\rho\otimes\chi$ est non-triviale, elle apparaît avec multiplicité $d^+$ dans l'espace $\ob{\bQ}_p \otimes \cE_{p,n}$ des unités globales. D'autre part, $\rho\otimes\chi$ apparaît avec multiplicité $d$ dans l'espace des unités locales. Donc la dimension de la source de $\beta_{n,\chi}$ est supérieure ou égale à $d^-$, avec égalité si la restriction de $\iota$ à $\left(\ob{\bQ}_p\otimes \cE_{p,n}\right)^{\rho\otimes\chi}$ est injective. Ceci est automatiquement vérifié si l'on suppose la conjecture de Leopoldt pour $H_n$ et $p$. Ceci est aussi vrai si l'on suppose que $d^+=1$ d'après la Proposition \ref{prop:iota_non_nulle}, car $\iota$ est équivariante et car $\left(\ob{\bQ}_p\otimes \cE_{p,n}\right)^{\rho\otimes\chi}$ est irréductible (et isomorphe à $V_{p,\chi}$). 
	 	
	Calculons la dimension du but $\Hom_{G_{n,v}}({I_p},V_{p,\chi}^-)$ de $\beta_{n,\chi}$. L'application de réciprocité d'Artin envoie isomorphiquement $U_{n,v} \subseteq U_n$ sur ${I_p} \subseteq \gX_n$. D'autre part, $\lambda_p$ identifie $\ob{\bQ}_p \otimes U_{n,v}$ avec la représentation régulière $\ob{\bQ}_p[G_{n,v}]$ de $G_{n,v}$. On a donc $\Hom_{G_{n,v}}(\ob{\bQ}_p[G_{n,v}],V_{p,\chi}^-) \simeq V_{p,\chi}^-$, qui est bien de dimension égale à $d^-$. 
\end{proof}

\subsection{Injectivité de $\beta_{n,\chi}$}\label{sec:BB}

On suppose dans cette section que $H^0(\bQ_n,V)=0$. Rappelons que, d'après la preuve de la Proposition \ref{prop:dim_source_but_beta}, $V_{p,\chi}$ est non-trivial. Grâce à l'application de réciprocité d'Artin (cf. (\ref{ex1})), le noyau de $\beta_{n,\chi}$ (cf. Notation \ref{nota:beta_n_chi}) s'identifie à 
\begin{equation}\label{eq:ker_beta_n_chi}
\begin{array}{rcl}
\ker \beta_{n,\chi}&\simeq&\ker\left[\Hom_{G_n}(\gX_n,V_{p,\chi}) \mathrel{\scalebox{1}[1.2]{$\hookrightarrow$}} \Hom_{G_n}(U_n,V_{p,\chi}) \twoheadrightarrow \Hom_{G_{n,v}}(U_{n,v},V_{p,\chi}^-)  \right] \\
&\simeq& \ker\left[ \Hom_{G_n}(U_n,V_{p,\chi}) \longrightarrow \Hom_{G_n}(\cE_{p,n},V_{p,\chi}) \times \Hom_{G_{n,v}}(U_{n,v},V_{p,\chi}^-) \right].
\end{array}
\end{equation}

Le $\ob{\bQ}_p$-espace vectoriel $\Hom_{\textrm{ct}}(U_n,\ob{\bQ}_p)$ est muni d'une action à gauche de $G_n$ donnée par $(g.f)(x\otimes c) = f(g^{-1}(x)\otimes c)$, pour $g\in G_n$ et $x\otimes c \in U_n$. On a un isomorphisme $G_n$-équivariant
\begin{equation}\label{isomQbar}
\left\{ \begin{array}{rcl}
\ob{\bQ}_p[G_n] & \longrightarrow & \Hom_\textrm{ct}(U_n,\ob{\bQ}_p) \\
g          & \mapsto         & \left( x\otimes c \mapsto \log_p\left(g^{-1}(x)\otimes c \right) \right).
\end{array} \right.
\end{equation}

\begin{lemme}\label{lem:description_morphismes_Qpbar}
	Soit $f \in \Hom_{G_n}(U_n,V_{p,\chi})$. Alors $f$ s'écrit sous la forme 
	$$f(x\otimes c) = \sum_{g \in G_n} \log_p \left(g^{-1}(x)\otimes c \right)\left(\rho\otimes\chi\right)(g)(\vec{v}) = \lambda_p(x\otimes c) \cdot \vec{v}, $$ 
	pour un unique vecteur $\vec{v}\in V_{p,\chi}$, et la restriction à $U_{n,v}$ est donnée par 
	$$f(u) = \sum_{g\in G_{n,v}}  \log_p\left(g^{-1}(u)\right)\left(\rho\otimes\chi\right)(g)(\vec{v}) = \lambda_p(u)\cdot \vec{v}. $$
	De plus, $f(U_{n,v}) \subseteq V_{p,\chi}^+$ si et seulement si $\vec{v}\in V_{p,\chi}^+$.
\end{lemme}
\begin{proof}
	On a $\Hom_{G_n}(U_n,V_{p,\chi}) = \left(\Hom_{\textrm{ct}}(U_n,\ob{\bQ}_p) \otimes V_{p,\chi} \right)^{G_n}$, donc d'après l'isomorphisme (\ref{isomQbar}) il existe des vecteurs $\vec{v}_g \in V_{p,\chi}$ tels que 
	$$f(x\otimes c) = \sum_{g \in G_n} \log_p \left(g^{-1}(x)\otimes c \right)\vec{v}_g.$$
	En utilisant la $G_n$-invariance, on a facilement $\vec{v}_g = \left(\rho\otimes\chi\right)(g)(\vec{v_1})$. 
	On a donc bien $f(x\otimes c) = \lambda_p(x\otimes c) \cdot \vec{v}$ avec $\vec{v}=\vec{v_1}$. 
	L'unicité de $\vec{v}$ est claire, car l'image de $\lambda_p$ engendre $\ob{\bQ}_p[G_n]$. 
	La description de $f$ sur $U_{n,v}$ se déduit de celle de $\lambda_p$. 
	Enfin, comme $\lambda_p(\ob{\bQ}_p U_{n,v})=\ob{\bQ}_p[G_{n,v}]$, on a aussi $f(\ob{\bQ}_p U_{n,v}) = \ob{\bQ}_p[G_{n,v}] \cdot \vec{v}$. 
	Or, $V_{p,\chi}^+$ est $G_{n,v}$-stable, donc $f(U_{n,v}) \subseteq V_{p,\chi}^+$ si et seulement si $\vec{v}\in V_{p,\chi}^+$.
\end{proof}

Dans la suite de cette section, on note $V_\chi = V \otimes_E \ob{\bQ}(\chi)$ le $\ob{\bQ}$-espace vectoriel muni de l'action de $G_n$ via $\rho\otimes \chi$, ainsi que $V^+_\chi=V^+ \otimes_E \ob{\bQ}(\chi)$ la $p$-stabilisation ordinaire (étendue linéairement à $\ob{\bQ}$) du motif $[\rho\otimes \chi]$, héritée de celle du motif $[\rho]$. On fixe une $\ob{\bQ}$-base de $V^+_\chi$ que l'on complète en une base $(\vec{v}_1,\ldots,\vec{v}_d)$ de $V_\chi$, ainsi qu'une $\ob{\bQ}$-base $(\Phi_1,\ldots,\Phi_{d^+})$ de $\Hom_{G_n}\left(V_{\chi}, \ob{\bQ} \otimes \cE_n \right)$ (cf. Section \ref{sec:TCC}). Ces choix définissent des éléments $\epsilon_{i,j}=\Phi_i(\vec{v}_j)$ (où $1\leq i \leq d^+, 1\leq j \leq d$) qui forment une base de $\left(\ob{\bQ}\otimes \cE_n \right)^{\rho\otimes\chi}$. On pose $s_{i,j}:=\log_p(\epsilon_{i,j})$.

\begin{lemme}\label{lem:calcul_matriciel}
	Fixons $1\leq i\leq d^+$. Pour tout $g\in G_n$, on a une égalité de matrices lignes à $d$ entrées : 
	$$\begin{pmatrix} \log_p\left(g(\epsilon_{i,1})\right) & \cdots &\log_p\left(g(\epsilon_{i,d})\right)\end{pmatrix}=\begin{pmatrix}
	s_{i,1} & \cdots &  s_{i,d}
	\end{pmatrix} \left(\rho\otimes\chi\right)(g).$$
\end{lemme}
\begin{proof}
	Soit $g\in G_n$. Notons $a_{j,k}$ les coefficients de $\rho\otimes\chi(g)$ dans la base $(\vec{v}_1,\ldots,\vec{v}_d)$. Pour tout $1\leq j\leq d$, on a $g(\epsilon_{i,j})=\epsilon_{i,1}^{a_{1,j}} \ldots \epsilon_{i,d}^{a_{d,j}}$. En prenant le logarithme $p$-adique et en concaténant les égalités pour $1\leq j\leq d$, on retrouve la formule souhaitée.
\end{proof}

\begin{proposition}\label{prop:S^+}
	L'application $\beta_{n,\chi}$ est injective si et seulement si la matrice carrée $S^+:=(s_{i,j})_{1\leq i,j\leq d^+} \in M_{d^+}(\ob{\bQ}_p)$ est inversible. 
\end{proposition}

\begin{proof}
	Soit $f \in \Hom_{G_n}(U_n,V_{p,\chi})$. D'après le Lemme \ref{lem:description_morphismes_Qpbar}, $f$ est de la forme $f(x\otimes c) = \lambda_p(x\otimes c) \cdot \vec{v} $ pour un unique vecteur $\vec{v}\in V_{p,\chi}$. La base fixée de $V_\chi$ définit une base de $V_{p,\chi}$ et on note $(v_1,\ldots,v_d)$ les coordonnées de $\vec{v}$ dans cette base. En étendant $f$ à $\ob{\bQ}_p \otimes U_n$ par linéarité, $f$ sera nulle sur (l'image par $\iota$ de) $\cE_{p,n}$ si et seulement si $f$ est nulle sur $\left(\ob{\bQ}\otimes \cE_n \right)^{\rho\otimes\chi}$. Fixons un entier $1\leq i\leq d^+$. D'après le Lemme \ref{lem:calcul_matriciel}, on a $ f(\epsilon_{i,j})=0$ pour tout $1\leq j\leq d$ si et seulement si la matrice 
	$$M_i:= \sum_{g \in G_n} \rho\otimes \chi(g) \begin{pmatrix} v_1 \\ \vdots \\ v_d \end{pmatrix} \begin{pmatrix}
	s_{i,1} & \ldots & s_{i,d} \end{pmatrix} \rho\otimes\chi(g)^{-1}
	$$
	est nulle. 
	On observe que $M_i$ commute avec $\rho\otimes\chi(G_n)$. Comme $\rho\otimes\chi$ est absolument irréductible, $M_i$ est scalaire, disons $M_i=\lambda_i I_d$. En prenant la trace matricielle, on obtient $\lambda_i=\frac{\#G_n}{d} \left(s_{i,1}v_1+\ldots + s_{i,d}v_d\right)$. Ainsi, grâce à (\ref{eq:ker_beta_n_chi}) et au Lemme \ref{lem:description_morphismes_Qpbar}, on obtient les équivalences : 
	$$\begin{array}{rcl}
	f\in \ker \beta_{n,\chi} &\Longleftrightarrow & \left\{\begin{array}{l}
	f(U_{n,v}) \subseteq V_{p,\chi}^+ \\ f(\cE_{p,n})=0 
	\end{array}\right. \\
	&\Longleftrightarrow & \left\{\begin{array}{l}
	\vec{v}\in V_{p,\chi}^+ \\ \forall 1 \leq i \leq d^+, \ M_i=0 
	\end{array}\right. \\
	&\Longleftrightarrow & \left\{\begin{array}{l}
	v_{d^++1}=\ldots=v_d=0 \\ \forall 1 \leq i \leq d^+, \ s_{i,1}v_1+\ldots + s_{i,d^+}v_{d^+}=0 
	\end{array}\right. \\
	&\Longleftrightarrow & \left\{\begin{array}{l}
	v_{d^++1}=\ldots=v_d=0 \\ S^+ \cdot\vec{v}^+=0 ,
	\end{array}\right.
	\end{array}
	$$
	où $\vec{v}^+$ est le vecteur colonne de coordonnées $(v_1,\ldots,v_{d^+})$. Cela prouve que $\beta_{n,\chi}$ est injective si et seulement si la matrice $S^+$ est inversible. 
\end{proof}

La matrice $S^+$ est toujours inversible, si l'on croit à l'analogue $p$-adique de la conjecture de Schanuel (cf. \cite[Conjecture 3.10]{mazurcalegari}). 

\begin{conjecture}[Conjecture de Schanuel $p$-adique faible]\label{conj:Schanuel}
	Soient $\alpha_1,\ldots,\alpha_n$ des nombres algébriques non-nuls, vus dans $\ob{\bQ}_p^\times$ via le plongement $\iota_p$. Si $\log_p \alpha_1,\ldots,\log_p \alpha_n$ sont $\bQ$-linéairement indépendants, alors le corps $\bQ(\log_p \alpha_1,\ldots,\log_p \alpha_n)$ a un degré de transcendance sur $\bQ$ égal à $n$.
\end{conjecture}

\begin{lemme}\label{lem:transcendance}
	Soit un corps $K \subseteq \ob{\bQ} \subseteq_{\iota_p} \ob{\bQ}_p$. Si la Conjecture de Schanuel $p$-adique faible est vraie, alors pour tout $\epsilon_1,\ldots,\epsilon_n \in \cO_K^\times \otimes \ob{\bQ}$ tels que les $\log_p \epsilon_1,\ldots,\log_p \epsilon_n$ sont $\ob{\bQ}$-linéairement indépendants, et pour tout polynôme $P\in \ob{\bQ}[X_1,\ldots,X_n]-\{0\}$, on a 
	$$P(\log_p \epsilon_1,\ldots,\log_p \epsilon_n)\neq 0.$$
\end{lemme}
\begin{proof}
	La Conjecture \ref{conj:Schanuel} implique que le corps $\bQ(\log_p \epsilon_1,\ldots,\log_p \epsilon_n)$ a un degré de transcendance sur $\bQ$ égal à $n$, et l'assertion du lemme en découle directement.
\end{proof}

\begin{corollaire}\label{coro:beta_n_chi_inj}
	Si $d^+=1$, ou si l'on suppose la Conjecture de Schanuel $p$-adique vraie, alors $\beta_{n,\chi}$ est injective, et en particulier, sa source est de dimension $d^-$.
\end{corollaire}
\begin{proof}
	Il suffit de montrer que la matrice $S^+$ est inversible, d'après la Proposition \ref{prop:S^+}. Or cela résulte de l'injectivité de l'application donnée en (\ref{eq:log_p_inj}) lorsque $d^+=1$, et du Lemme \ref{lem:transcendance} en général.
\end{proof}
Rappelons que l'on a supposé dans cette section que $H^0(\bQ_n,V)=0$. On a achevé de démontrer le théorème suivant.
\begin{theorem}[=Théorème \ref{th:sel_n_fini_etc} (i)] \label{th:sel_n_fini}
	 Si $d^+=1$, ou bien si l'on suppose la conjecture de Schanuel $p$-adique faible, alors $\Sel_n(\rho,\rho^+)$ est fini.
\end{theorem}
\begin{proof}
	D'après le Corollaire \ref{coro:beta_n_chi_isom?}, il suffit de montrer que $\beta_{n,\chi}$ est un isomorphisme pour tout $n$ et tout $\chi \in \widehat{\Gamma_n}$. Cela résulte immédiatement de la combinaison du Corollaire \ref{coro:beta_n_chi_inj} et de la Proposition \ref{prop:dim_source_but_beta} (i).
\end{proof}

\subsection{Structure du groupe de Selmer et phénomène des zéros triviaux}\label{sec:zerostriviaux}

On suppose dans cette section que $H^0(\bQ_\infty,V)=0$ ainsi que $d^+=1$, ou bien que la conjecture de Schanuel $p$-adique faible est vraie (cf. Conjecture \ref{conj:Schanuel}). Pour tout entier naturel $n$, le dual de Pontryagin de $X_\infty(\rho,\rho^+)_{\Gamma^{p^n}}$ s'identifie au $\cO[[\Gamma_n]]$-module $\Sel_\infty(\rho,\rho^+)^{\Gamma^{p^n}}$, que l'on va étudier. Comme $H \cap \bQ_\infty=\bQ_{n_0}$, le groupe de Galois $\Gamma':=\Gal(H_\infty/H)$ s'identifie au sous-groupe $\Gamma^{p^{n_0}}$ de $\Gamma$ par restriction des automorphismes. 

Soit $n\geq n_0$, et soit $m=n-n_0$. La flèche de restriction $H^1(\bQ_\infty,D_p) \longrightarrow \Hom_{G^{(\infty)}}(G_{H_\infty},D_p)$ est équivariante pour l'action de $\Gamma^{p^n}\simeq \Gamma^{'p^m}$, et son noyau et conoyau sont finis. Comme $\Gamma$ est cyclique, on montre facilement que $H^1(\bQ_\infty,D_p)^{\Gamma^{p^n}} \longrightarrow \Hom_{G^{(\infty)}}(G_{H_\infty},D_p)^{\Gamma^{'p^m}}$
a aussi un noyau et conoyau finis. Un argument similaire à celui de la preuve du Lemme \ref{lem:selnprime} montre que ceci est encore vrai pour la restriction aux groupes de Selmer. On en déduit le lemme suivant :
\begin{lemme} \label{lem:sel->sel'}
	L'application naturelle de restriction 
	$$\Sel_\infty(\rho,\rho^+)^{\Gamma^{p^n}} \longrightarrow \Sel'_\infty(\rho,\rho^+)^{\Gamma^{'p^m}}$$
	a un noyau et conoyau finis pour tout $m=n-n_0\geq 0$. C'est même un isomorphisme si $p$ ne divise pas $[H:\bQ]$ (auquel cas on a $n_0=0$).
\end{lemme}

D'après le Lemme \ref{lem:sel->sel'}, on est amenés à étudier les sous-modules invariants de $\Sel'_\infty(\rho,\rho^+)$. Rappelons que l'on a, par définition,
$$\Sel'_m(\rho,\rho^+) = \ker \left[ \xymatrix{\Hom_{G^{(m)}}(\gX_m,D_p) \ar[r]^{r_m \qquad} &  \Hom_{G^{(m)}_v}(I_p(M_m/H_m),D_p^-) } \right]. $$ 
Nous étudions le noyau et conoyau des \textit{applications de contrôle},
$$ \Sel'_m(\rho,\rho^+) \longrightarrow \Sel'_\infty(\rho,\rho^+)^{\Gamma^{'p^m}},$$
qui sont induites par l'application de restriction sur les groupes de Galois 
$$\gX_\infty =\Gal(M_\infty/H_\infty) \longrightarrow \gX_m=\Gal(M_m/H_m).$$
On a un diagramme commutatif à colonnes exactes :

$$\xymatrix{
	0 \ar[d] 
	& 0 \ar[d] 
	\\
	\Hom_{G^{(m)}}(\Gal(H_\infty/H_m),D_p) \ar[d] \ar[r]^{t_m} 
	& \Hom_{G^{(m)}_v}(I_p(H_\infty/H_m),D_p^-) \ar[d] 
	\\
	\Hom_{G^{(m)}}(\gX_m,D_p) \ar[d]^{s_m} \ar[r]^{r_m \qquad} 
	& \Hom_{G^{(m)}_v}(I_p(M_m/H_m),D_p^-) \ar[d]
	\\ 
	\Hom_{G^{(\infty)}}(\gX_\infty,D_p)^{\Gamma^{'p^m}} \ar[r]      
	& \Hom_{G^{(\infty)}_v}(I_p(M_m/H_\infty),D_p^-).
}$$
Notons que la source de $t_m$ est un $\cO$-module fini car on a $H^0(\bQ_m,T_p)=H^0(\bQ_\infty,T_p)=0$ par hypothèse. Notons aussi que $G^{(m)}=G^{(\infty)}$ dès que $m\geq n_0$, ce que l'on suppose dans la suite. Un argument élémentaire de théorie de Galois montre que l'application $(\gX_\infty)_{\Gamma^{'p^m}} \longrightarrow \gX_m$ est injective, et donc $s_m$ est a un conoyau fini. Le lemme du serpent donne une suite exacte :
\begin{gather}
\begin{aligned}
\xymatrix{
	 \Sel'_m(\rho,\rho^+) \ar[r] & \Sel'_\infty(\rho,\rho^+)^{\Gamma^{'p^m}} \bigcap \im\left(s_m\right) \ar[r] & \coker \left(t_m\right) \ar[r] & \coker\left(r_m\right).} 
\end{aligned}
\label{eq:grossesuite_exacte_sel'_m}
\end{gather}

\begin{proposition}\label{prop:rang_coinvariants}
	Supposons $n\geq 2n_0$. Le $\cO$-rang de $X_\infty(\rho,\rho^+)_{\Gamma^{p^n}}$ est égal à 
	$$e(\rho,\rho^+):=\dim H^0(\bQ_{\infty,v},V_p^-).$$
\end{proposition}

\begin{proof}
	Fixons $n\geq 2n_0$ et posons comme précédemment $m=n-n_0$. D'après le Lemme \ref{lem:sel->sel'}, le $\cO$-rang de $X_\infty(\rho,\rho^+)_{\Gamma^{p^n}}$ est égal au $\cO$-corang de $\Sel'_\infty(\rho,\rho^+)^{\Gamma^{'p^m}}$. D'après la suite exacte (\ref{eq:grossesuite_exacte_sel'_m}), il est égal au $\cO$-corang de $\coker\left(t_m\right)$, car le conoyau de $s_m$ est fini, ainsi que $\Sel'_m(\rho,\rho^+)=\ker\left(r_m\right)$ et $\coker\left(r_m\right)$. Cela découle en effet du Théorème \ref{th:sel_n_fini} et d'une adaptation de la Proposition \ref{prop:dim_source_but_beta} montrant que la source et le but de $r_m$ ont même corang. Comme la source de $t_m$ est de corang 0, le conoyau de $t_m$ a le même corang que $\Hom_{G^{(m)}_v}(I_p(H_\infty/H_m),D_p^-)$, c'est-à-dire $e(\rho,\rho^+)$.
\end{proof}

\begin{corollaire}[=Théorème \ref{th:sel_n_fini_etc} (ii) et (iii)]\label{coro:sel_infty_torsion+zéros_triviaux}
	Le $\Lambda$-module $X_\infty(\rho,\rho^+)$ est de torsion. Si $H\cap \bQ_\infty =\bQ$, alors la fonction $L$ $p$-adique algébrique $L_p^{\alg}(\rho,\rho^+;T) \in \Lambda$ de $X_\infty(\rho,\rho^+)$ ne s'annule pas en les $\zeta-1$, où $\zeta$ parcourt $\mu_{p^\infty}-\{1\}$. On a de plus $L_p^{\alg}(\rho,\rho^+;0) \neq 0 $ si et seulement si $e(\rho,\rho^+)=0$. Enfin, l'ordre d'annulation de $L_p^{\alg}(\rho,\rho^+;T)$ en $T=0$ est minoré par $e(\rho,\rho^+)$ :
	$$\ord_{T} L_p^{\alg}(\rho,\rho^+;T) \geq e(\rho,\rho^+).$$
\end{corollaire}

\begin{proof}
	D'après la Proposition \ref{prop:rang_coinvariants}, le $\cO$-rang de $X_\infty(\rho,\rho^+)_{\Gamma^{p^n}}$ est égal à $e(\rho,\rho^+)$ pour tout entier $n\geq 2n_0$. Soit $r$ le rang de $X_\infty(\rho,\rho^+)$ sur $\Lambda$. D'après le théorème de structure des $\Lambda$-modules de type fini (Théorème \ref{th:structure_sur_algebre_iwasawa}), il existe des polynômes $P_i(T)$ de $\Lambda$, des entiers $\mu_j>0$ et un pseudo-isomorphisme
	$$ X_\infty(\rho,\rho^+) \longrightarrow \Lambda^r \bigoplus \left( \bigoplus_i \Lambda / (P_i) \right) \bigoplus \left( \bigoplus_j \Lambda / (p^{\mu_j}) \right). $$
	Un calcul élémentaire sur les $\Gamma^{p^n}$-coinvariants fournit, pour tout entier $n\geq 2 n_0$, les égalités : 
	$$(\mathcal{EQ}_n) : \qquad e(\rho,\rho^+) = rp^n + \sum_i \deg_T \pgcd(P_i(T),\omega_n(T)),$$
	où $\omega_n(T):=(1+T)^{p^n}-1\in \Lambda$. En particulier, on obtient $r=0$, c'est-à-dire $X_\infty(\rho,\rho^+)$  est de torsion sur $\Lambda$. Sa fonction $L$ $p$-adique algébrique $L_p^{\alg}(\rho,\rho^+;T)$ peut être choisie comme étant égale à $p^{\sum_j \mu_j} \prod_i P_i(T)$. 
	
	Supposons maintenant de plus que $H\cap \bQ_\infty =\bQ$, c'est-à-dire $n_0=0$. Remarquons que dans ce cas, $e(\rho,\rho^+)=\dim H^0(\bQ_p,V^-)$, donc la définition de $e(\rho,\rho^+)$ est cohérente avec celle de l'introduction. Comme $\omega_0(T)=T$, l'équation $(\mathcal{EQ}_0)$ montre que $L_p^{\alg}(\rho,\rho^+;0)\neq 0$ si et seulement si $e(\rho,\rho^+)=0$, et en outre que $L_p^{\alg}(\rho,\rho^+;T)$ s'annule en 0 à un ordre $\geq e(\rho,\rho^+)$. Enfin, pour $n\geq 0$ quelconque, en retranchant $(\mathcal{EQ}_0)$ à $(\mathcal{EQ}_n)$ on voit que $L_p^{\alg}(\rho,\rho^+;T)$ est premier au polynôme $\frac{\omega_n(T)}{T}$, dont les racines sont exactement les $\zeta-1$, où $\zeta$ parcourt $\mu_{p^n}-\{1\}$.  Autrement dit, $L_p^{\alg}(\rho,\rho^+;\zeta-1) \neq 0$ pour tout $\zeta \in \mu_{p^\infty}-\{1\}$.
\end{proof}

Nous proposons une conjecture de type "zéros triviaux", imitant du côté algébrique les phénomènes prédits du côté analytique par le théorème de Ferrero-Greenberg \cite{ferrerogreenberg}, la Conjecture de Gross-Stark \cite{gross1981padic}, ou la Conjecture des zéros triviaux de \cite{benois}.

\begin{conjecture}\label{conj:zeros_triviaux}
	Avec les notations du Corollaire \ref{coro:sel_infty_torsion+zéros_triviaux}, si $H\cap \bQ_\infty =\bQ$, alors on a 
	$$\ord_{T} L_p^{\alg}(\rho,\rho^+;T) = e(\rho,\rho^+).$$
\end{conjecture}

Nous achevons la preuve du Théorème \ref{th:sel_n_fini_etc} en prouvant l'absence de sous-modules pseudo-nuls non-nuls de $X_\infty(\rho,\rho^+)$. Il s'agit d'une application du critère de la Proposition \ref{PsN}, rendue possible car on sait maintenant que $X_\infty(\rho,\rho^+)$ est de torsion.

\begin{proposition}\label{prop:sel_infty_sans_sous_modules_finis}
	Supposons que $D_p[\varpi]^G=\Hom_G(D_p[\varpi],\mu_p)=0$. Alors le $\Lambda$-module $X_\infty(\rho,\rho^+)$ n'a pas de sous-modules pseudo-nuls non-triviaux.
\end{proposition}

\begin{proof}
	On va vérifier les hypothèses de la Proposition \ref{PsN}, avec $\cA=\cO$ et $\cT=T_p$. Le point (a) a été montré dans le Corollaire \ref{coro:sel_infty_torsion+zéros_triviaux}. Montrons que $H^2(\bQ_\Sigma/\bQ_\infty,D_p)$ est de cotorsion sur $\Lambda$. D'après la suite spectrale de Hochschild-Serre (cf. \cite[Chap. 2, §4, Theorem 1]{neukirchcohomology}), le noyau de l'application de restriction $H^2(\bQ_\Sigma/\bQ_\infty,D_p) \longrightarrow H^2(\bQ_\Sigma/H_\infty,D_p)$ est contrôlé par la cohomologie du groupe fini $G^{(\infty)}$ à valeurs un $\cO$-module co-finiment engendré. Ce noyau est donc fini, et il suffit de montrer que $H^2(\bQ_\Sigma/H_\infty,D_p)$ est de cotorsion. Comme $\bQ_\Sigma=H_\Sigma$, on peut utiliser la formule de caractéristique d'Euler-Poincaré prouvée dans \cite[Proposition 4.1]{greenberg2006structure} et donnant (après application du lemme de Shapiro) :
	$$\corg_\Lambda \left( H^2(H_\Sigma/H_\infty,D_p)\right) = \corg_\Lambda  \left(H^1(H_\Sigma/H_\infty,D_p)\right) - \corg_\Lambda  \left(H^0(H_\Sigma/H_\infty,D_p)\right) -d \cdot r_2,$$
	où $r_2$ est le nombre de places complexes de $H$. Le $\Lambda$-corang du $H^0$ est 0, car il est égal à $D_p\simeq (L/\cO)^d$, qui est de co-torsion sur $\Lambda$. Il faut donc montrer que le corang du $H^1$ est égal à $d \cdot r_2$. Comme $\bQ_\Sigma$ est aussi la plus grande extension de $H_\infty$ non-ramifiée en dehors de $\Sigma$, et que l'action de $G_{H_\infty}$ sur $D_p$ est triviale, on a $H^1(\bQ_\Sigma/H_\infty,D_p)=\Hom_\textrm{ct}(\gX_{\infty,\Sigma},D_p)$ où $\gX_{\infty,\Sigma}$ est le groupe de Galois de la plus grande pro-$p$ extension abélienne $M_\Sigma$ de $H_\infty$ non-ramifiée en dehors de $\Sigma$. Or, $\gX_{\infty,\Sigma}$ est pseudo-isomorphe à $\gX_{\infty,\{p\}}=\gX_\infty$. En effet, on a une suite exacte 
	$$\xymatrix{\bigoplus_{\lambda|\ell\in\Sigma-\{p\}}I_\lambda(M_{\Sigma}/H_\infty)\ar[r] & \gX_{\infty,\Sigma}\ar[r] &\gX_{\infty} \ar[r] & 0,
	}
	$$
     où la somme porte sur toutes les places finies $\lambda$ de $H_\infty$ divisant une place $\ell\in\Sigma-\{p\}$. Or, cette somme directe est un $\cO$-module fini, car la somme est finie (puisqu'aucune place finie $\ell$ n'est totalement décomposée dans l'extension cyclotomique) et car chacun des groupes d'inertie est de torsion d'après \cite[Theorem 25]{iwasawa1973zl}. En particulier, $\gX_{\infty,\Sigma}$ a le même rang sur $\Lambda$ que $\gX_\infty$, qui est égal à $r_2$ (cf. \cite[Theorem 17]{iwasawa1973zl} ou \cite[Theorem 13.31]{washington1997introduction}).
	 Ainsi, $\Hom_\textrm{ct}(\gX_{\infty,\Sigma},D_p)^\vee \simeq \gX_{\infty,\Sigma}^{\oplus d}$ a pour rang $d \cdot r_2$, ce qu'on voulait montrer. L'hypothèse (b) est donc vérifiée. Les hypothèses (c) et (d) font partie de la définition de la $p$-stabilisation ordinaire. Enfin, l'hypothèse (e) est vérifiée d'après nos hypothèses sur la représentation résiduelle de $\rho$. On peut conclure que $X_\infty(\rho,\rho^+)$ n'a pas de sous-modules pseudo-nuls non-nuls.
\end{proof}

	\section{Terme constant de la fonction L p-adique algébrique}
\subsection{Cadre et énoncé de la formule du terme constant}\label{subsec:hyp_f_alpha}

On conserve les notations de la Section \ref{sec:différente_inverse} et nous supposons en outre que : 
\begin{itemize}
	\item[\textbf{(nr)}] $\rho$ est non-ramifiée en $p$,
	\item[\textbf{(deg)}] $p$ ne divise par l'ordre de l'image de $\rho$, \textit{i.e.}, $p\nmid [H:\bQ]$,
	\item[\textbf{(dim)}] $p$ ne divise pas la dimension de $\rho$, \textit{i.e.}, $p\nmid d$.
\end{itemize}
Cela entraîne que :
\begin{itemize}
	\item $\rho$ est (quitte à conjuguer) à valeurs dans l'anneau des entiers $\cO$ d'une extension \textit{non-ramifiée}\footnote{Si $g=\#G$ est premier à $p$, alors on peut en effet prendre $L=\Qp(\mu_g)$ d'après un théorème de Brauer (cf. \cite[§12.3, Théorème 24 \& Corollaire]{serrerepresentations})} $L$ de $\Qp$,
	\item la réduction modulo $p$ réalise un isomorphisme $\rho \simeq \ob{\rho}$, donc $\rho$ est résiduellement irréductible (et $X_\infty(\rho,\rho^+)$ ne dépend pas du choix du réseau $T_p$),
	\item et l'idempotent $e_\rho\in L[G]$ associé à $\rho$ est à coefficients dans $\cO$.
\end{itemize}

Le complété $H_v \subseteq \ob{\bQ}_p$ de $H$ en $v$ est une extension non-ramifiée de $\Qp$. Quitte à adjoindre à $E$ des racines de l'unité d'ordre premier à $p$, on peut supposer que $L$ contient $H_v$ et $\mu_{g}$, où $g=\#G$. Comme $\cO$ contient $\left(\#G\right)^{-1}$ et les racines de l'unité d'ordre $\#G$, l'algèbre $\cO[G]$ est semi-simple et admet une décomposition
$$\cO[G] = \bigoplus_\pi e_\pi \cO[G]\simeq \bigoplus_\pi T_{p,\pi}^{\dim \pi},$$
où $(\pi,V_{p,\pi})$ parcourt les représentations irréductibles de $G$ et où $T_{p,\pi}$ est un réseau $G$-stable de $V_{p,\pi}$ (unique à homothétie près). Pour tout $\cO[G]$-module $M$, on a ainsi une décomposition en somme directe $M = \bigoplus_{\pi} e_{\pi}M$. On note $M^\rho =e_\rho(M\otimes_{\Zp} \cO)$  la composante $\rho$-isotypique d'un $\Zp[G]$-module $M$.

Choisissons une base $\Phi_1,\ldots,\Phi_{d^+}$ du $\cO$-module libre $\Hom_G(T_p,\cO_H^\times \otimes_\bZ \cO)$, ainsi qu'une base de $T_p^+=T_p\cap V_p^+$, que l'on complète en une base $\vec{t}_1,\ldots,\vec{t}_d$ de $T_p$. Comme $H_v/\Qp$ est non-ramifiée, on a $\log_p(\cO_{H_v}) \subseteq p\cO_{H_v}$, et donc $\log_p : \cO_H^\times \otimes_\bZ \cO \longrightarrow L$ est à valeurs dans $p\cO$. 
\begin{definition}
	On définit le régulateur $p$-adique de $[\rho]$, muni de sa $p$-stabilisation $V^+$, comme étant le déterminant de taille $d^+\times d^+$ défini par
	$$R_p(\rho,\rho^+)= \det\left(\frac{\log_p\left(\Phi_i(\vec{t}_j)\right)}{p}\right)_{1\leq i,j\leq d^+} \in \cO.$$
\end{definition}

\begin{lemme}\label{lem:non_annulation_reg}
	$R_p(\rho,\rho^+)$ est défini à une unité de $\cO$ près, et il est non-nul si la Conjecture de Schanuel $p$-adique faible est vraie (cf. Conjecture \ref{conj:Schanuel}).
\end{lemme}
\begin{proof}
	La définition de $R_p(\rho,\rho^+)$ dépend du choix de deux $\cO$-bases, donc $R_p(\rho,\rho^+)$ est défini à multiplication par un élément de $\det (\GL_{d^+}(\cO))\simeq \cO^\times$ près. Par ailleurs, ces deux bases induisent des $L$-bases respectives de $\Hom_G(V_p,\cO_H^\times \otimes_\bZ L)$ et de $V_p$ (adaptée à $V_p^+$). Pour montrer que $R_p(\rho,\rho^+)\neq 0$, il suffit de montrer que le même déterminant ne s'annule pas pour un choix spécifique de ces $L$-bases. Or $V_p=V \otimes_E L$ et $V_p^+=V^+ \otimes_E L$, donc on peut en particulier considérer des $E$-bases respectives de $\Hom_G(V,\cO_H^\times \otimes_\bZ E)$ et de $V$ (adaptée à $V^+$). Si la Conjecture de Schanuel $p$-adique est vraie, alors ce déterminant est non-nul par application du Lemme \ref{lem:transcendance} de la même manière que dans le Corollaire \ref{coro:beta_n_chi_inj}, et donc  $R_p(\rho,\rho^+)\neq 0$.
\end{proof}

\begin{theorem}\label{th:terme_cst_dim_d}
	Supposons $d^+=1$, ou bien supposons que la conjecture de Schanuel $p$-adique faible est vraie (cf. Conjecture \ref{conj:Schanuel}). Si $e(\rho,\rho^+)=0$, alors $X_\infty(\rho,\rho^+)$ est de torsion sur $\Lambda$, et le terme constant de sa fonction $L$ algébrique est donné par 
	$$L_p^{\alg}(\rho,\rho^+;0) = R_p(\rho,\rho^+) \cdot \sqrt[d]{\#\left(\cC^\rho\right)}$$
	à une unité de $\cO$ près.
\end{theorem}

\subsection{Preuve de la formule du terme constant}
\label{sec:preuve_terme_cst}
On donne à présent la preuve du Théorème \ref{th:terme_cst_dim_d}. On suppose désormais $d^+=1$, ou bien que la conjecture de Schanuel $p$-adique faible est vraie. L'action de $G_{\Qp}$ sur $V$ et sur $T_p$ se factorise par celle de $\Frob_v$, qui est diagonalisable car d'ordre fini. On peut supposer que la base $\vec{t}_1,\ldots,\vec{t}_d$ de $T_p$ diagonalise $\Frob_v$. Comme en Section \ref{sec:BB}, on note $\epsilon_{i,j}:=\Phi_i(\vec{t}_j)$, $s_{i,j}:=\log_p(\epsilon_{i,j})$ et $S^+=(s_{i,j})_{1\leq i,j\leq d^+}\in M_{d^+}(p\cO)$. On a ainsi $R_p(\rho,\rho^+)=\dfrac{\det S^+}{p^{d^+}}$.
\begin{lemme}\label{lem:terme-cst}
	Le $\cO$-module $\Sel_0(\rho,\rho^+)$ est d'ordre fini. Si $e(\rho,\rho^+)=0$, alors $X_\infty(\rho,\rho^+)$ est de torsion sur $\Lambda$, et le terme constant de sa fonction $L$ algébrique est donné par 
	$$L_p^{\alg}(\rho,\rho^+;0) = \#\Sel_0(\rho,\rho^+)$$
	à une unité de $\cO$ près.  
\end{lemme}
\begin{proof}
	Le fait que $\Sel_0(\rho,\rho^+)$  soit d'ordre fini sous nos hypothèses a été prouvé dans le Théorème \ref{th:sel_n_fini}. Par ailleurs, sous nos hypothèses, la suite exacte (\ref{eq:grossesuite_exacte_sel'_m}) (pour $m=0$) peut être complétée et simplifiée en une suite exacte courte :
	$$\xymatrix{ 0 \ar[r] & \Sel_0(\rho,\rho^+) \ar[r] & \Sel_\infty(\rho,\rho^+)^{\Gamma} \ar[r] & \left(D_p^-\right)^{e(\rho,\rho^+)} \ar[r] &  0.}$$
	En effet, $s_0$ est surjective car $p\nmid \#G$, et la source de $t_0$ est triviale car $H^0(G,D_p)=0$. Cela découle en effet du lemme de Nakayama et du fait que $H^0(G,\ob{\rho})=0$ car $\rho\simeq\ob{\rho}$. D'autre part, les hypothèses entraînent que $\Sel_\infty(\rho,\rho^+)^{\Gamma}$ et son dual $X_\infty(\rho,\rho^+)_{\Gamma}$ sont d'ordre fini (de même cardinal), et le lemme de Nakayama montre que $X_\infty(\rho,\rho^+)$ est de torsion sur $\Lambda$. On montre aussi en dupliquant la preuve du Théorème \ref{th:sel_n_fini_etc}.(iv) que ce dernier n'a pas de sous-$\Lambda$-modules finis non-nuls. D'après le Lemme \ref{cst}, on a donc 
	$$L_p^{\alg}(\rho,\rho^+;0) = \# \Sel_\infty(\rho,\rho^+)^{\Gamma} = \#\Sel_0(\rho,\rho^+).$$
\end{proof}
La théorie des corps de classes et la Proposition \ref{prop:iota_non_nulle} montrent que l'on a, sous nos hypothèses, une suite exacte de $\cO[G]$-modules
$$\xymatrix{0 \ar[r] & \cE_{p}^\rho \ar[r]^\iota & U^\rho \ar[r]^\rec & \gX^\rho \ar[r] &  \cC^\rho \ar[r] & 0,}$$
où $\cE_p=\cE_{p,0}$, $U=U_0$, $\gX=\gX_0$ et $\cC=\cC_0$. On a un diagramme commutatif :
$$
\xymatrix{
	&&0 \ar[d] && \\
	&&\Sel_0(\rho,\rho^+) \ar[d] &&\\
	0 \ar[r] & \Hom_G(\cC,D_p)  \ar[r] & \Hom_G(\gX,D_p) \ar[r] \ar[d] & \Hom_G(U/\cE_p,D_p) \ar[r] \ar[d] & 0 \\
	&& \Hom_{G_v}(I_p(M/H),D_p^-) \ar[r]^{\quad \simeq} & \Hom_{G_v}(U_{v},D_p^-),  & 
}
$$
où la ligne et colonne centrales sont exactes car $H^1(G,\Hom(\cC,D_p))=0$ grâce à l'hypothèse $p\nmid\#G$. Le lemme du serpent donne en particulier une autre suite exacte de $\cO$-modules finis
\begin{align} \label{3selmer}
0 \longrightarrow \Hom_G(\cC,D_p)  \longrightarrow \Sel_0(\rho,\rho^+)  \longrightarrow  \Sel^\# \longrightarrow  0 
\end{align} 
où l'on a posé 
\[
\Sel^\# := \ker\left[\Hom_G(U/\cE_p,D_p) \longrightarrow \Hom_{G_v}(U_{v},D_p^-) \right]
\]

Comme $\mu_p \not\subseteq H_v$, le $\Zp$-module $U$ est libre, et donc $\Hom_G(U,V_p) \longrightarrow \Hom_G(U,D_p)$ est surjective. On en déduit une autre description de $\Sel^\#$ : 
\[
\Sel^\# \simeq \{f\in \Hom_G(U,V_p) \ /\ f(\cE_p) \subseteq T_p, \ f(U_v)\subseteq V_p^+ + T_p\}/\Hom_G(U,T_p)
\]

\begin{proposition}\label{prop:cyc}
	Le $\cO$-module $\Sel^\#$ est isomorphe au quotient $T_p^+/p^{-1}S^+\left(T_p^+\right)$.
\end{proposition}
\begin{proof}
	Le choix de la base de $T_p$ (adaptée à $T_p^+$ et diagonalisant $\Frob_v$) identifie $T_p^-$ avec un supplémentaire de $T_p^+$ dans $T_p$, \textit{i.e.}, $T_p=T_p^+\bigoplus T_p^-$. On a de même $V_p=V_p^+ \bigoplus V_p^-$, ainsi que $D_p=D_p^+ \bigoplus D_p^-$. 
	On note $(f^+,f^-)$ les coordonnées d'un morphisme $f$ à valeurs dans $V_p$ relativement à cette décomposition. Si $f\in \Hom_G(U,V_p)$, alors le Lemme \ref{lem:description_morphismes_Qpbar} permet d'écrire $f$ sous la forme 
	$$f(x\otimes c) = \sum_{g \in G} \log_p \left(g^{-1}(x)\otimes c \right)\rho(g)(\vec{v}) = \lambda_p(x\otimes c) \cdot \vec{v}, $$ 
	pour un unique vecteur $\vec{v}\in V_p$, et pour tout $x\otimes c \in U \subseteq \left(\cO_H \otimes \Zp\right)^\times$. La restriction à $U_{v}$ est donnée par 
	\begin{equation}\label{eq:expression_f_locale}
	f(u) = \sum_{g\in G_{v}}  \log_p\left(g^{-1}(u)\right)\rho(g)(\vec{v}) = \lambda_p(u)\cdot \vec{v}. 
	\end{equation}
	
	\begin{lemme}
		\'Ecrivons $\vec{v}=\vec{v}^+ \oplus \vec{v}^-$ selon la décomposition $V_p=V_p^+\bigoplus V_p^-$. On a $f^\pm(U_v) \subseteq T_p^\pm$ si et seulement si $p \vec{v}^\pm \in T_p^\pm$. En particulier, on a $f(U)\subseteq T_p$ si et seulement si $p\vec{v}\in T_p$.
	\end{lemme}
	\begin{proof}[Preuve du Lemme]
		 Pour $1\leq i \leq d$, notons $f_i:U_v \longrightarrow L$ la $i$-ème coordonnée de $f$, notons $v_i \in L$ celle de $\vec{v}$ et notons $\zeta_i\in L$ la $i$-ème valeur propre de $\Frob_v$ dans la base ambiante $\{\vec{t}_i\}_i$. Il nous suffit de montrer que, pour tout $1\leq i \leq d$, on a $f_i(U_v)\subseteq \cO$ si et seulement si $pv_i\in \cO$. Fixons $1\leq i \leq d$. L'expression (\ref{eq:expression_f_locale}) pour $f$ montre que l'on a $f_i(u)=\sigma_i\circ \log_p(u) v_i$ pour $u\in U_v=1+p\cO_{H_v}$, où $\sigma_i$ est l'application $\Zp$-linéaire $$\sigma_i :  \left\{\begin{array}{ccl} H_v & \longrightarrow & L \\
		  x&\mapsto & \sum_{j=0}^{\#G_v-1}\Frob^{j}_v(x)\zeta_i^{-j}. 
		  \end{array}\right.$$ 
		 Notons que $\sigma_i$ est effectivement à valeurs dans $L$ car on a supposé que $H_v\subseteq L$. Comme $H_v/\Qp$ est non-ramifiée, le logarithme $p$-adique envoie bijectivement $U_v$ sur $p\cO_{H_v}$ et donc $f_i(U_v)=pv_i\sigma_i\left(\cO_{H_v}\right)$. Pour montrer l'équivalence, il suffit donc de montrer que $\sigma_i\left(\cO_{H_v}\right)\subseteq \cO$ mais que $\sigma_i\left(\cO_{H_v}\right) \not\subseteq p\cO$. Autrement dit, il suffit de voir que $\sigma_i\left(\cO_{H_v}\right)\subseteq \cO$ et qu'il existe $x_0\in\cO_{H_v}$ tel que $\sigma_i(x_0)\in\cO^\times$. La première affirmation est claire car $\zeta_i$ est une racine de l'unité (puisque $\Frob_v\in G$ est d'ordre fini) et donc $\zeta_i\in \cO$. Pour la deuxième, on réduit $\sigma_i$ modulo $p$ (qui est une uniformisante de $H_v$ et de $L$) pour obtenir 
		 $$\ob{\sigma}_i :  \left\{\begin{array}{ccl} \bF_{H_v} & \longrightarrow & \bF_{L} \\
		 \bar{x}&\mapsto & \sum_{j=0}^{\#G_v-1}\zeta_i^{-j}\bar{x}^{p^j}, 
		 \end{array}\right.$$ 
		 où $\bF_{H_v}$ et $\bF_{L}$ sont les corps résiduels respectifs de $H_n$ et de $L$. On doit vérifier que $\ob{\sigma}_i$ n'est pas identiquement nulle. Or, ceci est vrai pour toute application polynômiale sur $\bF_{H_v}$ de degré strictement inférieur à $\#\left(\bF_{H_v}\right)=p^{\#G_v}$, et donc c'est vrai en particulier pour $\ob{\sigma}_i$.
	\end{proof}
	On termine la preuve de la Proposition \ref{prop:cyc}. Soit $f$ comme précédemment. D'après le Lemme \ref{lem:calcul_matriciel} et la Proposition \ref{prop:S^+}, et avec les notations de la preuve de la Proposition \ref{prop:S^+}, on a $f(\cE_p)\subseteq T_p$ si et seulement si les matrices $M_i$ de taille $d\times d$ sont à coefficients dans $\cO$ pour tout $1\leq i \leq d^+$. Or, $\rho$ étant absolument irréductible, $M_i$ est la matrice d'une homothétie de rapport $\lambda_i=\frac{\#G}{d} \left(s_{i,1}v_1+\ldots + s_{i,d}v_d\right)$, où $(v_1,\cdots,v_d)$ sont les coordonnées de $\vec{v}$ dans la base fixée de $T$. Comme $p\nmid \#G \cdot d$ d'après \textbf{(deg)} et \textbf{(dim)}, et comme $s_{i,j}\in p\cO$, on a les équivalences : 
	
	$$\begin{array}{rcl}
	\left\{\begin{array}{l}
	f(U_{v}) \subseteq V_p^+ +T_p \\ f(\cE_p)\subseteq T_p 
	\end{array}\right. &\Longleftrightarrow & \left\{\begin{array}{l}
	f^-(U_{v}) \subseteq T_p^- \\ \forall 1\leq i \leq d^+, \  M_i\in M_d(\cO) 
	\end{array}\right. \\
	&\Longleftrightarrow & \left\{\begin{array}{l}
	p\vec{v}^-\in T_p^- \\ \forall 1 \leq i \leq d^+, \ s_{i,1}v_1+\ldots + s_{i,d}v_{d}\in \cO
	\end{array}\right. \\
	&\Longleftrightarrow & \left\{\begin{array}{l}
	pv_{d^++1},\cdots,pv_d\in \cO \\ \forall 1 \leq i \leq d^+, \ s_{i,1}v_1+\ldots + s_{i,d^+}v_{d^+}\in \cO 
	\end{array}\right. \\
	&\Longleftrightarrow & \left\{\begin{array}{l}
	pv_{d^++1},\cdots,pv_d\in \cO \\ S^+ \cdot\vec{v}^+ \in T_p^+ ,
	\end{array}\right. \\
	&\Longleftrightarrow & \left\{\begin{array}{l}
	\vec{v}^-\in p^{-1}T_p^- \\ \vec{v}^+ \in \left(S^+\right)^{-1}(T_p^+).
	\end{array}\right. 
	\end{array}
	$$
	Ainsi, le $\cO$-module $\Sel^\#$ s'identifie à $\left(\left(S^+\right)^{-1}(T_p^+) \times p^{-1}T_p^-\right) /\left(p^{-1}T_p^+ \times p^{-1}T_p^-\right) \simeq T_p^+/p^{-1}S^+\left(T_p^+\right)$.
	
\end{proof} 
\begin{proof}[Preuve du Théorème \ref{th:terme_cst_dim_d}]
	D'après le Lemme \ref{lem:terme-cst}, il suffit de prouver que l'ordre de $\Sel_0(\rho,\rho^+)$ est égal à $R_p(\rho,\rho^+) \cdot \sqrt[d]{\#\left(\cC^\rho\right)}$ à une unité $p$-adique près. La suite exacte (\ref{3selmer}) montre que cet ordre est égal au produit des cardinaux de $\Hom_G(\cC,D_p)$ et de $\Sel^\#$. La Proposition \ref{prop:cyc} montre d'une part que $\Sel^\#$ est d'ordre $p^{-d^+}{\det\left(S^+\right)}=R_p(\rho,\rho^+)$. D'autre part, les $\Zp$-modules finis sont auto-duaux pour la dualité de Pontryagin (il suffit de le vérifier pour les $p$-groupes cycliques) donc l'ordre du $\Hom_G(\cC,D_p)$ est le même que son dual de Pontryagin. Or, pour un $\cO[G]$-module $M$ de type fini, $\Hom_G(M,D_p)^\vee$ s'identifie (non-canoniquement) à $\Hom_G(T_p,M)$. Cela est vrai en effet pour $M=\cO[G]$ car $D_p^\vee\simeq T_p$, et on en déduit le cas général en prenant une présentation finie de $M$. Ainsi, on doit prouver que $\Hom_G(T_p,\cC)$ est d'ordre $\sqrt[d]{\#\left(\cC^\rho\right)}$. Comme $\cO[G] = \bigoplus_\pi T_{p,\pi}^{\dim\pi}$, où la somme parcourt les $\cO$-représentations irréductibles $(\pi,T_{p,\pi})$ de $G$, on a de plus par orthogonalité des idempotents : 
	$$\Hom_G(T_p,\cC)^{\oplus d} =\Hom_G(T_p,\cC^\rho)^{\oplus d} = \Hom_G(T_p^{\oplus d},\cC^\rho) = \Hom_G(\cO[G],\cC^\rho) = \cC^\rho, $$
	d'où l'égalité souhaitée.
\end{proof}

	\section{Théorie d'Iwasawa des formes modulaires de poids 1}
 \subsection{Groupe de Selmer d'une forme modulaire classique de poids 1}

Soit $(\rho,V)$ une représentation d'Artin de dimension 2 et irréductible sur $\bQ$, de conducteur d'Artin égal à $N$ et à coefficients dans un corps de nombres $E$. On suppose que $\rho$ est impaire c'est-à-dire que $d^+=1$. On sait associer à $\rho$ une forme modulaire parabolique primitive de poids 1 $f$ grâce au travail de Khare et Wintenberger (\cite{KW}[Corollary 10.2.(ii)]). La forme $f$ est de niveau $N$, a pour caractère $\epsilon=\det \rho$, et son $q$-développement $f(z)=\sum_{n\geq 1} a_nq^n \in E[[q]]$ à la pointe $\infty$ satisfait 
$$\Tr \rho(\Frob_\ell) = a_\ell,$$
pour tout nombre premier $\ell$ ne divisant pas $N$. Soient $\alpha$ et $\beta$ les racines du $p$-ième polynôme de Hecke $X^2-a_p X+\epsilon(p)\in E[X]$ de $f$. Nous supposerons dans toute la suite que : 
\begin{itemize}
	\item[\textbf{(rég)}] $f$ est régulière en $p$, \textit{i.e.}, $\alpha\neq \beta$,  
	\item[\textbf{(nr)}] $\rho$ est non-ramifiée en $p$, \textit{i.e.}, $p\nmid N$,
	\item[\textbf{(deg)}] $p$ ne divise par l'ordre de l'image de $\rho$, \textit{i.e.}, $p\nmid [H:\bQ]$,
\end{itemize}

Le complété $H_v \subseteq \ob{\bQ}_p$ de $H$ en la place privilégiée $v|p$ est une extension non-ramifiée de $\Qp$. On fixe l'anneau $\cO$ des entiers d'une extension non-ramifiée $L$ de $\Qp$ contenant $H_v$ et $\mu_{\#G}(\ob{\bQ}_p)$. Les valeurs propres $\alpha,\beta\in \mu_{\#G}(L)$ de $\rho(\Frob_v)$ sont distinctes, donc $V$ admet exactement deux droites $G_v$-stables et supplémentaires sur lesquelles $\Frob_v$ agit respectivement par multiplication par $\alpha$ et $\beta$. Ainsi, le choix d'une $p$-stabilisation ordinaire de $[\rho]$ revient au choix d'une de ces deux droites stables, c'est-à-dire au choix d'une $p$-stabilisation de $f$, à savoir $f_\alpha(z)=f(z)-\beta f(pz)$ ou bien $f_\beta(z)=f(z)-\alpha f(pz)$.

\begin{definition}\label{def:sel_f_alpha}
	Soit $f_\alpha$ la $p$-stabilisation de $f$ de valeur propre $\alpha$ pour $U_p$. Le groupe de Selmer attaché à $f_\alpha$ est le $\cO$-module 
	$$\Sel_\infty(f_\alpha) := \Sel_\infty(\rho,\rho^+),$$
	où l'on a choisi la $p$-stabilisation ordinaire $V^+\subseteq V$ de $\rho$ comme étant égale à la droite sur laquelle $\Frob_v$ agit par multiplication par $\beta$. On note aussi 
	$$X_\infty(f_\alpha) := \Sel_\infty(f_\alpha)^\vee$$
	le groupe de Selmer dual de $f_\alpha$ et $L_p^{\alg}(f_\alpha;T)$ la fonction $L$ $p$-adique algébrique de $f_\alpha$, c'est-à-dire un générateur de l'idéal caractéristique de $X_\infty(f_\alpha)$.
\end{definition}

\begin{corollaire}\label{prop:rappel_prop_sel_n_f_alpha}
	
	\begin{enumerate}
		\item La définition de $\Sel_\infty(f_\alpha)$ ne dépend pas du choix du réseau $G_{\bQ}$-stable $T_p$ de $V_p$.
		\item $X_\infty(f_\alpha)$ est un module de torsion sur $\Lambda=\cO[[T]]$. Il n'a, de plus, aucun sous-module fini différent de $\{0\}$.
		\item On a 
		$$L_p^{\alg}(f_\alpha;0)= \left\{\begin{array}{rcl}
		\frac{\log_p\left(\epsilon_{f_\alpha}\right)}{p} \sqrt{\#\cC^{\rho}} & \mbox{si}& \alpha\neq 1 \\
		0 & \mbox{si}& \alpha=1
		\end{array}\right.$$
		à une unité de $\cO$ près. Ici, $\cC^{\rho}$ est la composante $\rho$-isotypique du $p$-groupe des classes d'idéaux de $H$, et $\epsilon_{f_\alpha}$ désigne un générateur du sous-$\cO$-module (libre de rang 1) de $(\cO_H^\times \otimes_{\bZ} \cO)^{\rho}$ sur lequel $\Frob_v$ agit par multiplication par $\beta$. De plus, $\log_p\left(\epsilon_{f_\alpha}\right)$ est non-nul.
	\end{enumerate}
\end{corollaire}

\begin{proof}	
	L'hypothèse \textbf{(deg)} implique que $\rho$ est résiduellement irréductible, donc $\Sel_\infty(f_\alpha)$ ne dépend pas du choix de $T_p$ d'après la Proposition \ref{prop:dépendance_réseau}. Par ailleurs, la droite $V_p^-$ est celle sur laquelle $\Frob_v$ agit par multiplication par $\alpha$, donc $e(\rho,\rho^+)=\dim H^0(\Qp,V_p^-)=0$ si et seulement si $\alpha\neq 1$. Les points (2) et (3) découlent ainsi des Théorèmes \ref{th:sel_n_fini_etc} et \ref{th:terme_cst_dim_d}.
\end{proof}

\subsection{Exemple des formes modulaires à multiplication complexe} \label{sec:CM}
Soit $\psi : G_F \longrightarrow E^\times$ un caractère d'ordre fini sur un corps quadratique imaginaire $F$. On note $\tau$ le caractère non-trivial de $\Gal(F/\bQ)$ (qui est aussi la restriction de $\Frob_\infty$ à $F$) et $\psi_\tau : G_F \longrightarrow E^\times$ le caractère donné par $\psi_\tau(h)=\psi(\tau h \tau)$ pour tout $h\in G_F$.  Si $\psi \neq \psi_\tau$, alors la série thêta associée à $\psi$, vu comme un caractère de Hecke, est une forme parabolique $f$ de poids 1 dont la représentation de Deligne-Serre est donnée par $\rho=\Ind_F^{\bQ} \psi$. On suppose que $\rho$ satisfait les hypothèses \textbf{(rég)}, \textbf{(nr)}, \textbf{(deg)} et on reprend les notations de la section précédente. On a ainsi une décomposition de $G_F$-modules 
\begin{equation}\label{eq:decomposition_D_p_CM}
D_p=\Qp/\Zp \otimes \cO(\psi) \bigoplus \Qp/\Zp \otimes \cO(\psi_\tau), \end{equation}
où $\cO(\psi)$ (resp. $\cO(\psi_\tau)$) est le $\cO$-module libre de rang 1 sur lequel $G_F$ agit via $\psi$ (resp. via $\psi_\tau$). Lorsque $p$ est décomposé dans $F$, on a $\psi(\Frob_v)=\psi(\gp)$ et $\psi_\tau(\Frob_v)=\psi(\ob{\gp})$, où $p\cO_F=\gp\ob{\gp}$ et où $\gp$ est l'idéal premier de $\cO_F$ au-dessus de $p$ déterminé par $v$ (\textit{i.e.}, par $\iota_p$). On a en particulier $\{\alpha,\beta\}=\{\psi(\gp),\psi(\ob{\gp})\}$. 
\begin{proposition}\label{prop:sel_CM}
	Supposons que $p$ est décomposé dans $F$ et choisissons $\alpha=\psi(\ob{\gp})$. Soit $K/F$ l'extension galoisienne découpée par $\psi$ et soit $\gX_{\gp}$ le groupe de Galois de la pro-$p$ extension abélienne maximale de $K\bQ_\infty$ non-ramifiée en dehors de $\gp$. Alors $X_\infty(f_\alpha)$ est isomorphe à la composante $\psi$-isotypique $\left(\gX_\gp \otimes \cO\right)^{(\psi)}$ de $\gX_\gp$.
\end{proposition}

\begin{proof}
	Notons $D_\psi$ et $D_{\psi_\tau}$ les deux facteurs de $D_p$ de la décomposition (\ref{eq:decomposition_D_p_CM}). Les hypothèses sur $\rho$ impliquent que $\psi(\gp)$ et $\psi(\ob{\gp})$ sont distincts modulo $p$, donc $D_\psi$ est le plus grand sous-module de $D_p$ sur lequel $\Frob_v$ agit par multiplication par $\psi(\gp)$. Comme $\psi(\gp)=\beta$, on a ainsi $D_p^+=D_\psi$. D'autre part, le lemme de Shapiro donne un isomorphisme 
	\[H^1(\bQ_\infty,D_p) \simeq H^1(F\bQ_\infty,D_\psi), \]
	qui envoie une classe de cocycles $\sigma$ à valeurs dans $D_p$ sur la classe $\tau:=\textrm{pr}_\psi\circ \sigma_{|G_{F\bQ_\infty}}$, où $\textrm{pr}_\psi$ est la projection $D_p \twoheadrightarrow D_\psi$. Comparons les conditions locales définissant $\Sel_\infty(f_\alpha)$ dans $H^1(\bQ_\infty,D_p)$ : on a $\sigma$ non-ramifié en $\lambda|l\neq p$ si et seulement si $\tau$ est non-ramifié en $\lambda$. Comme $D_p^+=D_\psi$, on vérifie de même que $\sigma(I_p)\subseteq D_p^+$ si et seulement si $\tau$ est non-ramifié en $\ob{\gp}$. Ainsi, l'image de $\Sel_{\infty}(f_\alpha)$ par l'isomorphisme de Shapiro, composé avec l'application de restriction $H^1(F\bQ_\infty,D_\psi)\simeq \Hom_{\Gal(K/F)}(G_{K\bQ_\infty},D_\psi)$ (qui est un isomorphisme car $p\nmid [K:F]$ par le même argument qu'en Section \ref{sec:isom_de_restriction}), est égale à $\Hom_{\Gal(K/F)}(\gX_\gp,D_\psi)$, c'est-à-dire à $\Hom\left(\left(\gX_\gp \otimes \cO\right)^{(\psi)},\Qp/\Zp \otimes \cO\right)$. En prenant les duaux de Pontryagin, on obtient le résultat voulu.
\end{proof}
\begin{remarque}
	En utilisant la Proposition \ref{prop:sel_CM}, on voit que les résultats du Corollaire \ref{prop:rappel_prop_sel_n_f_alpha} ont déjà été prouvés dans \cite[Theorem 5.3 (i)-(iii)-(v)]{rubin1991main} pour les formes modulaires à multiplication complexe par un corps quadratique imaginaire dans lequel $p$ est décomposé. 
\end{remarque}

\subsection{Familles de Hida et déformations ordinaires}
\label{sec:familles_de_Hida}

\subsubsection{Spécialisations}

On pose $u=1+p$ un générateur du groupe multiplicatif $1+p\Zp$. Soit $\bH$ une extension finie de $\Zp[[X]]$. On dit qu'un morphisme d'anneaux $\phi:\bH \longrightarrow \ob{\bQ}_p$ est une \textit{spécialisation classique} de $\bH$ s'il existe un entier $k\geq 1$ et une racine de l'unité $\zeta \in \mu_{p^\infty}$ primitive d'ordre $p^{r-1}$ (où $r>0$) tels que $\phi(X)=\zeta u^{k-1} -1$. On dit que $\phi$ est une \textit{spécialisation arithmétique}\footnote{C'est la terminologie utilisée dans des références standards, notamment dans \cite{ghatevatsal}.} si, de plus, $k\geq 2$. On pose $k_\phi:=k$, $r_\phi:=r$, et aussi $\chi_\phi= \chi_\zeta$. La spécialisation $\phi$ définit par ailleurs un idéal $\gp_\phi:=\ker \phi$ de $\bH$, qui est premier de hauteur 1, et un anneau $\cO_{\phi}:= \im \phi$ qui est une extension finie de $\Zp$. 

\subsubsection{Familles de Hida} 
Une forme parabolique ordinaire $\Zp[[X]]$-adique de niveau modéré $N$ et de caractère $\epsilon$ est un $q$-développement formel $\textbf{f}=\sum_{n\geq 0} a_n(\textbf{f})q^n$ à coefficients dans une extension finie $\bH$ de $\Zp[[X]]$ telle que, pour toute spécialisation arithmétique $\phi$ de $\bH$, le $q$-développement
$$g_\phi:=\phi(\textbf{f})= \sum_n \phi(a_n(\textbf{f}))q^n \in S_{k_\phi}(Np^{r_\phi},\epsilon \chi_{\phi} \omega^{1-k_\phi},\cO_{\phi})$$
définit une forme parabolique ordinaire $p$-stabilisée de poids $k_\phi$ et de niveau $Np^{r_\phi}$. On dit que $\textbf{f}$ est $N$-nouvelle si toutes ses spécialisations arithmétiques le sont. Une famille de Hida primitive $\textbf{f}$ est une forme parabolique ordinaire $\Zp[[X]]$-adique $N$-nouvelle qui est propre pour les opérateurs de Hecke $U_\ell$ (resp. $T_\ell$, $\langle \ell \rangle$) pour $\ell|Np$ (resp. $\ell\nmid Np$).

Soit $\bH_N$ l'algèbre de Hecke universelle ordinaire de niveau modéré $N$. C'est une $\Zp[[X]]$-algèbre générée par les opérateurs de Hecke agissant sur l'espace des formes paraboliques ordinaires. Elle est libre de rang fini d'après \cite[Theorem 3.1]{hida1986iwasawa}, et ses spécialisations arithmétiques correspondent aux formes paraboliques propres de niveau modéré $N$. Le quotient $\bH_N^\textrm{new}$ de $\bH_N$ agissant fidèlement sur l'espace des formes $N$-nouvelles est de même une $\Zp[[X]]$-algèbre finie réduite et sans torsion. Ses spécialisations arithmétiques sont de plus en bijection avec les (orbites galoisiennes de) formes propres classiques de poids $k\geq 2$ et niveau modéré $N$ qui sont nouvelles en $N$. 

Le théorème de contrôle de Hida montre que toute forme primitive $p$-ordinaire $p$-stabilisée de poids $\geq 2$ est la spécialisation d'une famille de Hida primitive. Ce résultat a été étendu aux formes de poids 1 par Wiles (\cite[Theorem 3]{wiles1988ordinary}).

\begin{theorem}[Wiles]
	Il existe une famille de Hida primitive $\textbf{f}$ se spécialisant en $f_\alpha$. 
\end{theorem}

Comme $f$ est $p$-régulière, on sait même que $\textbf{f}$ est unique\footnote{Il existe des exemples explicites de formes modulaires primitives de poids 1 non-régulières en $p$ et en lesquelles se spécialisent deux familles de Hida non-conjuguées, cf. \cite[§7.4]{dimitrovghate}.} à conjugaison galoisienne près, d'après le théorème principal de \cite{bellaiche2016eigencurve} (voir aussi \cite[Corollary 1.15]{dimitrov2014local}). La famille $\textbf{f}$ définit un morphisme d'anneaux $\bH_N^\textrm{new} \longrightarrow \ob{\Frac(\Zp[[X]])}$, dont le noyau $\ga$ est un idéal premier minimal de $\bH_N^\textrm{new}$. On peut alors voir $\textbf{f}$ à coefficients dans $\bH_\textbf{f}:=\bH^\textrm{new}_N/\ga$. On note $\phi_{\alpha} : \bH_\textbf{f} \longrightarrow \ob{\bQ}_p$ la spécialisation classique de poids 1 associée à $f_\alpha$, et on note simplement $\gp_{\alpha}=\ker{\phi_{\alpha}}$ l'idéal premier de hauteur 1 associé à $f_\alpha$. Les coefficients de Fourier de $f_\alpha$ sont dans $\cO$, donc $\phi_{\alpha}$ est à valeurs dans $\cO$.
\begin{remarque}\label{p-old}
	Soit $\phi$ une spécialisation arithmétique de poids $k$ de $\bH_\textbf{f}$ telle que $p-1|k-1$ et $\chi_\phi=1$. Alors, par définition, $g_\phi$ est de niveau $Np$ et son caractère $\epsilon$ est de niveau $N$. Comme $p \nmid N$, $g_\phi$ est nécessairement $p$-old d'après \cite[Theorem 4.6.17/2]{miyake2006modular}. Ainsi, $g_\phi$ est la $p$-stabilisation d'une forme primitive de niveau $N$. 
\end{remarque} 

\subsubsection{Déformation ordinaire}
Soit $\phi$ une spécialisation classique de $\bH_\textbf{f}$. Les travaux de Deligne (\cite{deligne1971formes}), généralisant ceux de Eichler et de Shimura, montrent l'existence d'une représentation galoisienne $(\rho_\phi,V_{p,\phi})$ irréductible sur $\bQ$ et à coefficients dans $\cO_\phi \otimes \Qp$, non-ramifiée en-dehors de $Np$, de dimension 2 et impaire attachée à $g_\phi$, au sens où
$$\forall \ell \nmid Np,\quad \rho_\phi(\Frob_\ell)=a_\ell(g_\phi).$$
Par ailleurs, comme $g_\phi$ est ordinaire de poids $\geq 2$, $V_{p,\phi}$ admet une unique droite $G_{\Qp}$-stable $V^+_{p,\phi} \subseteq V_{p,\phi}$ de quotient  $V_{p,\phi}^-$ non-ramifié, et $\Frob_v$ agit sur $V_{p,\phi}^-$ par multiplication par $a_p(g_\phi)$. D'autre part, la représentation résiduelle de $\rho_\phi$ est isomorphe à $\ob{\rho}$, donc $\rho_\phi$ est résiduellement irréductible et $V_{p,\phi}$ admet une seule classe d'homothétie de réseaux $G_{\bQ}$-stables. On rappelle à présent comment interpoler un réseau stable de chacun des $V_{p,\phi}$ en famille. 

D'après \cite[Theorem 2.1]{hida1986galois}, il existe une représentation galoisienne sur $\bH_\textbf{f} \otimes_{\Zp[[X]]} \Frac(\Zp[[X]])$, non-ramifiée en tout $\ell \nmid Np$, qui envoie $\Frob_\ell$ sur $a_\ell(\textbf{f})$. Comme $\bar{\rho}$ est absolument irréductible, cette représentation est même définie sur $\bH_\textbf{f}$ d'après \cite[Théorème 1]{nyssen1996pseudo}  (voir aussi \cite[Corollaire 5.2]{rouquier1996caracterisation}). Comme $\bar{\rho}$ est $p$-distinguée, l'espace des $I_p$-coinvariants est $\bH_\textbf{f}$-libre de rang 1, et $\Frob_v$ agit par multiplication par $a_p(\textbf{f})$, d'après \cite[Theorem 2.2.2]{wiles1988ordinary}. On obtient le résultat suivant.

\begin{theorem}\label{bigT}
	Supposons que $\ob{\rho}$ est absolument irréductible et $p$-distinguée (ce qui résulte de nos hypothèses \textbf{(rég)} et \textbf{(deg)}, cf. Paragraphe \ref{subsec:hyp_f_alpha}). Il existe un $\bH_\textbf{f}$-module libre de rang 2, noté $\cT_\textbf{f}$, et une représentation continue impaire et irréductible 
	$$\rho_\textbf{f} : G_\bQ \longrightarrow \GL_{\bH_\textbf{f}}(\cT_\textbf{f}),$$
	non-ramifiée en-dehors de $Np$ telle que
	$\Tr(\rho_\textbf{f}(\Frob_\ell))=a_\ell(\textbf{f})$ pour tout $\ell\nmid Np$. On a en outre les propriétés suivantes :
	\begin{itemize}
		\item Pour toute spécialisation arithmétique $g_\phi=\phi(\textbf{f})$ de $\textbf{f}$, le $\cO_{\phi}$-module $T_{p,\phi}:=\cT_\textbf{f} \otimes_{\bH_\textbf{f},\phi} \cO_{\phi}$  s'identifie à un réseau de $V_{p,\phi}$. De même, on a $T_p \simeq \cT_\textbf{f} \otimes_{\bH_\textbf{f},\phi_{\alpha}} \cO$.  
		\item  Il existe $\cT^+_\textbf{f} \subseteq \cT_\textbf{f}$ un sous-$\bH_\textbf{f}$-module libre de rang 1 qui est stable par $G_{\Qp}$ vérifiant les propriétés suivantes. Le quotient $\cT^-_\textbf{f}:=\cT_\textbf{f}/\cT^+_\textbf{f}$ est libre de rang 1, et l'action de $G_{\Qp}$ sur $\cT_\textbf{f}^-$ est donnée par le caractère non-ramifié envoyant $\Frob_v$ sur $a_p(\textbf{f})$. De plus, pour toute spécialisation arithmétique $g_\phi$ de $\textbf{f}$, la filtration obtenue se spécialise sur la filtration ordinaire $T^\pm_{p,\phi}$ attachée à $T_{p,\phi}$. De même, on a $T_p^\pm\simeq \cT_\textbf{f}^\pm \otimes_{\bH_\textbf{f},\phi_{\alpha}} \cO$.
	\end{itemize}
\end{theorem}

\subsection{Fonctions $L$ $p$-adiques de formes modulaires ordinaires en famille}\label{sec:Lpfamille}
Soit $g$ une forme modulaire cuspidale primitive de poids $k\geq 2$ et $p$-stabilisée, à coefficients dans une extension finie de $\Qp$. Lorsque $v_p(a_p(g))<k-1$, on sait attacher grâce à Amice-Vélu \cite{amice1975distributions} et Vishik \cite{vishik1976non} (voir aussi \cite{mtt}) une distribution d'ordre de croissance $\leq v_p(a_p(g))$ interpolant les valeurs spéciales en $s=1,\ldots,k-1$ de la série $L(g, \chi,s)=\sum_{n\geq 1} a_n(g)\chi(n)n^{-s}$, où $\chi$ est un caractère fini galoisien sur $\bQ$ de conducteur une puissance de $p$ (que l'on voit comme un caractère de Dirichlet), convenablement normalisées par deux périodes complexes $\Omega_g^\pm$. Lorsque $g$ est $p$-ordinaire et que le choix des périodes est convenablement normalisé, on obtient une série formelle $L^\an_p(g,T)$ à coefficients dans une extension finie de $\Zp$. Dans notre situation où nous aimerions construire une série interpolant celles associées aux spécialisations $g_\phi$, il convient de faire un choix de "périodes canoniques en famille". Comme $\ob{\rho}$ et irréductible et $p$-distinguée, nous pouvons faire appel à la construction de la fonction $L$ $p$-adique en famille donnée par Emerton-Pollack-Weston \cite{emerton2006variation} pour obtenir une série formelle $L^\an_p(\textbf{f},T)\in \bH_\textbf{f}[[T]]$ dont on va rappeler la propriété d'interpolation.

La $\Zp[[X]]$-algèbre $\bH_N$ est finie donc semi-locale, elle est isomorphe au produit fini de ses localisés aux idéaux maximaux $\prod_{\gm'}\left(\bH_N\right)_{\gm'}$. On note $\gm$ l'idéal maximal correspondant à $\bar{\rho}$. On note aussi $\bH_N^*$ l'algèbre de Hecke universelle ordinaire agissant sur toutes les formes $\Zp[[X]]$-adiques ordinaires (non-nécessairement paraboliques). Pour $k\geq 2$ et $r\geq 1$, soit $\bH^*_{N,r,k}$ (resp. $\bH_{N,r,k}$) l'algèbre de Hecke agissant sur toutes les formes modulaires (resp. les formes paraboliques) $p$-ordinaires de poids $k$ et de niveau $Np^r$. La suite exacte en homologie de la paire $(X_1(Np^r),\{\textrm{pointes}\})$ (où $X_1(Np^r)$ est la courbe modulaire compacte de niveau $Np^r$) induit (après localisation en l'idéal maximal induit par $\gm$ et projection ordinaire) un isomorphisme 
$$H_1(X_1(Np^r),\tilde{L}_k(\Zp))_\gm^\textrm{ord} \simeq H_1(X_1(Np^r),\{\textrm{pointes}\},\tilde{L}_k(\Zp))_\gm^\textrm{ord}$$
de $\left(\bH_{N,r,k}\right)_\gm\simeq \left(\bH^*_{N,r,k}\right)_\gm$-modules, où $\tilde{L}_k(\Zp)$ est le système local associé au $\Zp$-module des polynômes à coefficients dans $\Zp$ de degré $\leq k-2$. On note $\left(\gM_{N,r,k}\right)_\gm$ ces modules isomorphes. $\left(\gM_{N,r,k}\right)_\gm$ est la somme directe des deux sous-espaces propres $\left(\gM_{N,r,k}\right)_\gm^\pm$ pour l'action induite par la conjugaison complexe sur $X_1(Np^r)$, qui sont tous deux libres de rang 1 sur $\bH_{N,r,k}$ d'après \cite[Proposition 3.1.1]{emerton2006variation}, car $\ob{\rho}$ est irréductible et $p$-distinguée. Fixons un isomorphisme $\alpha_{N,r,k}^\pm : \left(\gM_{N,r,k}\right)_\gm^\pm \simeq \left(\bH_{N,r,k}\right)_\gm$, ainsi qu'un isomorphisme $\bar{\bQ}_p\simeq \bC$. 

\begin{proposition}\label{prop:fonction_L_p_adique_poids_k}
	Soit $g$ une forme propre parabolique $p$-ordinaire de poids $k\geq 2$, niveau $Np^r$, à coefficients dans une extension finie $\cO_g$ de $\Zp$, et de représentation résiduelle $\ob{\rho}$. Il existe deux périodes $\Omega_g^\pm\in \bC$ appelées périodes canoniques, bien définies à une unité de $\cO_g$ près et ne dépendant que de l'isomorphisme $\alpha_{N,r,k}^\pm$, ainsi qu'une unique série formelle $L^\an_p(g,T)\in \cO_g[[T]]$ satisfaisant la formule d'interpolation suivante : pour tout $0 \leq m \leq k-2$ et $\xi$ racine primitive $p^{t-1}$-ème de l'unité, on a 
	$$L^\an_p(g,\xi(1+p)^m-1) = e_g(\xi,m) \dfrac{p^{t'(m+1)} m!}{a_p(g)^{t'} (-2i\pi)^{m+1}\gG(\omega^{-m} \chi_\xi) \Omega_g^{(-1)^m}} L(g,\omega^{-m} \chi_\xi,m+1),$$
	où $\omega$ est le caractère de Teichmüller et $\chi_\xi : G_\bQ \twoheadrightarrow \Gamma \twoheadrightarrow \mu_{p^{t-1}}$ est le caractère envoyant $\gamma$ sur $\xi$, avec $e_g(\xi,m):=1-a_p(g)^{-1}p^m$ et $t':=0$ lorsque $\xi=1$ et $p-1|m$, et $e_g(\xi,m):=1$ et $t':=t$ sinon, où $\gG(-)$ désigne la somme de Gauss de caractères de Dirichlet.
\end{proposition} 
\begin{proof}
	La construction de $L^\an_p(g,T)$ est donnée dans \cite[§3.2 et Prop. 3.4.3]{emerton2006variation}. 
\end{proof}

Soit $\gp_{N,r,k}$ le produit de tous les idéaux premiers de hauteur 1 de $\bH_N$, de poids $k$ et de niveau divisant $Np^r$ et de représentation résiduelle $\ob{\rho}$. Les théorèmes de contrôle donnent des identifications $\left(\bH_N\right)_\gm/\gp_{N,r,k}\simeq \left(\bH_{N,r,k}\right)_\gm$  et $\left(\bH^*_N\right)_\gm/\gp_{N,r,k}\simeq \left(\bH^*_{N,r,k}\right)_\gm$, de même qu'un isomorphisme $\left(\gM_{N}\right)_\gm^\pm \otimes \left(\bH_N\right)_\gm/\gp_{N,r,k} \simeq \left(\gM_{N,r,k}\right)_\gm^\pm$, où $\left(\gM_N\right)_\gm$ est le $\left(\bH_N\right)_\gm\simeq \left(\bH^*_N\right)_\gm$-module libre de rang 2 défini comme étant la limite projective 
$$\varprojlim_r H_1(X_1(Np^r),\Zp)_\gm^\textrm{ord} \simeq \varprojlim_r H_1(X_1(Np^r),\{\textrm{pointes}\},\Zp)_\gm^\textrm{ord} $$
(cf. \cite[Proposition 3.3.1]{emerton2006variation} et sa preuve). En particulier, le choix d'un isomorphisme 
$$\alpha_N^\pm :\left(\bH_N\right)_\gm \simeq \left(\gM_N\right)^\pm_\gm$$
détermine, pour tout $k$ et $r$, un choix d'isomorphisme 
$$\alpha_{N,r,k}^\pm : \left(\gM_{N,r,k}\right)_\gm^\pm \simeq \left(\bH_{N,r,k}\right)_\gm,$$
en posant $\alpha_{N,r,k}^\pm = \alpha_N^\pm \mod \gp_{N,r,k}$. En ce sens, on peut dire que le choix des périodes $\Omega_g^\pm$ est effectué \textit{en famille} (voir \textit{loc. cit.}). On peut de même construire une mesure sur $\bZ_p^\times$ à coefficients dans $\left(\bH_N\right)_\gm$ interpolant les mesures $\mu_k^\pm$. En vue de nos applications, nous considérons plutôt la composée avec la projection 
$$\left(\bH_N\right)_\gm \twoheadrightarrow \left(\bH^\textrm{new}_N\right)_\gm \twoheadrightarrow \bH_\textbf{f}.$$
Fixons l'isomorphisme précédent $\alpha_N^\pm$. On a l'analogue en famille de la Proposition \ref{prop:fonction_L_p_adique_poids_k} : 

\begin{proposition}\label{spé}
	Il existe une série formelle $L^\an_p(\textbf{f};T) \in \bH_\textbf{f}[[T]]$ telle que, pour toute spécialisation arithmétique $g_\phi=\phi(\textbf{f})$ de $\textbf{f}$, on ait 
	$$ \tilde{\phi} \left( L^\an_p(\textbf{f},T)\right) = L^\an_p(g_\phi;T), $$
	où $\tilde{\phi}: \bH_\textbf{f}[[T]] \twoheadrightarrow \cO_\phi[[T]]$ est morphisme appliquant $\phi$ aux coefficients des séries formelles, et où $L^\an_p(g_\phi;T)$ est la fonction $L$ $p$-adique analytique de $g_\phi$ calculée dans la Proposition \ref{prop:fonction_L_p_adique_poids_k} avec l'isomorphisme $\alpha_{N,r,k}^\pm$ choisi comme étant égal à $\alpha_N^\pm \mod \gp_{N,r,k}$.
	
\end{proposition}

L'élément $L^\an_p(\textbf{f};T)$ est défini à multiplication par une unité de $\bH_\textbf{f}$ près.

\begin{definition}\label{def:cL_p_f_alpha}
	On définit $L^\an_p(f_\alpha;T) \in \cO[[T]]$ la fonction $L$ $p$-adique analytique de $f_\alpha$ comme étant l'image de $L^\an_p(\textbf{f},T)$ par l'application $\tilde{\phi_{\alpha}} : \bH_\textbf{f}[[T]]  \longrightarrow \cO[[T]]$. Elle est bien définie à une unité multiplicative de $\cO$ près. 
\end{definition}

\begin{remarque}
	\begin{enumerate}
		\item Il semble a priori difficile de montrer que la série formelle que l'on a défini n'est pas identiquement nulle. On verra néanmoins que cela est vrai, en admettant la véracité de la conjecture principale pour les formes modulaires de poids supérieur (Conjecture \ref{IMCg}).
		\item Sous la seule hypothèse \textbf{(rég)}, Bellaïche et Dimitrov ont proposé une définition d'une fonction $L$ $p$-adique (\cite[Corollary 1.3]{bellaiche2016eigencurve}) seulement définie à un élément de $L^\times$ près. Il serait intéressant d'essayer de généraliser cette définition au cas où il existe plusieurs familles de Hida se spécialisant en une forme modulaire non-régulière en $p$. Chaque famille devrait produire une fonction $L$ $p$-adique à deux variables qu'on pourrait spécialiser en $f_\alpha$. Du point de vue algébrique, un candidat naturel pour $V^+_p$ serait alors la spécialisation en poids 1 de la $G_{\Qp}$-droite ordinaire de la famille construite dans \cite[Theorem 2.2.2]{wiles1988ordinary}. Il est cependant possible que $V_p^+\subseteq V_p$ ne soit pas rationnelle, i.e., qu'elle ne provienne pas d'une $p$-stabilisation ordinaire $V^+$ de $V$.
	\end{enumerate}
\end{remarque}

\subsection{Conjectures principales}
\subsubsection{}\label{selphi}
Soit $g_\phi$ une spécialisation classique de $\textbf{f}$. Le groupe de Selmer (resp. groupe de Selmer dual) attaché à $g_\phi$ est par définition le $\cO_\phi[[T]]$-module
$$\Sel_{\infty}(\phi):=\Sel_{\infty}(T_{p,\phi},T_{p,\phi}^+), \qquad \textrm{resp.} \qquad X_{\infty}(\phi):=\Sel_{\infty}(\phi)^\vee.$$
On pose $L_p^{\alg}(g_\phi,T):=L_p^{\alg}(X_\infty(\phi),T)$ lorsque $X_\infty(\phi)$ est de torsion sur $\cO_\phi[[T]]$. La Conjecture Principale pour les formes primitives $p$-ordinaires $p$-stabilisées de poids $k \geq 2$ prédit l'égalité suivante (\cite[Conjecture 3.24]{skinner2014iwasawa}, à noter que \textit{loc.cit.} définit les représentations galoisiennes et leurs fonctions $L$ avec le Frobenius géométrique, ce qui change - en apparence seulement - l'énoncé).

\begin{conjecture}\label{IMCg}
	Soit $g_\phi$ une spécialisation arithmétique de $\textbf{f}$. Le $\cO_\phi[[T]]$-module $X_\infty(\phi)$ est de torsion, et il existe une unité $u_\phi$ de $\cO_\phi[[T]]$ telle que
	$$ u_\phi \cdot L_p^{\alg}(g_\phi,T)=L^\an_p(g_\phi,T).$$
\end{conjecture}

Kato a prouvé une divisibilité partielle dans le cas où $g_\phi$ est la $p$-stabilisation d'une forme primitive ordinaire de niveau premier à $p$ (\cite[Theorem 17.4]{kato2004p}). D'après la Remarque \ref{p-old}, ceci est valable précisément lorsque $p-1|k_\phi-1$ et $\chi_\phi=1$. On a donc le théorème suivant.

\begin{theorem}[Kato]\label{Katotors}
	Soit $g_\phi$ une spécialisation arithmétique de $\textbf{f}$ telle que $p-1|k_\phi-1$ et $\chi_\phi=1$. Le $\cO_\phi[[T]]$-module $X_\infty(\phi)$ est de torsion, et 
	$$ L_p^{\alg}(g_\phi,T) \quad \textrm{divise} \quad L^\an_p(g_\phi,T))$$
	dans $\cO_\phi[[T]][\tfrac{1}{p}]$.
\end{theorem}

Malheureusement, nous ne pouvons faire appel aux résultats de Skinner-Urban \cite[Theorem 3.29]{skinner2014iwasawa}, qui établit l'autre divisibilité sous l'hypothèse additionnelle (trop restrictive pour nos applications) de trivialité du caractère central de $\textbf{f}$. 

Nous formulons la conjecture suivante :

\begin{conjecture}\label{IMC_f_alpha}
	Il existe une unité $u$ de $\cO[[T]]$ telle que 
	$$ u \cdot L_p^{\alg}(f_\alpha,T)=L^\an_p(f_\alpha,T).$$
\end{conjecture}

Notons que les conditions $p-1|k_\phi -1$ et $\chi_\phi=1$ sont vraies pour $\phi=\phi_\alpha$. Les arguments de passage à la limite utilisé dans la preuve du Théorème \ref{ThC} montreront (sous les hypothèses du Théorème \ref{ThC}) que la Conjecture \ref{IMCg} implique la Conjecture \ref{IMC_f_alpha}. Notre résultat est le suivant :

\begin{theorem}\label{ThC}
	On a $$ L_p^{\alg}(f_\alpha,T)\quad \textrm{divise}\quad L^\an_p(f_\alpha,T) $$
	dans $\cO[[T]][\tfrac{1}{p}]$. De plus, si la Conjecture \ref{IMCg} est vraie pour toute spécialisation $g_\phi$ de poids $k_\phi$ suffisamment proche $p$-adiquement de 1 et de niveau $Np$, alors la Conjecture \ref{IMC_f_alpha} est vraie.
\end{theorem}

Il est facile d'inclure le poids 1 dans les résultats classiques sur la variation des $\mu$-invariants dans les familles de Hida (voir par exemple \cite[Theorem 4.3.3]{emerton2006variation}). On obtient la proposition suivante :

\begin{proposition}\label{prop:mu=0}
	Si le $\mu$-invariant de $X_\infty(f_\alpha)$ est nul, alors il en est de même pour $X_\infty(g_\phi)$ pour toute spécialisation classique $\phi$ de \textbf{f}, et réciproquement.
\end{proposition}
\begin{proof}
	\'Etant donnée une spécialisation classique $g_\phi$ de $\textbf{f}$ (incluant $g_\phi=f_\alpha$), le $\mu$-invariant de $X_\infty(g_\phi)$ est nul si et seulement si le $\cO/p\cO$-espace vectoriel $X_\infty(g_\phi)/pX_\infty(g_\phi)$ est de dimension finie. Or, en appliquant la Proposition \ref{changement_base} (comme dans la preuve du Lemme \ref{spécialisation_torsion}), ce dernier s'identifie à un "groupe de Selmer résiduel" (\textit{i.e.}, ne dépendant que de $\ob{\rho}$ et du choix de $\alpha \mod p$), ce qui ne dépend pas de $\phi$.
\end{proof}

\subsection{Articulation de la preuve du Théorème C}

D'après \cite[§7.3, Proposition 1.12]{dimitrov2014local}, l'anneau localisé $\left(\bH_\textbf{f}\right)_{\gp_\alpha}$ est de valuation discrète. Dans la Section \ref{para_étale}, on construit un paramétrage de $\textbf{f}$ au voisinage de $f_\alpha$, donnant une suite de formes modulaires primitives $p$-stabilisées $g_n$ (cf. Notation \ref{gn}) dont les coefficients de Fourier vivent dans une extension finie $\cO'$ de $\cO$ et convergent vers ceux de $f_\alpha$. Les étapes de la preuve du Théorème \ref{ThC} sont rassemblées dans le schéma suivant : 

$$
\xymatrix{
	L_p^{\alg}(g_n,T) \ar[d]_{(c)}^{n \rightarrow +\infty} & \stackrel{(a)}{\mbox{divise}} & L^\an_p(g_n,T) \ar[d]^{(b)}_{n \rightarrow +\infty} \\
	L_p^{\alg}(f_\alpha,T) & & L^\an_p(f_\alpha,T) 
}
$$ 
Toutes les fonctions $L$ $p$-adiques sont des éléments de l'anneau topologique $\cO'[[T]]$, et les divisibilités sont dans $\cO'[[T]][\tfrac{1}{p}]$. Le point (a) est une application du théorème de Kato (Théorème \ref{Katotors}). Les points (b) et (c) sont respectivement démontrés dans le Lemme \ref{bn} et la Proposition \ref{an}, après avoir construit une fonction $L$ $p$-adique analytique et algébrique au voisinage de $f_\alpha$ dans les Sections \ref{para_analytique} et \ref{para_alg}. Une fois les points (a), (b) et (c) prouvés, la preuve du Théorème \ref{ThC} découle immédiatement du lemme suivant, pour $A=\cO'[[T]]$ (muni de sa topologie d'anneau local), $\pi$ une uniformisante de $\cO'$, $a_n=L_p^{\alg}(g_n,T)$, $a=L_p^{\alg}(f_\alpha,T)$, $b_n=L^\an_p(g_n,T)$, et $b=L^\an_p(f_\alpha,T)$.

\begin{lemme}\label{elem}
	Soit $A$ un anneau topologique compact. Soient $(a_n)_n$ et $(b_n)_n$ deux suites d'éléments de $A$ convergeant respectivement vers $a$ et $b$.
	\begin{enumerate}
		
		\item Si $a_n$ divise $b_n$ dans $A$ pour tout $n$, alors $a$ divise $b$.
		\item Soit $\pi \in A$ un élément premier, régulier et topologiquement nilpotent. Si $a_n$ divise $b_n$ dans $A[\tfrac{1}{\pi}]$ pour tout $n$, et si $a \neq 0$, alors $a$ divise $b$ dans $A[\frac{1}{\pi}]$.
	\end{enumerate}
\end{lemme}

\begin{proof}
	Le point (1) se prouve facilement en extrayant des sous-suites convergentes. Traitons (2), et soit $n \in \bN$. Il existe $k(n) \in \bN$ et $c_n\in A$ tels que $\pi^{k(n)}b_n=a_n c_n$. L'élément $\pi$ étant régulier, quitte à simplifier suffisamment de fois par $\pi$, on peut supposer que l'on a, ou bien $k(n)=0$, ou bien $\pi \nmid c_n$. Dans tous les cas, $\pi$ étant premier, on a $\pi^{k(n)}|a_n$. Par hypothèse, $a_n$ converge vers $a\neq 0$ et $\pi$ est topologiquement nilpotent, donc $k(n)$ est borné avec $n$ d'après (1). Ainsi, il existe un entier $K$ suffisamment grand tel que $a_n$ divise $\pi^K b_n$ pour tout $n$, et donc $a$ divise $\pi^K b$ par (1), ce qu'on voulait démontrer. 
\end{proof}

\subsection{Lissité et paramétrage local de la famille de Hida}\label{para_étale}

Soit $F(W)=\sum_n c_n W^n \in \ob{\bQ}_p[[W]]$ une série formelle de rayon de convergence positif. Alors il existe un entier $r$ tel que $F(W)$ converge sur le disque de rayon $p^{-r}$, c'est-à-dire que la suite $(c_n p^{nr})_n$ est bornée. On peut donc voir $F(Y)$ comme un élément du sous-anneau $\ob{\bZ}_p[[W/p^r]][\tfrac{1}{p}]$ de $\ob{\bQ}_p[[W]]$.

\begin{lemme}\label{artin_approx}
	Soit $F(W) \in \ob{\bQ}_p[[W]]$. Supposons que $F(W)$ est entier sur $\Zp[[W]]$. Alors il existe une extension finie $M$ de $\Qp$ et un entier $r$ tel que $F(W) \in \cO_M[[W/p^r]]$.
\end{lemme}
\begin{proof}
	Soit $Q(W,Z)\in \Zp[[W]][Z]$ un polynôme (en la variable $Z$) unitaire de degré $d_Q$ qui s'annule en $F(W)$. Le théorème d'approximation d'Artin \cite[Theorem 1.2]{artin1968solutions} implique que, pour tout entier $c\geq 0$ et pour toute racine $G(W)$ de $Q(W,Z)$, il existe une série formelle $\hat{G}_c(W) \in \ob{\bQ}_p[[W]]$ de rayon de convergence positif qui est une racine de $Q(W,Z)$ et qui satisfait la congruence $G(W) \equiv \hat{G}_c(W) \mod W^c$. Comme $\ob{\bQ}_p[[W]]$ est intègre, $Q(W,Z)$ a un nombre fini de racines dans $\ob{\bQ}_p[[W]]$, et on peut donc choisir un entier $c$ tel que, pour toute paire de racines distinctes $G_1(W),G_2(W)$ de $Q(W,Z)$, on ait $G_1(W)\not\equiv G_2(W) \mod W^c$. Pour un tel $c$, on voit que l'on a $G(W) = \hat{G}_c(W)$, et donc toute racine du polynôme $Q(W,Z)$ converge au voisinage de $0$. En particulier, il existe un entier $r$ tel que $F(W) \in \ob{\bZ}_p[[W/p^r]][\tfrac{1}{p}]$. 
	
	Soit $M$ le compositum de toutes les extensions de $\Qp$ de degré inférieur ou égal à $d_Q$. Pour tout $w\in p^{r+1}\Zp$, l'équation $Q(w,F(w))=0$ implique que $F(w)$ satisfait une équation de degré $d_Q$, et donc $F(w) \in M$. On en déduit que $F(W)$ est à coefficients dans $M$, et donc $F(W) \in \cO_M[[W/p^r]][\tfrac{1}{p}]$. Enfin, comme $F(W)$ est entier, il est $p$-entier et donc $F(W) \in \cO_M[[W/p^r]]$ comme annoncé.
\end{proof}

Soient $r\geq 0$ et $e\geq 1$ des entiers et $M$ une extension finie de $L$. On notera $\cO_M[[X^{1/e}/p^r]]$ la $\cO_M[[X]]$-algèbre $\cO_M[[X,Y]]/(p^{re}Y^e-X)$. Un morphisme de spécialisation $X=a$ pour $a\in \gm_{\bC_p}$ s'étend à $\cO_M[[X^{1/e}/p^r]]$ dès que $p^{-re}a\in \gm_{\bC_p}$, et dépend du choix d'une racine $e$-ième de $p^{-re}a$ dans $\bC_p$.

\begin{proposition}\label{phi_infty}
	Il existe des entiers $r$ et $e\geq 1$, une extension finie $M/L$ et un morphisme (injectif) de $\Zp[[X]]$-algèbres $\Phi$ tel que le diagramme suivant soit commutatif :
	
	\begin{equation}\label{diagramme_phi_infty}
	\xymatrix{
		\bH_\textbf{f} \ar[r]^{\phi_{\alpha}} \ar@{.>}[rd]_{\Phi} & \cO_M \\
		\Zp[[X]] \ar[u] \ar[r] & \cO_M[[X^{1/e}/p^r]] \ar[u]_{X=0} 
	}
	\end{equation}
	
\end{proposition}

\begin{proof}
	Comme $\left(\bH_\textbf{f}\right)_{\gp_\alpha}$ est un anneau de valuation discrète d'égale caractéristique, son complété est un anneau de séries formelles en une variables sur son corps résiduel. Donc il existe un morphisme injectif d'anneaux locaux
	$$\Phi : \left(\bH_\textbf{f}\right)_{\gp_\alpha} \hookrightarrow \ob{\bQ}_p[[Y_0]],$$
	où $Y_0$ est une variable formelle. Comme $\Phi$ est local, l'image de $X\in \gp_\alpha$ par $\Phi$ est divisible par $Y_0$. Donc $X$, vu comme élément de $\ob{\bQ}_p[[Y_0]]$, est une série formelle sans terme constant $X=aY_0^e + bY_0^{e+1} + \ldots$ où $e\geq 1$ et $a\neq 0$. Soit $H(Y_0)\in \ob{\bQ}_p[[Y_0]]$ une racine $e$-ième de la série $a+bY_0+\ldots$, de sorte que $X=\left(Y_0 H(Y_0)\right)^e$. Comme $H(0)\neq 0$, on voit que $Y=Y_0 H(Y_0)$ définit une nouvelle variable formelle, \textit{i.e.}, on a un isomorphisme d'anneaux locaux $\ob{\bQ}_p[[Y_0]]\simeq\ob{\bQ}_p[[Y]]\simeq \ob{\bQ}_p[[X^{1/e}]]$. La restriction de $\Phi$ à $\bH_\textbf{f}$ donne un morphisme de $\Zp[[X]]$-algèbres $\bH_\textbf{f} \hookrightarrow \ob{\bQ}_p[[X^{1/e}]]$. D'après le Lemme \ref{artin_approx}, l'image de ce morphisme est incluse dans un anneau de la forme $\cO_M[[X^{1/e}/p^r]]$ pour une certaine extension finie $M$ de $L$, car $\bH_\textbf{f}$ est fini sur $\Zp[[X]]$. On note encore $\Phi$ le morphisme de $\Zp[[X]]$-algèbres obtenu :
	$$\Phi : \bH_\textbf{f} \hookrightarrow \cO_M[[X^{1/e}/p^r]].$$
	Il vérifie par construction $\Phi^{-1}\left((X^{1/e}/p^r)\right)=\gp_\alpha$, ce qui rend le diagramme (\ref{diagramme_phi_infty}) commutatif.
\end{proof}

\begin{corollaire}\label{fdag}
	Soit $Y$ la variable formelle $X^{1/e}/p^r$ et soit $\cA:=\cO_M[[X^{1/e}/p^r]]\simeq \cO_M[[Y]]$. On définit 
	$$\textbf{f}^\dag(Y)=\sum_{m\geq 1} a_m(\textbf{f}^\dag;Y)q^m \in \cA[[q]], \quad \textrm{où} \quad a_m(\textbf{f}^\dag;Y):=\Phi(a_m(\textbf{f})).$$
	Alors $\textbf{f}^\dag(Y)$ paramètre $\textbf{f}$ au voisinage de $f_\alpha$ au sens suivant : on a $\textbf{f}^\dag(0)=f_\alpha$, et pour tout entier $k\geq 2$ tel que $p^{re}|k-1$, et pour tout choix $y\in M$ d'une racine $e$-ième de $\frac{u^{k-1}-1}{p^{re}}$, $\textbf{f}^\dag(y)$ est une spécialisation arithmétique de $\textbf{f}$ de poids $k$, de niveau $Np$ et caractère $\epsilon \omega^{1-k}$.
\end{corollaire}

\begin{proof}
	On a en effet $p^{re+1}|u^{k-1}-1$, donc $|y|_p <1$ et il existe un unique morphisme $\Psi_y :\cA \longrightarrow \cO_M$ satisfaisant $\Psi_y(Y)=y$. On a de plus $\Psi_y(X)=u^{k-1}-1$, donc la composée $\phi_y =\Psi_y \circ \Phi : \bH_\textbf{f} \longrightarrow \cO_M$ est une spécialisation arithmétique de $\bH_\textbf{f}$ de poids $k$ et de caractère $\chi_{\phi_y}=1$. La forme $\textbf{f}^\dag(y)=\phi_y(\textbf{f})$ définit bien une spécialisation arithmétique de $\textbf{f}$ avec les propriétés annoncées.
\end{proof}

\begin{notation}\label{gn}
	Quitte à agrandir $M$, on peut supposer que $M$ contient toutes les extensions de $\Qp$ de degré $\leq e$, et donc $M$ contient toutes les racines $e$-ièmes d'éléments de $\Qp^\times$. On fixe dans toute la suite une telle extension $M$.
	
	Pour tout entier naturel $n$, on note $k(n)=(p-1)p^{n+re}+1$. On fixe $y_{n} \in M$ une racine $e$-ième de $\frac{u^{k(n)-1}-1}{p^{re}}$ et on pose $g_n=\textbf{f}^\dag(y_n)$. Comme $p-1|k(n)-1$, $g_n$ est la $p$-stabilisation d'une forme primitive classique de poids $k(n)$, de niveau $N$, de caractère $\epsilon$ et à coefficients dans $\cO_M$ d'après le Corollaire \ref{fdag} et la Remarque \ref{p-old}. On note $\phi_n :\bH_\textbf{f} \longrightarrow \cO_M$ le morphisme de spécialisation correspondant à $g_n$ et $\gp_n=\gp_{\phi_n} \subseteq \bH_\textbf{f}$. Notons que dans $\cO_M$, on a
	$$\lim_{n\rightarrow+\infty} y_n=0,$$
	donc les coefficients du $q$-développement de $g_n$ convergent vers ceux de $f_\alpha$ lorsque $n\rightarrow \infty$.
	
\end{notation}

\subsection{Fonction $L$ $p$-adique analytique au voisinage de $f_\alpha$}\label{para_analytique}

On définit une fonction $L$ $p$-adique analytique au voisinage de $f_\alpha$ comme étant l'élément $L^{\an,\dag}_p(Y,T) \in \cA[[T]]$ obtenu en prenant l'image de $L^\an_p(\textbf{f},T)$ par le morphisme $\tilde{\Phi} :\bH_\textbf{f}[[T]] \rightarrow \cA[[T]]$ appliquant $\Phi$ aux coefficients. Ainsi, d'après le Corollaire \ref{fdag} et avec les Notations \ref{gn}, on a : 
$$L_p^{\an,\dag}(0,T)=L^\an_p(f_\alpha,T) \qquad \textrm{et} \qquad  L^{\an,\dag}_p(y_n,T)=L^\an_p(g_n,T),$$
pour tout entier naturel $n$.

\begin{lemme}\label{bn}
	La suite $(L^\an_p(g_n,T))_n$ converge vers $L^\an_p(f_\alpha,T)$ dans $\cO_M[[T]]$. 
\end{lemme}

\begin{proof}
	Il est clair que toute série formelle $F(Y,T) \in \cO_M[[Y,T]]$ définit une application continue $y \longmapsto F(y,T)$ sur $\cO_M-\cO_M^\times$ à valeurs dans $\cO_M[[T]]$. En considérant $F(Y,T)= L_p^{\an,\dag}(Y,T)$, on a donc $\lim_{n} L^\an_p(g_n,T)=\lim_{n} L_p^{\an,\dag}(y_n,T)=L_p^{\an,\dag}(0,T)=L^\an_p(f_\alpha,T)$ dans $\cO_M[[T]]$.
\end{proof}

\subsection{Fonction $L$ $p$-adique algébrique au voisinage de $f_\alpha$}\label{para_alg}
Il n'existe pas a priori de fonction $L$ $p$-adique algébrique associée au groupe de Selmer $X_\infty(\cT_\textbf{f},\cT_\textbf{f}^+)$, ni même d'idéal caractéristique, car l'anneau $\bH_\textbf{f}[[T]]$ n'est pas nécessairement factoriel, ni même intégralement clos. Notre paramétrage local de la famille de Hida nous permet de surmonter ce problème et de travailler sur un anneau de séries formelles $\cA[[T]]\simeq \cO_M[[Y,T]]$, qui est factoriel et sur lequel on peut en outre appliquer les résultats de la Section \ref{section3}.

On notera simplement les $\cA$-modules $\bT = \cT_\textbf{f} \otimes_{\bH_\textbf{f},\Phi} \cA$ et $\bD=\bT \otimes_{\cA} \cA^\vee$. On définit similairement $\bT^\pm$ et $\bD^\pm$. On définit aussi les $\cA[[T]]$-modules 
$$\Sel_\infty(\textbf{f}^\dag) = \Sel_\infty(\bT,\bT^+), \qquad \textrm{resp.} \qquad X_\infty(\textbf{f}^\dag) = \Sel_\infty(\textbf{f}^\dag)^\vee.$$
Rappelons que pour tout entier naturel $n$, le groupe de Selmer $\Sel_\infty(\phi_n)$ attaché à $g_n$ est de torsion d'après le Théorème \ref{Katotors}. Nous démontrons dans la suite la proposition suivante.

\begin{proposition}\label{an}
	La suite $(L_p^{\alg}(g_n,T))_n$ converge vers $L_p^{\alg}(f_\alpha,T)$ dans $\cO_M[[T]]$.
\end{proposition}

\begin{lemme}\label{spécialisation_torsion}
	On a des isomorphismes de $\cO_M[[T]]$-modules:
	\[
	X_\infty(\textbf{f}^\dag)/\left(Y-y_n\right)\cdot X_\infty(\textbf{f}^\dag) \simeq X_\infty(\phi_n) \otimes_{\cO_{\phi_n}} \cO_M, \]\[ X_\infty(\textbf{f}^\dag)/Y\cdot X_\infty(\textbf{f}^\dag) \simeq X_{\infty}(f_\alpha) \otimes_{\cO} \cO_M
	\]
	pour tout entier $n$. En particulier, $X_\infty(\textbf{f}^\dag)$ est de torsion sur $\cA[[T]]$.
\end{lemme}

\begin{proof}
	Cela résulte d'une application directe de la Proposition \ref{changement_base} pour le groupe de Selmer dual $X_\infty(\textbf{f}^\dag)$. Avec les notations de la Proposition, $\ga$ est l'idéal principal de $\cA$ engendré par $Y-y_n$ ou bien par $Y$. Montrons que ses hypothèses sont vérifiées. L'idéal $\ga$ est principal par définition. Le groupe d'inertie en $p$ agit trivialement sur $\bD^-$ parce que ceci est déjà le cas pour $\bT_\textbf{f}^-$. Il reste à montrer que $\bD^{G_{\bQ_\infty}}$ et $\bD^{I_\ell}$ (pour $\ell|N$) sont $\cA$-divisibles. Il suffit de démontrer qu'ils sont $\cA$-colibres. On note $D_{g_n}=V_{g_n}/T_{g_n}$ le $\cO_M$-module discret usuel construit à partir de la représentation de Deligne de $g_n$, et on note encore $D$ le $\cO_M$-module $D \otimes_{\cO} \cO_M$. Les propriétés de spécialisation énoncées dans le Théorème \ref{bigT} donnent les identifications $D_{g_n}\simeq \bD[Y-y_n]$ et $D \simeq \bD[Y]$. Soit $\varpi$ une uniformisante de $\cO_M$, soit $\bF_M=\cO_M/\varpi$ et soit $\gM=(\varpi,Y)$ l'idéal maximal de $\cA$. La représentation résiduelle $\ob{D}$ de $\rho_\textbf{f}$ vérifie $\ob{D} \simeq D_{g_n}[\varpi] \simeq D[\varpi]$ en tant que $\bF_M[G_\bQ]$-modules.
	
	On a déjà $\bD^{G_{\bQ_\infty}}=0$. En effet, comme les $G_{\bQ}$-modules $D$ et $\bD[Y]$ sont isomorphes, on a $\bD^{G_{\bQ_\infty}}[Y]= \bD[Y]^{G_{\bQ_\infty}}\simeq D^{G_{\bQ_\infty}}$. Comme $H\cap \bQ_\infty=\bQ$ et $\rho$ est irréductible et non-trivial, on a par ailleurs $D^{G_{\bQ_\infty}}=D^{G_{\bQ}}=0$. D'après le lemme de Nakayama, on a bien $\bD^{G_{\bQ_\infty}}=0$.
	
	Soit $\ell|N$, et montrons que $\bD^{I_\ell}$ est colibre sur $\cA$. Soit $\bP=\left(\bD^{I_\ell}\right)^\vee$ le dual de Pontriyagin de $\bD^{I_\ell}$. Soit $n\in\bN$. Par la dualité de Pontryagin, on a les trois identifications suivantes :
	$$\bP/Y\bP \simeq \left(D^{I_\ell} \right)^\vee,\qquad \bP/(Y-y_n)\bP \simeq \left(D_{g_n}^{I_\ell} \right)^\vee,\qquad \bP/\gM \bP\simeq \ob{D}^{I_\ell}.$$
	Comme $\rho\simeq \ob{\rho}$, le conducteur modéré de $\ob{\rho}$ est égal à $N$, et en particulier on a $\ord_\ell\left(\textrm{Cond}(\ob{\rho})\right)=\ord_\ell(N)=\ord_\ell\left(\textrm{Cond}(\rho_{g_n})\right)$. Par invariance du conducteur de Swan \cite[Proposition 1.1]{livne}, on a donc 
	$$\dim_{\bF_M} \ob{D}^{I_\ell} = \dim_M V^{I_\ell}_{g_n}.$$
	En outre, un théorème classique de Tate en cohomologie des groupes profinis \cite[Proposition 2.3]{tate1976relation} montre que $\dim_M V^{I_\ell}_{g_n}$ est égal au $\cO_M$-rang de $\left(D_{g_n}^{I_\ell} \right)^\vee$ (voir \cite[Proposition 3.10]{greenberg2006structure} pour une généralisation de ce résultat). 
	
	On peut maintenant montrer que $\bP$ est libre sur $\cA$. Si $\ob{D}^{I_\ell}=0$ alors on a automatiquement $\bP=0$ par le lemme de Nakayama. Si $\ob{D}^{I_\ell}$ est de dimension 1, alors $\bP$ est monogène, et donc $\bP\simeq \cA/(f(Y))$ pour un certain $f(Y) \in \cA$. De plus, le module $\bP/ \left(Y-y_n\right) \bP \simeq \cO_M/(f(y_n))$ est de rang 1, donc il est infini. Cela implique que l'élément $f(Y)$ s'annule en toutes les valeurs $Y=y_n$, $n\geq 0$, et donc $f(Y)=0$ d'après le théorème de préparation de Weierstrass. Ainsi, $\bP$ est libre comme voulu, et l'on a terminé la vérification.
	
	D'après le Corollaire \ref{prop:rappel_prop_sel_n_f_alpha}, $X_\infty(f_\alpha)$ est de $\cO[[T]]$-torsion, donc $X_\infty(\textbf{f}^\dag)/Y\cdot X_\infty(\textbf{f}^\dag)$ est de torsion sur $\cO_M[[T]]$. L'assertion élémentaire suivante montre pour finir que $X_\infty(\textbf{f}^\dag)$ est de torsion sur $\cA[[T]]$.
	
	\textit{Fait :} Soit $M$ un module de type fini sur un anneau $B$. Supposons qu'il existe un idéal premier $\gQ \subseteq B$ tel que $M/\gQ M$ est de torsion sur $B/\gQ$. Alors il existe une élément de $B - \gQ$ qui tue $M$, et en particulier $M$ est de torsion. 
\end{proof}

\begin{lemme}\label{Och}
	Le $\cA[[T]]$-module $X_\infty(\textbf{f}^\dag)$ n'a pas de sous-modules pseudo-nuls non-triviaux.
\end{lemme}

\begin{proof}
	On va montrer que l'on peut appliquer la Proposition \ref{PsN}. L'hypothèse (a) est vérifiée d'après le Lemme \ref{spécialisation_torsion}. Montrons que (b) est satisfaite, à savoir que le module $\cH:=H^2(\bQ_\Sigma/\bQ_\infty,\bD)^\vee$ est de torsion sur $\cA[[T]]$. On a montré dans la preuve de la Proposition \ref{prop:sel_infty_sans_sous_modules_finis} que $H^2(\bQ_\Sigma/\bQ_\infty,D)^\vee$ est de type fini et de torsion sur $\cO_M[[T]]$. La multiplication par $Y$ définit une suite exacte courte 
	$$\xymatrix{ 0 \ar[r] & D \ar[r] & \cD \ar[r] & \cD \ar[r] &  0, }$$
	induisant une application surjective $H^2(\bQ_\Sigma/\bQ_\infty,D) \twoheadrightarrow H^2(\bQ_\Sigma/\bQ_\infty,\bD)[Y]$. Donc le $\cO_M[[T]]$-module $\cH/Y\cdot\cH$ est un sous-module d'un module de type fini et de torsion. Il est donc de type fini et de torsion, et de même pour le $\cA[[T]]$-module $\cH$. L'hypothèse (c) est clairement vérifiée, car la représentation $\rho_\textbf{f}$ est impaire. L'hypothèse (d) aussi, car $I_p$ agit trivialement sur $\bD^-$. Enfin, l'hypothèse (e) est vérifiée car $\rho_\textbf{f}$ est résiduellement irréductible de dimension 2. Cela termine la vérification des hypothèses, et donc la preuve du lemme.
\end{proof}
\begin{remarque}
	Une variante du lemme précédent pour $X_\infty(\textbf{f})=X_\infty(\bT_\textbf{f},\bT_\textbf{f}^+)$ est montrée dans \cite[Proposition 8.1]{ochiai2006two} sous l'hypothèse que $\bH_\textbf{f}$ est régulier.
\end{remarque}
\begin{proof}[Preuve de la Proposition \ref{an}]
	Le groupe de Selmer $X_\infty(\textbf{f}^\dag)$ est de type fini et de torsion sur l'anneau $\cA[[T]]$ qui est factoriel, donc possède une fonction $L$ $p$-adique algébrique $L_p^{\alg,\dag}(Y,T) \in \cA[[T]]$. D'après les Lemmes \ref{spécialisation_torsion}, \ref{Och} et en appliquant la Proposition \ref{changement_base}, les idéaux caractéristiques de $\Sel_\infty(\phi_n)$ et de $\Sel_\infty(f_\alpha)$ sont engendrés respectivement par $L^{\alg,\dag}_p(y_n,T)$ et $L_p^{\alg,\dag}(0,T)$. Autrement dit, on a :
	$$ L_p^{\alg}(g_n,T) = L^{\alg,\dag}_p(y_n,T) \qquad \textrm{resp.} \qquad L_p^{\alg}(f_\alpha,T)=L^{\alg,\dag}_p(0,T).$$
	L'argument de la preuve du Lemme \ref{bn} montre alors que $\lim_{n}L_p^{\alg}(g_n,T) = L_p^{\alg}(f_\alpha,T)$ dans $\cO_M[[T]]$, ce qui termine la preuve de la Proposition \ref{an}.
\end{proof}

	\bibliography{bib}
	\bibliographystyle{alpha-fr}
\end{document}